\documentclass[11pt]{amsart}
\usepackage{amsfonts, amsthm, amssymb, amsmath, stmaryrd}
\usepackage{mathrsfs,array}
\usepackage{eucal,color,times,enumerate,accents}
\usepackage[all]{xy}
\usepackage{url}
\usepackage{enumitem}
\usepackage[pdftex,bookmarksnumbered,bookmarksopen]{hyperref}
\hypersetup{
  colorlinks,
  citecolor=black,
  linkcolor=black,
  urlcolor=black}
  
  \newcommand{\MT}{\mathbf{MT}}
  \DeclareMathOperator{\Sl}{SL}

\usepackage{tikz-cd}
\usetikzlibrary{cd}
\usepackage{leftidx}
\usetikzlibrary{positioning}
\usepackage{dsfont}
\usepackage{tikz}
\usetikzlibrary{matrix,arrows,decorations.pathmorphing,positioning}

\newcommand{\HL}{\textnormal{HL}}
\newcommand{\an}{\textnormal{an}}

\usepackage[utf8]{inputenc}
\usepackage[T1]{fontenc}

    \newcommand\blfootnote[1]{%
  \begingroup
  \renewcommand\thefootnote{}\footnote{#1}%
  \addtocounter{footnote}{-1}%
  \endgroup
}
  
  \newcommand{\Addresses}{{
  \bigskip
  \footnotesize

\textsc{CNRS, IMJ-PRG, Sorbonne Universit\'{e}, 4 place Jussieu, 75005 Paris, France}\par\nopagebreak
  \textit{E-mail address}, G.~Baldi: \texttt{baldi@imj-prg.fr}

}}

\usepackage[nameinlink]{cleveref}

\usepackage{bbm}
\newcommand{\VV}{\mathbb{V}}
\newcommand{\V}{\mathbb{V}}

\newcommand{\Gad}{G^{\text{ad}}}

\newcommand{\Xad}{X^{\text{ad}}}

\newcommand{\ch}[1]{\check{{#1}}}

\usepackage{leftidx}

\newcommand{\piet}{\pi_1^{\acute{e}t}}

\usepackage[bottom=1.5cm, top=1.5cm, left=2cm, right=2cm]{geometry}

\input xy
\xyoption{all}

\newcommand{\Ag}{\mathcal{A}_g}
\newcommand{\Mg}{\mathcal{M}_g}
\newcommand{\Og}{\Omega \mathcal{M}_g}

\newcommand{\DT}{\mathbb{S}}

\usepackage{tikz-cd}
\theoremstyle{plain}
\newtheorem{thm}{Theorem}[section]

\newtheorem{conj}{Conjecture}[section]

\newtheorem{lemma}[thm]{Lemma}

\theoremstyle{plain}

\theoremstyle{plain}

\newcommand{\teic}{Teichm\"{u}ller }
\newcommand{\kob}{Kobayashi }

\newtheorem{prop}[thm]{Proposition}
\newtheorem{cor}[thm]{Corollary}
\theoremstyle{definition}
\newtheorem{defi}[thm]{Definition}

\newtheorem{rmk}[thm]{Remark}
\newtheorem{question}[thm]{Question}
\theoremstyle{remark}

\newtheorem*{exs}{Example}

\numberwithin{equation}{section}

\newcommand{\typ}{\textnormal{typ}}

\DeclareMathOperator{\Jac}{Jac}

\DeclareMathOperator{\PU}{PU}
\DeclareMathOperator{\Sh}{Sh}

\DeclareMathOperator{\Sp}{Sp}

\DeclareMathOperator{\End}{End}

\DeclareMathOperator{\Hom}{Hom}
\DeclareMathOperator{\Gal}{Gal}

\DeclareMathOperator{\Res}{Res}

\DeclareMathOperator{\Mod}{Mod}

\DeclareMathOperator{\NL}{NL}

\DeclareMathOperator{\codim}{codim}
\DeclareMathOperator{\Lie}{Lie}

\DeclareMathOperator{\Gl}{GL}

\newcommand{\Gm}{\mathbb{G}_m}
\newcommand{\Ga}{\mathbb{G}_a}

\newcommand{\Aut}{\operatorname{Aut}}

\newcommand{\N}{\mathbb{N}}

\newcommand{\Z}{\mathbb{Z}}
\newcommand{\Q}{\mathbb{Q}}

\newcommand{\R}{\mathbb{R}}
\newcommand{\Pp}{\mathbb{P}}
\newcommand{\A}{\mathbb{A}}

\newcommand{\Oo}{\mathcal{O}}

\newcommand{\F}{\mathcal{F}}

\newcommand{\Hh}{\mathbb{H}}

\newcommand{\C}{\mathbb{C}}

\newcommand{\Qbar}{\overline{\mathbb{Q}}}

\makeatletter
\def\@settitle{\begin{center}%
  \baselineskip14\p@\relax
  \bfseries
  \uppercasenonmath\@title
  \@title
  \ifx\@subtitle\@empty\else
     \\[1ex]\uppercasenonmath\@subtitle
     \footnotesize\mdseries\@subtitle
  \fi
  \end{center}%
}
\def\subtitle#1{\gdef\@subtitle{#1}}
\def\@subtitle{}
\makeatother

\usepackage{microtype}
\usepackage{fullpage}

\begin{document}

\newcommand{\adjunction}[4]{\xymatrix@1{#1{\ } \ar@<-0.3ex>[r]_{ {\scriptstyle #2}} & {\ } #3 \ar@<-0.3ex>[l]_{ {\scriptstyle #4}}}}

\title{What makes an algebraic curve special?}

\date{\today}

\author{Gregorio Baldi}

\begin{abstract}
A survey of special curves, special subvarieties of $\mathcal{M}_g$, and related topics. A large portion of the text discusses various possible interpretation of the word ``special'' in this context by giving also concrete examples. One highlight is the bi-algebraic viewpoint for atypical intersections appearing in Hodge theory as well as, more recently, in Teichm\"{u}ller theory.

\end{abstract}
\maketitle

\tableofcontents

\newpage

\section{Introduction}
Curves of genus $g$ and their moduli space, $\mathcal{M}_g$, are central objects in both classical and modern mathematics. See, for example, Mumford's famous outline \cite{zbMATH03494560}. Many mathematicians have a favorite curve, and there are many possible explanations for why they like such a curve. Often, it is either because the curve is unexpectedly linked to another area of mathematics, or because it has some surprising properties. Of course, there might be a non-trivial connection between the two explanations. Rarely does anyone claim to like the generic curve, or even the generic hyperelliptic curve.

The main topic of these notes is the question posed in the title:

\begin{question}
What makes a curve special? 
\end{question}

From the above informal discussion, we should at least keep in mind the following slogan:
\begin{center}
special curves should be \textbf{rare}.
\end{center}
Here, "rare" should be understood as meaning "there are only finitely many families of curves with such a property," or even better, "there are only finitely many such curves."

There are at least two angles (metric and algebro-geometric) that are related but arise from different communities that we will survey and we will highlight some of their common features. Along the way, we will popularize some conjectures, present various results, and showcase examples of families of curves, seeking similarities. In essence, the \emph{Zilber-Pink viewpoint}, in one form or another, will pervade our discussions. The text is organized in three parts, and we now survey the main players and cornerstone results.

For an introduction, we also refer to the volumes on the geometry of algebraic curves \cite{zbMATH03891516, zbMATH05798333}. For a more topological stand point, we recommend the book by Farb and Margalit \cite{zbMATH05960418}.

\subsection{Hodge theory}
From the Hodge-theoretic perspective, the $H^1$ of a curve naturally determines a principally polarized abelian variety, called the \emph{Jacobian}. Indeed, Riemann's theorem defines an equivalence between polarized abelian varieties over $\C$ and the category of polarizable integral Hodge structures of type $\{(1,0), (0,1)\}$. Algebraic cycles on self-products of a curve $X$ determine Hodge cycles on $\Jac(X)$, and it is natural to investigate Jacobians with extra algebraic/Hodge cycles. One notable example is provided by Jacobians whose endomorphism algebra is larger than expected. Among these, Jacobians with \emph{complex multiplication} occupy a special place. The theory of Complex Multiplication lies at the heart of Number Theory, and understanding such Jacobians is a fascinating challenge. We will explore this aspect, which will lead us to the study of sub-Shimura varieties of $\mathcal{A}_g$ (and their Hecke translates) that are generically contained in $\mathcal{M}_g$.

\subsection{\teic theory}
A simple yet powerful observation lies at the heart of this way of thinking, borrowed from the introduction of \cite{zbMATH06353727}. When $X$ is equipped with a (non-zero) holomorphic one-form $\omega \in \Omega(X)$, $X$ acquires a \textbf{geometric} character (contrasted with the algebro-geometric flavor described above). For example, such a form determines a Euclidean metric on $X$ (with singularities at the zeros of $\omega$) and a foliation $\mathcal{F}(\omega)$ by horizontal geodesics. Similarly, the Hodge bundle $\Omega \mathcal{M}_g$ exhibits features absent in the underlying moduli space of curves $\Mg$: a stratified linear structure and a natural action of $\operatorname{GL}_2(\R)^+$.

Indeed, while $\Omega \mathcal{M}_g$ is a rank $g$ vector bundle (possibly deprived of the zero section), it is stratified by the algebraic subsets $\Omega \mathcal{M}_g(\kappa)$ of holomorphic 1-forms with zeros of multiplicities given by $\kappa = (\kappa_0, \dots, \kappa_n)$, where $\sum_i \kappa_i = 2g - 2$ (thanks to the Riemann-Hurwitz formula). We refer to \cite{zbMATH02001031} for a discussion on connected components of the strata and to \cite{2022arXiv220411943C} for results concerning the Kodaira dimension of (projectivized) strata of Abelian differentials.

A holomorphic 1-form $\omega$ on $X$ provides a collection of "charts" on $X$ mapping to $\mathbb{C}$, where the transition maps are translations. These charts ramify at finitely many points corresponding to the zeros of $\omega$ and are locally described by $z \mapsto \int_{z_0}^z \omega$. These charts induce the so-called \textbf{period coordinates} on $\Omega \mathcal{M}_g(\kappa)$, which are modeled by the relative cohomology group $H^1(X, Z(\omega); \C)$. 

We will later discuss the flat picture and the $\operatorname{GL}_2(\R)^+$-action in that "global" setting. For the introduction, we take the period coordinates perspective: decompose $v \in H^1(X, Z(\omega); \C)$ into real and imaginary parts, $v = \operatorname{Re} v + i \operatorname{Im} v$ and set
\[
\begin{bmatrix}
a & b \\
c & d
\end{bmatrix}
\cdot
\begin{bmatrix}
\operatorname{Re} v \\
\operatorname{Im} v
\end{bmatrix}
=
\begin{bmatrix}
a \operatorname{Re} v + b \operatorname{Im} v \\
c \operatorname{Re} v + d \operatorname{Im} v
\end{bmatrix}.
\]

This is a vast topic, but here we will primarily focus on the algebro-geometric aspects of the theory. For an introduction to \teic dynamics from this angle, we refer to the notes \cite{zbMATH06863806}. We also note that, in the dynamical approach, Hodge theory appears in formulas and results about the Lyapunov exponents of a linear cocycle (the so-called \emph{Kontsevich–Zorich cocycle}), which encodes the topological/homological behavior of trajectories of translation flows and the tangent cocycle of the \teic flow \cite{kzhodge}. For more on this perspective, see \cite{zbMATH06436256}. Another interesting discussion about the relationship between the two viewpoints described so far can be found in the problem section at the end of the introduction of \cite{zbMATH06272629}.

\subsection{The bi-algebraic viewpoint}
Bi-algebraic geometry, as outlined in \cite[Sec. 3 and 4]{MR3821177}, has recently emerged due to its deep connections with the so-called \emph{André-Oort conjecture}. The basic idea can be summarized as follows. Let $V$ be a smooth quasi-projective complex variety, with universal covering $\widetilde{V}$. In many interesting cases, $\widetilde{V}$ is \textbf{not} algebraic. (For example, if $V$ is projective, are there examples of such $V$ with infinite fundamental group but not dominated by abelian varieties?) Nevertheless, $\widetilde{V}$ may admit holomorphic maps to algebraic varieties $W$. However, in the one dimensional case, $\widetilde{V}$ turns out to be algebraic if and only if it is bi-holomorphic to the projective line or the complex numbers. By fixing one such map $d: \tilde{V} \to W$, bi-algebraic geometry suggests that \emph{bi-algebraic} subvarieties (i.e., the Zariski-closed subvarieties of $V$ that can also be locally described algebraically in $W$) should play a special role. We will explain how this angle relates the previous two viewpoints.

We present a few examples:
\begin{itemize}
\item Let $V$ be a product of $n$ Riemann surfaces of higher genus, then $\widetilde{V} \cong \mathbb{H} \times \cdots \times \mathbb{H}$ (cf. \Cref{unifthm}), $W$ is the product of $n$ copies of $\mathbb{P}^1 $, and the map is induced by the standard inclusion $\mathbb{H}\to \mathbb{P}^1$;
\item If $V$ supports a family of abelian varieties, $\widetilde{V}$ comes with a holomorphic map to a Hermitian symmetric domain $D$. Even if $D$ is not algebraic itself (it is only semi-algebraic), it comes with various maps to algebraic varieties, as we will review in \Cref{sec71};
\item Another related example is given by $\Og$ where one has \emph{local} period coordinates. We will see later examples of the bi-algebraic viewpoint in this setting, even though the period coordinates do not cleanly relate to the universal cover.
\end{itemize}

\subsection{What is in this text?}
\begin{itemize}
\item Various explicit examples of families of curves whose Jacobians have endomorphism ring unusually large;
\item We survey the known results regarding sub-Shimura varieties generically contained in $\Mg$ as well as other results on the \emph{Hodge locus} of $\Mg$; 
\item An introduction to \teic geometry and \teic curves;
\item The Eskin-Filip-Wright finiteness theorem for atypical orbit closures;
\item An exposition of functional transcendence that tries to describe uniformly the Shimura and \teic sides;
\item Further considerations inspired by the Zilber-Pink philosophy;
\item Several classical and new questions on the topic.
\end{itemize}

\subsection{Notation}
Two types of curves will appear: $[X]\in \mathcal{M}_g$ and $C\subset \mathcal{M}_g$, and they should not be confused. The first is a smooth projective curve of genus $g$ (sometimes described by affine equations rather than their projective models), while the second may be neither smooth nor projective. Often, though not always, we will implicitly replace $\mathcal{M}_g$ by a finite covering. Our focus is mainly on objects defined over $\overline{\Q}$ and $\mathbb{C}$ (with a fixed embedding of the former into the latter).

Every curve here will be algebraic; nevertheless, it is worth noting that transcendental curves played a fundamental role in the development of mathematics and physics in the 17th century (e.g., spirals, catenaries, the brachistochrone, and other tautochrones). However, some proofs presented here may actually involve non-algebraic curves (specifically, \emph{definable curves}). This will be discussed only in \Cref{sectaoproof} and relates to the beautiful Pila–Wilkie theorem \cite{pilawilkie, zbMATH04183528}.

\subsection{Disclaimer}
We selected a few topics that can be unified by the role of periods and that are not commonly found together in other surveys. These notes may be viewed as a complement/update to the beautiful \cite{moonenoort}, with the addition\footnote{In \emph{op. cit.}, this is mentioned only once, following \cite[Qtn. 6.11]{moonenoort}, due to the work of M\"{o}ller \cite{zbMATH05919111}.} of a discussion on the Teichm\"{u}ller perspective on $\mathcal{M}_g$. The motivation for this came from the recent works \cite{2022arXiv220206031K, filipnotes, 2024arXiv240616628B}, which presented a perspective that integrates aspects of both theories.

Unsurprisingly, the literature on this topic is extensive, and we did not attempt the impossible task of surveying every result (although we ended up collecting a large number of references). We apologize to anyone we may have inadvertently omitted. These notes are an extended version of a talk given by the author during the Simons Symposium on Periods and $L$-values of Motives (2024), held at Schloss Elmau, based on \cite{2024arXiv240616628B} and, partially, on \cite{2021arXiv210708838B}. We hope that the emphasis we put on period coordinates viewpoint on \teic theory fits well a conference on Periods!

While preparing this text, we aimed to refine certain results throughout. The only new contributions can be found in \Cref{sectworktol}, \Cref{predsection}, \Cref{ftmerdiff}, and several questions that we have interspersed in the notes.

\subsection{Acknowledgements}
First of all we wish to thank J.-B. Bost and S. Zhang for organizing the Simons Symposium. Recently we gave also a talk based on \Cref{part1} in Berlin, during the conference \emph{Hodge theory, periods and special loci} (in memory of Tobias Kreutz) and we thank all the participants for remarks and questions.

We are grateful to Y. André, B. Farb, S. Filip, B. Klingler, J. Lam, C. Matheus, C. McMullen, E. Ullmo, D. Urbanik, and S. Zhang for various insightful discussions. We gained much knowledge on the \teic perspective during the study group on \emph{Orbit closures and their properties}, organized with C. Matheus and it is a pleasure to thank him as well as all the participants. Thanks to D. Lombardo, N. Khelifa, and C. McMullen for their comments on a previous draft.

Finally, in 2016 at Monte Verità, I attended a lecture series by Frans Oort on CM Jacobians. This took place the summer before I began my Ph.D. studies, and I have always kept in mind the setting and questions I first heard from Oort, along with some discussions we shared over lunches. It is therefore a pleasure to write these related notes and to thank him as well.

During the final stages of this work, the author was partially supported by the grant ANR-HoLoDiRibey of the Agence Nationale de la Recherche.

\newpage
\part{Hodge theoretic viewpoint}\label{part1}

\section{Curves and their Jacobians}\label{sec2}
Let $X$ be an irreducible smooth projective algebraic curve over $\C$, or simply a (compact) \emph{Riemann surface}. The first invariant one associates to a Riemann surface $X$ is its genus $g$. One of the most important theorems is the following:

\begin{thm}[Koebe, Poincaré 1907]\label{unifthm}
The only connected, simply connected Riemann surfaces are $\Pp^1_\C$, $\C$, and\  $\mathbb{B}^1_{\C}$.
\end{thm}
(Here we denote by $\mathbb{B}^1_{\C}$ the unit disc in $\C$. Notice that it is biholomorphic to the upper half plane $\mathbb{H}$. In the sequel we will use both descriptions.)

As a corollary the universal covering $\widetilde{X}$ of any connected Riemann surface is one of the three, and depends only on the genus: $g=0$ iff $\widetilde{X}\cong\Pp^1_\C$, $g=1$ when $\widetilde{X}\cong \C$ and higher genus in the remaining cases. For a complete and recent discussion on the above theorem, we refer to \cite{zbMATH05835819}.

We have already hinted to at least three ways of presenting $X$, and each of them hints already to a possible definition of \emph{special} (taking the nomenclature from \cite{zbMATH02204321}):
\begin{itemize}
\item (\emph{Fuchs viewpoint}) If $X\cong \Gamma \backslash \mathbb{H}$, one could say that $X$ is special the more special $\Gamma$ is, among the other lattices in $\operatorname{PSL}_2(\R)$;
\item (\emph{Jacobi viewpoint}) Considers the periods associated to $(X,\omega)$, where $\omega$ is a nonzero meromorphic differential form on X, $\int_\gamma \omega \in \C$;
\item (\emph{Riemann viewpoint}) By looking at the algebraic equations defining the curve $X$ (let's say, by imagining it embedded in some projective space).
\end{itemize}
We will start with the second perspective. To compare one curve to another, it is most useful to have a moduli space of such objects. We denote by $\Mg$ the moduli stack of curves of genus $g$ (over $\C$, but we notice here that it can be defined over $\operatorname{Spec}(\Z)$). It is a quasi-projective smooth irreducible Deligne-Mumford stack of dimension $\max{ \{g,3g-3\}}$.  

 One of the most famous period maps is the \textbf{Torelli map} \cite{zbMATH02622397} (see also, e.g. \cite{zbMATH03975085}):
\[
j : \mathcal{M}_g\to \mathcal{A}_g=[\operatorname{Sp}_{2g}(\Z)\backslash \mathbb{H}_g], \quad C \mapsto j(C)= (J(C),\theta),
\]
where $J(C)$ is the Jacobian of the curve $C$ with its principal polarization $\theta$. Here $\mathbb{H}_g$ is the Siegel upper half-space of genus $g$: the set of $g\times g$ symmetric matrices over the complex numbers whose imaginary part is positive definite. It is the symmetric space associated to the symplectic group $\operatorname{Sp}_{2g}(\R)$. $\mathcal{A}_g$ is the moduli space of principally polarized abelian varieties of dimension $g$ and has dimension $g(g+1)/2$.

 Denote by $\mathcal{T}_g$ the Zariski closure of the image of $j$ in $ \mathcal{A}_g$. From several perspectives, researchers have examined subvarieties of $\mathcal{M}_g$ that exhibit special properties. A few key examples include: Brill-Noether loci \cite{2018arXiv180905980L}, Teichmüller geodesics (arising from billiard dynamics \cite{zbMATH07666127}, initiated by Veech), and the Kodaira-Parshin construction.

\begin{rmk}
In the boundary $\mathcal{T}_g - j(\Mg)$ one finds the \emph{decomposed Jacobians}: A point of $\mathcal{T}_g$ is in the boundary iff the corresponding principally polarized abelian variety is isomorphic to the product of two principally polarized abelian varieties (of dimension >0, and as ppav). Cf. the end of Section 1.3 in \cite{moonenoort}.
\end{rmk}

\subsection{Why Jacobians}
With the goal in mind of defining special curves by associating some other variety/Hodge structure to a curve one might argue that there are many possibilities, apart from the Jacobians. Of course the Jacobians look very natural (see for example the Rmk.  below), and Farb \cite{zbMATH07807769} proved that the Torelli map is essentially unique in $\Hom (\mathcal{M}_g,\Ag)$ (improving upon a representation theoretic result of Korkmaz). This is one of the few theorems that uses $\Mg$ rather than a covering thereof, see indeed the various Prym constructions, like in \cite{zbMATH00015078}. Prym varieties form a special class of principally polarized abelian varieties, more general than Jacobians, and were discovered by Wirtinger in 1895. Unfortunately we'll not discuss this fascinating topic.

There are however other constructions of local systems/variations of Hodge structures on $\Mg$. For example, in the early 90’s, quantum representations of the mapping class group of a surface were discovered  as a byproduct of a more general structure called topological quantum field theory (TQFT) \cite{zbMATH07438861}.

\begin{rmk}
It is an old theorem of Matsusaka \cite{zbMATH03073579} that every abelian variety over an algebraic closed field is a quotient of a Jacobian. 
\end{rmk}

\subsection{Endomorphisms}
Abelian varieties with the largest endomorphism ring are called \emph{with Complex Multiplication}, or simply CM. We will come back to this topic several times, see \Cref{defcm} for an equivalent definition. Given an abelian variety $A$, we write $\End^0(A)=\End(A)\otimes \Q$.
\begin{defi}\label{cm22}
We say that an abelian variety $A$ has \textbf{complex multiplication} if $\End^0(A)$ contains a commutative, semisimple $\Q$-subalgebra of dimension $2\dim A$.
\end{defi}

\begin{defi}
Let $F$ be a totally real number field. We say that an abelian variety $A$ has \textbf{real multiplication} by $\mathcal{O}_F$, if there is an action of the ring of integers $\mathcal{O}_F$ on $A$ such that the Lie algebra of $A$ is free of rank one over $\mathcal{O}_F$ (which acts by functoriality)
\end{defi}
In particular, this requirement forces the dimension of $ A$ and the degree of $F$ to be equal. We will often be concerned with the case where $ \End^0(A)$ is equal to (and not bigger than) a totally real field of dimension $\dim(A$). On the other hand, we do not require always the full ring of integers of this field to act on $A$. We call this more general case \emph{real multiplication by $F$}. 

\begin{rmk}
Both CM abelian varieties and abelian varieties with real multiplication (RM) are dense in $\Ag$. We will see that this can drastically change when we restrict to Jacobians. 
\end{rmk}

We record here the following folklore and open ended question.

\begin{question}
 Every endomorphism of 
$\Jac(X)$, can be represented by a divisor on $X \times X$. What can be said on such divisor? Do they have a conceptual explanation, e.g. in the case of Jacobians with real or complex multiplication?
\end{question}

It's worth mentioning at least one example of a family with a rather special property--we will see more examples in a moment.
\begin{thm}[Earle {\cite{zbMATH03657936}}]
For each $n=2k+1$ odd, the families of hyperelliptic curves (for $t\neq 0,1$)
\begin{equation}
C_{t,n}: \{y^2=(x^n-1)(x^n-t)\}
\end{equation}
have the following property: there are two curves $Y_{1,t,k}, Y_{2,t,k} $ of genus $k$ such that $\Jac(C_{n,t})$ is isomorphic to the product of the two jacobians.

\end{thm}
\begin{rmk}
We will be mainly concerned with abelian varieties over the complex numbers. However, Robert Coleman conjectured that for a given number field $k$ only finitely many rings, up to isomorphism, can be realised as the ring of $\End^0$ of an abelian variety defined over $k $. See Conjecture $C(e,g)$ from \cite{zbMATH05123079}.
\end{rmk}
\subsection{Albert classification}
The structure of $\End^0(A)$, for $A$ simple, is described by the Albert classification of division algebras with involution. To a principally polarized abelian variety $(A,\lambda)$ we associate the pair $(D=\End^0 (A),\dagger)$, $\dagger$ the Rosati involution. $D$ is a simple $\Q$-algebra of finite dimension and $\dagger$ is a positive involution.
Let $K$ be the center of $D$ (so that $D$ is a central simple $K$-algebra), and $K_0$ be the subfield of symmetric elements in $K$. Write $e=[K:\Q]$. We know that either $K_0 = K$, in which case $\dagger$ is said to be of the first kind, or that $K_0\subset K$ is a quadratic extension (second kind).

(If $Q$ is a quaternion algebra over a field $L$, its \emph{canonical involution} is the involution given by $q \mapsto Trd_{Q/L}(q)-q$.)
\begin{thm}[Albert] The pair $(D,\dagger)$ is of one of four types: 
\begin{itemize}
\item[Type I.] $K_0 =K=D$ is a totally real field and $\dagger=$identity
\item[Type II.]$K_0 = K$ is a totally real field, and $D$ is a quaternion algebra over $K$ with $D \otimes_{K,\sigma} \mathbb{R} \cong M_2(\mathbb{R})$ for every $\sigma\colon K \to \mathbb{R}$.  

Let $d \mapsto d^*$ be the canonical involution on $D$. Then there exists an element $a \in D$ such that $a^2 \in K$ is totally negative, and such that $d^\dagger = a d^* a^{-1}$ for all $d \in D$.  

We have an isomorphism $D \otimes_{\mathbb{Q}} \mathbb{R} \cong \prod_{\sigma\colon K \to \mathbb{R}} M_2(\mathbb{R})$ such that the involution $\dagger$ on $D \otimes_{\mathbb{Q}} \mathbb{R}$ corresponds to the involution $(d_1, \dots, d_e) \mapsto (d_1^t, \dots, d_e^t)$.  
\item[Type III.] $K_0 = K$ is a totally real field, and $D$ is a quaternion algebra over $K$ with $D \otimes_{K,\sigma} \mathbb{R} \cong \mathcal{H}$ for every embedding $\sigma\colon K \to \mathbb{R}$. (We write $\mathcal{H}$ for the Hamiltonian quaternion algebra over $\R$). $\dagger$ is the canonical involution on $D$.  

We have an isomorphism $D \otimes_{\mathbb{Q}} \mathbb{R} \cong \prod_{\sigma\colon K \to \mathbb{R}} \mathcal{H}$ such that the involution $\dagger$ on $D \otimes_{\mathbb{Q}} \mathbb{R}$ corresponds to the involution $(d_1, \dots, d_e) \mapsto (\bar{d}_1, \dots, \bar{d}_e)$.  
\item[Type IV.] $K_0$ is a totally real field, $K$ is a totally imaginary quadratic field extension of $K$. Write $a \mapsto \bar{a}$ for the complex conjugation of $K$ over $K_0$. If $v$ is a finite place of $K$, write $\bar{v}$ for its complex conjugate. The algebra $D$ is a central simple algebra over $K$ such that:  
\begin{itemize}
\item If $v$ is a finite place of $K$ with $v = \bar{v}$ then $\operatorname{inv}_v(D) = 0$;  
\item For any place $v$ of $K$ we have $\operatorname{inv}_v(D) + \operatorname{inv}_{\bar{v}}(D) = 0$ in $\mathbb{Q}/\mathbb{Z}$.  
\end{itemize}
If $m$ is the degree of $D$ as a central simple $K$-algebra, we have an isomorphism  
$$D \otimes_{\mathbb{Q}} \mathbb{R} \cong \prod_{\sigma\colon K_0 \to \mathbb{R}} M_m(\mathbb{C})$$
such that the involution $\dagger$ on $D \otimes_{\mathbb{Q}} \mathbb{R}$ corresponds to the involution  
$(d_1, \dots, d_e) \mapsto (\bar{d}_1^t, \dots, \bar{d}_e^t)$.  
\end{itemize}

\end{thm}

\subsection{Automorphisms}\label{autsection}
We pause for a moment to discuss what can actually be said regarding automorphisms of a curve, rather than endomorphism of its Jacobian.

\begin{thm}
Let $X$ be a curve of genus $\geq 2$, then $\Aut(X)$ is finite.
\end{thm}

\begin{proof}[Sketch of the proof]
Since $X$ is of higher genus, we saw before that it is uniformized by $\mathbb{H}$. Since the automorphism group of $\mathbb{H}$ preserves the metric of constant curvature $-1$ on $\mathbb{H}$, we obtain a hyperbolic metric on $X$. Schwarz Lemma implies that $\Aut(X) $ acts by isometries on $X$. In particular it is a compact group. To obtain finiteness, it is enough to show that it is discrete. This follows from the fact that the natural map $\Aut(X)\to \Mod(X)$ is injective, cf. \cite[Thm. 2.2]{zbMATH01414943} (here $\Mod(X)$ denotes the \emph{mapping class group} of $X$, cf. \cite{zbMATH05960418}).
\end{proof}

\begin{thm}[Hurwitz's automorphisms theorem, 1893]
If $\geq 2$, he group $\Aut(X)$ has cardinality at most $84(g-1)$. Moreover the equality is reached if and only if $X$ is a branched covering of $\mathbb{P}^1$ with three ramification points, of indices 2,3 and 7.

\end{thm}

\begin{exs}
The famous Klein quartic discovered in 1879, $\mathcal{K}= \{ x^{3}y+y^{3}z+z^{3}x=0\}\in \mathcal{M}_3$ has automorphism group $\operatorname{PSL}_2(\mathbb{F}_7)$ with cardinality $168=84(3-1)$.
\end{exs}
Modular curves are simply quotients of the complex upper half-plane $\mathbb{H}$ by the action of a congruence subgroup  $\Gamma$ of $\Sl_2(\Z)$. The resulting curve is however not compact. See \Cref{examples} for a more detailed discussion on this topic. For the time being we just point out that they naturally parametrize elliptic curves with extra structures, depending on the congruence subgroup $\Gamma$. We will say more on this in the appendix. There are also examples of compact Shimura curves that would naturally fit into the narrative. 
\begin{rmk}
The Klein quartic described before is a modular curve of level $7$.
\end{rmk}

The \emph{Fermat curve} is another example of highly studied curve:
\begin{equation}\label{fermat}
F_d:=\{ x^d+y^d+z^d=0\}.
\end{equation}
It is smooth, of genus $(d-1)(d-2)/2$ (and gonality $d-1$).

Of course many more example of famous algebraic curves are around-- one could look at Pierre de Fermat and his 1659 treatise on quadrature, for example.

\begin{rmk}
Extra automorphisms are rare, and several question can be asked on the structure of $\Aut(X)$. Perhaps the most natural one is how often the Hurwitz bound is achieved.
 Macbeath showed that this bound is attained for infinitely many values of $g$ and Accola proved that it fails to be attained infinitely often. A further interesting discussion is offered by Larsen in \cite{zbMATH01716778}. See also \cite{zbMATH06708998}.
\end{rmk}

\subsubsection{Cyclic automorphisms}
Irokawa and Sasaki \cite{zbMATH00828341} gave a complete classification of curves over $\C$ with an automorphism of order $N \geq 2g +1$: such curves are either hyperelliptic with $N = 2g +2$ and $g$ even, or are quotients of the Fermat curve of degree $N$ by a cyclic group of order $N$.

\subsubsection{Related works}
Let $X$ be a curve of genus $\geq 2$. Following  Rauch and Wolfart, we say that $X$ \emph{has many automorphisms} if it cannot be deformed non-trivially together with its automorphism group. M\"{u}ller and Pink \cite{zbMATH07536477} determine all complex hyperelliptic curves with many automorphisms and decide which of their jacobians have complex multiplication. For a study of superelliptic curves with many automorphisms and CM Jacobians, see also \cite{zbMATH07390224}. 

We refer to \cite[Prob. 2.19]{zbMATH05124673}, for an interesting related problem.

\subsection{Completely decomposable Jacobians, à la Ekedahl-Serre}
Ekedahl and Serre \cite{zbMATH00437625} produce many
examples of curves of suitable genera $g$ up to 1297, whose Jacobian is isogenous to a product of elliptic curves (aka \emph{completely decomposable Jacobians}). These curves are constructed either as modular curves or as coverings of curves of genus 2 or 3; and in the end have genus up to 1297.

\begin{exs}[{\cite[Thm. 2]{zbMATH06488053}}]
The genus 5 hyperelliptic curve 
\begin{displaymath}
y^2 =x(x^{10} +11x^5 -1)
\end{displaymath}
has Jacobian isogenous to the fifth power of the elliptic curve $y^2 =x(x^2+11x-1)$.
\end{exs}

 In the paper  \cite{zbMATH00437625}, they also ask:

\begin{question}[Ekedahl-Serre]For every $g$, is there a curve of genus $g$ whose Jacobian is completely decomposable?
Or, are the genera of curves with completely decomposable Jacobians bounded?
\end{question}

 We refer to the preprint \cite{arXiv:1703.10377} for some partial results in that direction, as well as to the work of Paulhus and Paulhus-Rojas \cite{zbMATH06815607} for the new values of $g$ for which there exist totally decomposable Jacobians of genus $g$. Another interesting results of theirs is:

\begin{thm}[{\cite[Thm. 3.3]{zbMATH06815607}}]\label{paul}
Let $g\in \{11-19, 21-29, 31, 33-35, 37, 40, 41, 43-47, 49, 52,$  $53, 55, 57, 61, 65, 57, 69, 73, 82, 91, 93, 95, 97, 109, 129, 145, 193\}$. Then there is a dimension one (or larger) family of completely decomposable Jacobians of curves of genus $g$. 
\end{thm}

Contrary to the previous cases, the approach in \Cref{paul} is based on the use of Riemann surfaces with non-trivial automorphism group (and computational tools). For further results on this topic, we refer also to the monograph \cite{zbMATH07566633}.

\begin{rmk}
There are also various approaches to study questions of the following form. Consider the one-parameter family of hyperelliptic curve 
\begin{displaymath}
C_t=\{ y^2=f(x)(x-t)  \}
\end{displaymath}
for some polynomial $f\in \Z[X]$ of degree $2g$, $g\geq 1$. Let $K$ be a number field. What can be said on the locus of $t\in K$ such that $\Jac (C_t)$ is non-simple? Even if one knows that only finitely many such $t$ can appear, it is hard to understand how this finite number varies as $f $ changes. Various approaches to such question are discussed in \cite{zbMATH05588918}. Related are also \cite{zbMATH07794951} and \cite{zbMATH07243792}.
\end{rmk}
We can now believe that Completely decomposable Jacobians and we now try to see whether CM (and RM) are also rare phenomena for Jacobians.

\subsection{Real and Complex multiplication}\label{realcommult}

\begin{conj}[Coleman-Oort, arithmetic version {\cite[Conj. 6]{zbMATH04065160}}]\label{cofirstapp}
For sufficiently large genus $g$ (specifically, $g \geq 8$), $\mathcal{M}_g$ contains only finitely many curves with a CM Jacobian.
\end{conj}
\begin{rmk}
An example of a CM Jacobian is given by the Fermat \eqref{fermat} equation --see for example \cite{arXiv:2303.18056} (and references therein) for more on this, as well as \cite[Sec. 5]{2024arXiv240520394G}. Another family of examples will be discussed in \Cref{rmklombetal}. See also \cite[Table 1]{zbMATH07536477} and \cite{zbMATH05987975} for further interesting examples.
\end{rmk}

We will discuss more about this in the later sections, but for now, we just record the following nice result:
\begin{thm}[\cite{zbMATH06775945, zbMATH07740282}]\label{cmhype}
For $g\geq 8$, only finitely many genus $g$ hyperelliptic curves can have CM Jacobians
\end{thm}

See also \cite[Sec. 6]{moonenoort} for more conjectures and questions (due to Oort and Moonen-Oort). To mention one striking example (see also \cite[Prob. 4]{zbMATH01587304}):
\begin{question}\label{moonenoortquestion62}
Do we know the existence of, or can we construct, a curve $C$ of genus at least $4$, with $\Aut(C) = \{\operatorname{id}\}$  and CM Jacobian?
\end{question}

Versions of the above conjecture for the RM case will be discussed in \Cref{predsection}.

\subsection{Some examples (RM)}\label{secrm}
Ellenberg \cite{zbMATH01675790}, using branched covers of the projective line, constructed interesting families of curves whose Jacobians are acted on by large rings of endomorphisms. In his construction, the endomorphisms arise as quotients of double coset algebras of the Galois groups of these coverings. It is a generalisation of the work of Mestre and Brumer, who produce, for each odd prime $p$, families of curves X of genus $(p-1)/2$ such that $\End^0(\Jac(X))$ contains
the real cyclotomic field $\Q(\zeta+ \zeta^{-1})$, as well as work of Tautz-Top-Verberkmoes that we explain below. We refer to \emph{op. cit.} for more references and constructions (including the work of Hashimoto-Murabayashi, Bending, Shimada, etc.). See also \cite{zbMATH01545729}, for other related examples.

\begin{thm}[Tautz-Top-Verberkmoes {\cite{zbMATH00059649}}]\label{ttv}
Let $p$ be an odd prime, $\zeta_p$ a primitive $p$-th root of unity, and let $F\in \Z[x]$ be the minimal polynomial of $-\zeta_p-\zeta_p^{-1}$ , then 
\begin{equation}
C_t=C_{t,p}=\{y^2=x \cdot F(x^2-2)+t\} 
\end{equation}
defines a family of curves of genus $(p-1)/2$ whose jacobians contain the totally real field $\Q(\zeta_p +\zeta_p^{-1})$ in their endomorphism algebra.
\end{thm}

\begin{rmk}
We will get back to this example after we discuss \teic curves in \Cref{teiccurves}.
\end{rmk}

\begin{proof}[Sketch of the proof]
To see where the above equation comes from, consider the family of curves
\begin{displaymath}
D_t= \{y^2=x(x^{2p}+tx^p+1)\}.
\end{displaymath}
We picked it because it has many automorphisms related to $\zeta_p$. Indeed $\Z/2\Z \times \Z / p \Z$ acts via the hyperelliptic involution and $\zeta :(x,y)\mapsto (\zeta_p x, {\zeta_p}^{p+1/2}y)$. There's also the involution $\sigma : (x,y)\mapsto (\frac{1}{x},\frac{y}{x^{p+1}})$ which commutes with the hyperelliptic involution (therefore the quotient is again hyperelliptic), but not $ \Z / p \Z$. One verifies that $C_{t,p}=D_t/\langle\sigma\rangle$ \cite[Prop 3.]{zbMATH00059649}. The point now is that $\Z/p\Z $ does not descend to automorphisms but define correspondences. Indeed $[\zeta]$ defines an automorphisms of $\Jac (D_t)$ and we show now that $[\zeta] + [\zeta]^{-1}$ restricts to $\Jac (C_t)$, as desired. By working with the tangent spaces, it is enough to show that $\zeta + \zeta^{-1}$, seen as an element of the group ring $\C[\Aut (D_t)]$ stabilises the space of $\sigma$-invariant differentials in $H^0(D_t,\Omega^1_{D_t/\C})$. A basis of such space is given by the differentials $\omega_j= (x^j-x^{p-1-j})\frac{dx}{y}$, for $0\leq j \leq (p-3)/2$. And one computes
\begin{displaymath}
(\zeta ^* + (\zeta^{-1})^*) \cdot \omega_j= \left(\zeta_p^{(p+1)/2-j-1} +{(\zeta_p^{-1})}^{(p+1)/2-j-1} \right) \cdot \omega_j.
\end{displaymath}
\end{proof}
\begin{question}
For $p\geq 17$, are there any CM points on the families $C_{t,p}$? Can the CM points be described?
\end{question}
A priori it is not fully clear that the set of CM points on $C_{t,p}$ is actually finite. (At each $t_0 \in \C -\overline{\Q}$ we will have necessarily $\End^0(J(C_{t_0}))=\Q(\zeta_p +\zeta_p^{-1})$, but at algebraic numbers it can have bigger endomorphism ring).
\subsection{Some examples (CM)}\label{seccm}

\begin{thm}[De Jong-Noot {\cite{zbMATH04210319}}]\label{dejongnoot}
The families

\begin{equation}
\{y^3=x(x-1)(x-t_1)(x-t_2)(x-t_3) \}\subset \mathcal{M}_4;
\end{equation}

\begin{equation}\label{eq22}
\{y^5=x(x-1)(x-t)\} \subset \mathcal{M}_4;
\end{equation}

\begin{equation}
\{y^7=x(x-1)(x-t)\} \subset \mathcal{M}_6;
\end{equation}
have infinitely many fibers with CM Jacobians.
\end{thm}

\begin{rmk}
For a discussion about the corresponding Higgs bundle in \eqref{eq22} we refer to \cite[Example 4.4]{zbMATH06744545} ``An interesting Shimura curve in the Torelli locus''. See also the discussion about a question of Fujita in \emph{op. cit.} and references therein.
\end{rmk}

It is easy to observe that the curves in \eqref{eq22} have an automorphism of order 5 given by multiplying $y$ by a 5-th root of unity. This already implies $\End^0$ of their Jacobians contain $\Q[\zeta _5]$. An observation that we learned from a MathOverflow answer of Tim Dokchitser is that a large group of automorphisms of a curve forces presence of roots of unity in the endomorphism algebra of its Jacobian, and this is likely to produce products of abelian varieties with CM by abelian fields. We will come back to this point later in the text.

\begin{rmk}
Examples with $g = 5$ and $g = 7$ were later found by Rohde \cite{zbMATH05549716}. Moonen \cite{MR2735989} considered families of cyclic covers of $\mathbb{P}^1$, fixing the covering group and the local monodromies and varying the branch points. He proves that there are precisely twenty such families that give rise to a special subvariety in the moduli space of abelian varieties. See also \cite[Sec. 5]{moonenoort} for a complete discussion (especially Table 1 in \emph{op. cit.}).
\end{rmk}
To justify the $g\geq 8$ in \Cref{cofirstapp}, we conclude by presenting the two known examples of one dimension families with a dense set of CM points in $\mathcal{M}_7$:
\begin{thm}[Rohde]
The families 
\begin{displaymath}
\{y^9= (x-t_1)^3 (x-t_2)^5 (x-t_3)^5(x-t_4)^7  \}    \subset \mathcal{M}_7
\end{displaymath}
\begin{displaymath}
\{y^{12}= (x-t_1)^4 (x-t_2)^6(x-t_3)^7(x-t_4)^7   \}   \subset \mathcal{M}_7
\end{displaymath}
have infinitely many fibers with CM Jacobians.

\end{thm}

\subsection{Schottky problem}
This is the problem of characterizing which principally polarized abelian varieties are Jacobians of curves. There are several solutions or conjectural solutions of this problem: there is an analytical approach as well as a geometric one. Another interesting Tannakian approach was recently proposed by Weissauer and Kr\"{a}mer, and we refer to \cite{2025arXiv250107512P} for recent results in genus $\leq 5$ (and references there in, for more about this topic).

 However, to the best of our knowledge, they are not clearly related to the questions motivating our text. We refer to \cite{zbMATH00884942} for more about this. See also the discussion \cite[Sec. 4.17]{moonenoort}. Interestingly the Schottky problem appears implicitly also in the \teic viewpoint that we will describe later, as indeed  \cite[Thm. 1.1]{zbMATH06353727} and \cite{zbMATH06296347} demonstrate.

As Oort writes (in the notes \emph{CM Jacobians}, from a 2012 conference in Bordeaux): \emph{In general it is hard to see from properties of a given curve whether its Jacobian is a CM abelian variety. On the other hand known criteria which decide whether a given principally polarized abelian variety is a Jacobian does not enable us, it seems, to single out from the large class of CM abelian variety many CM Jacobians:
\begin{itemize}
\item starting with a curve $ C$ it is hard to decide whether $\Jac(C)$ is a CM abelian variety;
\item starting with an abelian variety A (maybe CM, principally polarized, $g>3$) it is hard to decide whether it is the Jacobian of an algebraic curve.
\end{itemize}}
In this direction, see also the discussion from \cite[Sec. 4.4]{zbMATH05124673}.

We pause here to say some words on computing endomorphisms of Jacobians. If the curve is defined over a number field, then Costa-Mascot-Sijsling-Voight \cite{zbMATH07009723} gave an answer (and Lombardo in genus 2 \cite[Sec. 3]{zbMATH06989568}, initiated by van Wamelen for the CM case). Their algorithm, if terminates, gives the correct answer and, conditionally on the Mumford-Tate conjecture it terminates \cite{zbMATH07370428}.

Unfortunately the examples worked out in \cite[Sec. 8]{zbMATH07009723} are of small genus. Can this algorithm be used to address \Cref{moonenoortquestion62}? Implemented in MAGMA there is also an algorithm that computes the group of automorphism of a curve. See also \cite{zbMATH02194080} as well as the book \cite{zbMATH01505376}.

\subsection{Noether-Lefschetz Locus}\label{NLsection}

Debarre-Lazlo \cite{zbMATH00006589} classified the Noether-Lefschetz locus in $\mathcal{A}_g$, i.e. the locus of abelian varieties with Neron-Severi rank $>1$. They proved that the irreducible components NL are of two types:
\begin{enumerate}
\item For each integer $k \leq g$ , the locus of principally polarized abelian varieties containing 
an abelian subvariety of dimension $k$ such that the induced polarization is of a fixed degree.
\item For every divisor $n$ of $ g$, $n \neq g$, the irreducible components of the locus of Hilbert modular varieties (cf. \Cref{examples}).
\end{enumerate}

The Noether-Lefschetz locus is a special case of a more general concept, namely the Hodge locus which we will discuss. More about Hodge theory and the framework of Shimura varieties can be found in \Cref{notationshimura}.

\subsection{Hidden symmetries and the André--Oort conjecture}\label{sectionao}

The angle presented here traces back to Serre’s work on \emph{open image theorems} for abelian varieties and the related work of Mumford \cite{zbMATH03271259} and Tate from 1969. They sought to predict the image of the Galois group acting on the $\ell$-adic Tate module of an abelian variety $A$, introducing the notion of the Mumford-Tate group. We refer to \Cref{hodgesection} for an introduction to Hodge structures and related topics.

\begin{defi}
The \textbf{Mumford-Tate group} of a Hodge structure $V$, $\MT(V)$, is the Tannakian group associated to the category $\langle V \rangle^\otimes \subset \Q \mathrm{HS}$. Equivalently, it is the $\Q$-Zariski closure of the $\C^*$-action defining the Hodge decomposition, or the stabilizer in $\operatorname{GL}(V)$ of the Hodge classes in arbitrary tensor operations.
\end{defi}

Deligne \cite{MR0498551} proved that the image of the Galois group acting on Tate module is a subgroup of $\MT \otimes \Q_\ell$. Pohlmann \cite{zbMATH03319219}, in `68, proved the Mumford-Tate conjecture for abelian varieties with Complex Multiplication. We recall now the most general definition of CM:
\begin{defi}\label{defcm}
A Hodge structure $V$ is \textbf{CM} if its Mumford-Tate group $\MT(V)$ is commutative.
\end{defi}

In its most elementary form the André--Oort conjecture (often abbreviated by AO) asserts the following.
\begin{thm}[André \cite{zbMATH01251645}]
Let $V$ be an irreducible algebraic curve in the complex affine plane, which is neither horizontal nor vertical. Then $V$ is a modular curve $Y_0(N)$, for some $N>0$, if and only if it contains infinitely many points $(j',j'')\in \C^2$ such that $j'$ and $j''$ are $j$-invariants of elliptic curves with complex multiplication.
\end{thm}

The Mordell-Lang conjecture, proven by Faltings, includes the famous statement that a smooth projective curve of genus $>1$ has only finitely many $K$-points over any number field $K$, and the Manin-Mumford conjecture which asserts the following:
\begin{thm}[Raynaud {\cite{zbMATH03929174}}]\label{mmthm}
Let $A$ be an abelian variety over a field of characteristic $0$ and $E$ be a subset of torsion points. Then the Zariski closure of $E$ is \emph{special}. That is, it is a finite union of torsion cosets $a+B$ for some abelian subvariety $B$ of $A$ and a torsion point $a\in A$.
\end{thm}

In 1989, André \cite{zbMATH00041964}, and later Oort in 1995 \cite{zbMATH01054988}, proposed the André-Oort conjecture (AO) for Shimura varieties, partly inspired by non-abelian analogues of the Manin-Mumford conjecture and the distribution of CM points on moduli spaces. Shimura varieties, a class of varieties that include modular curves and, more generally, moduli spaces parametrizing principally polarized Abelian varieties of a given dimension (possibly with additional prescribed structures), were originally introduced by Shimura in the 1960s in his study of the theory of complex multiplication.

\begin{defi}\label{defspec} An irreducible component of a Shimura subvariety of a Shimura variety $S$, or of its image under a Hecke operator, is called a \emph{special subvariety} of $S$.
\end{defi}
See for example \cite[Sec. 3.6]{moonenoort} for more about special subvarieties. Often, by abuse of notation, we will call special subvarieties of a Shimura variety simply \emph{sub-Shimura varieties}.

\begin{thm}[AO for Shimura varieties, \cite{2021arXiv210908788P, MR3744855}]\label{aothm}
Let $S$ be a Shimura variety. An irreducible subvariety is special if and only if it contains a Zariski-dense set of CM points.
\end{thm}

More generally, the above questions are inspired by the following. Given a reductive group $ G \subset \operatorname{GSp}_{2g}$, what can be said of the set
\begin{displaymath}
H_G:=\{x \in \mathcal{M}_g : \MT(x) \text{  is  conjugate to } G\} ?
\end{displaymath}
Assuming the Hodge conjecture, the points in $H_G$ correspond to curves whose Jacobian (or a power of it) has some \emph{unexpected} algebraic cycle. We will come back on this perspective in \Cref{zpconj}.

\begin{rmk}\label{rmklombetal}
Principally polarized abelian varieties with Mumford-Tate group smaller than the very general one (i.e. $\operatorname{GSp}_{2g}$) can have extra Hodge classes related to divisor classes or not (cf. \Cref{NLsection}). The Hodge classes of latter case are called \emph{exceptional}. So far, we didn't say much this fascinating subject. Often a abelian variety supporting exceptional Hodge classes is called \emph{degenerate}. An example of a curve whose Jacobian is CM and degenerate is 
\begin{displaymath}
C_m : \{y^2 = x^m + 1\} \in \Mg,
\end{displaymath}
(for $g=\lfloor m-1/2 \rfloor$)
for $m$ is an odd composite number \cite{arXiv:2211.03909}. We wish to refer the reader to the interesting recent work of Gallese, Goodson, and Lombardo \cite{2024arXiv240520394G}, for the computations of several interesting arithmetic invariants of $\Jac (C_m)$. We will see other examples, of a different kind, of degenerate Jacobians in \Cref{questionserregross}.

\end{rmk}
To connect with \Cref{autsection}, Riemann surfaces of higher genus with automorphism group of cardinality $>2$ have necessarily Mumford-Tate group smaller than the very general one.

\begin{rmk}[D. Lombardo]
From the examples discussed so far, it seems that most of CM points on $\Mg$ actually lie on the hyperelliptic locus (cf. \Cref{cmhype}) and one can wonder how hard it is to construct, for every $g$, a non-hyperelliptic Jacobian with CM other than Fermat. Davide Lombardo pointed out to the author the following construction, and we reproduce here his argument. Fix $d$ and consider
\begin{displaymath}
 F_{a,b,c} : y^d = x^a(1-x)^b
\end{displaymath}
for $a,b,c$ positive integers whose sum is equal to $d$. Each $F_{a,b,c} $ is a quotient of the Fermat curve $F_d$ \eqref{fermat} and therefore have CM Jacobians. Following the work of Coleman \cite[Sec. IV]{zbMATH04132336}, one can determine which $ F_{a,b,c}$ are hyperelliptic, and this seldom happens. Take $d=2k+1$ odd such that $d=a+b+c$ for some distinct $a,b,c$ such that $(d,a)=(d,b)=(d,c)=1$, then $ F_{a,b,c}$ is not hyperelliptic \cite[Prop. 8]{zbMATH04132336}. 

One could still ask for examples of non-hyperelliptic CM points of $\Mg$ that are not quotients of Fermat.
\end{rmk}

\section{Coleman--Oort conjecture}\label{coconj}

 One of the most ambitious conjectures about the distribution of CM points on $\Mg$ is due to Coleman and Oort, which we have already encountered in \Cref{realcommult}. We state here a geometric reformulation:

\begin{conj}[Coleman-Oort, geometric version]\label{cogeom}
For sufficiently large genus $g$ (specifically, $g \geq 8$), there are no sub-Shimura\footnote{Here by sub-Shimura varieties we mean special subvarieties in the sense of \Cref{defspec} (i.e. we allow Hecke translates). We prefer this, admittedly looser, terminology, because later we will have a different definition of special subvarieties. Namely the pull-back along the Torelli map of any special subvariety from $\Ag$ to $\Mg$.} varieties $Z \subset \mathcal{A}_g$ of positive dimension that are generically contained in $\mathcal{M}_g$. 
\end{conj}

Thanks to \Cref{aothm}-- the solution of  the André-Oort conjecture-- the above geometric question is equivalent to an arithmetic one stated in \Cref{cofirstapp}.

It is worth mentioning also the work by Toledo \cite{zbMATH04057672}, who seems to be the first to have considered such type of geometric questions. Indeed, in 1987, he wrote:
\begin{center}
 \emph{\dots ask if the image of $\Mg$ contains any straight lines of the symmetric geometry, i.e., any complex totally geodesic curves.}
\end{center}
(It's easy to see that sub-Shimura curves are totally geodesic. See \Cref{weakspe} and the discussion there for more about \emph{totally geodesic subvarieties}).)
Toledo's motivation came from differential geometry and curvature properties of the Torelli map, and we will discuss his result later, see \Cref{toledothm}. 
\begin{rmk}
Coleman was advertising his conjecture on CM Jacobians \cite[Conj. 6]{zbMATH04065160} in `87. The link between Coleman's conjecture and Toledo's question is the AO conjecture, as we explained before. The first appearance of a problem of AO type was in `89 in the work of André \cite{zbMATH00041964}, and, a posteriori, also in `88 in the work of Wolfart \cite{MR931211}. It is interesting to observe how many different perspectives and motivations lead mathematicians at the end of the 80s to consider these kind of questions. Understanding the precise link between them, and finding a framework for actually proving some of these conjecture kept the community occupied for several decades! To conclude, the first formulation of a AO for \emph{mixed} Shimura varieties appears in 2001 \cite[Lecture 3]{andre2001shimura} (see also later work of Gao \cite{zbMATH06801925} as well as \Cref{mixedshim} for a brief discussion about mixed Shimura varieties).
\end{rmk}
The most general question in this spirit is the following \cite[Qtn. 6.11]{moonenoort}. 
  
 \begin{question}
  Does the Torelli locus contain any totally
geodesic subvarieties (for the symmetric metric on $\Ag$) of positive dimension?  
  \end{question}
  
Related is also the following formulated by Farb, with different motivations and after his earlier work with Masur \cite{zbMATH01307806}:
\begin{question}[{\cite[Qtn. 2.5]{zbMATH05124673}}]
Does there exist some $g$ for which $\pi_1 (\Mg)$ contains a subgroup $\Gamma$ isomorphic to a cocompact (resp. noncocompact) lattice in $\operatorname{SO}(m,1)$ with $m \geq 5$ (resp. $m \geq 4$)? A cocompact lattice in $\operatorname{SU}(n, 1)$, $n\geq 2$? Must there be only finitely many conjugacy classes of any such fixed $\Gamma$ in $\pi_1 (\Mg)$?
\end{question}

We will later see a proof of the following, which asserts that such subvarieties are rare.
\begin{thm}[Geometric AO for $\Mg$ {\cite{MR3177266}}]\label{geomao}
If $g\geq 4$, the collection of totally
geodesic subvarieties generically contained in $\Mg$ is not Zariski dense in $\Mg$.
\end{thm}
\begin{rmk}
We remark here that the Coleman-Oort conjecture is \textbf{not} a consequence of AO, nor the more general Zilber-Pink conjecture (or just ZP, for brevity) that we will later discuss. CO can be thought as an (very!) effective version of \Cref{geomao}. Albeit effective approaches to the geometric ZP conjecture exist, it still seems difficult to implement them to obtain evidences in favour of the CO conjecture (especially for \emph{all} $g$).
\end{rmk}

There are actually many proofs of the above. It should be compared to the more surprising (and for which essentially one proof is known):
  \begin{thm}[Corollary of \cite{MR3744855}]
If $g\geq 4$, the collection of genus $g$ curves with CM Jacobian is not Zariski dense in $\Mg$.
\end{thm}
In fact a stronger result in proven in \emph{op. cit.}: there are only finitely many maximal special subvarieties of $\Ag$ generically contained in $\Mg$.

\subsection{Known cases}\label{knowncas}
The current state of this area is nicely summarized in Moonen’s article \cite{zbMATH07740282}, which builds on the work of Hain \cite{MR1722540} (which builds on the work of Farb and Masur mentioned above \cite{zbMATH01307806}), and de Jong–Zhang \cite{zbMATH05263283}. There has also been extensive and interesting work by the ``Italian school'' (Colombo, Frediani, Ghigi, Penegini, Pirola, Porru, Tamborini, Torelli, among others). Notably, they offer an extensive study of the differential-geometric properties of the Torelli map, particularly through the second fundamental form (see, e.g., \cite{zbMATH07308524}). Combining the two strands of research:

\begin{thm}\label{thmmoonenetal}
Let $g \geq 8$. Assume there are infinitely many smooth complex curves of genus $g$ whose Jacobians are CM abelian varieties. Then there exists a special subvariety $S \subset \mathcal{T}_g$ of positive dimension such that $S \cap \mathcal{M}_g \neq \emptyset$, and at least one of the following holds: 
\begin{enumerate}
\item $\dim(S) = 1$;
\item $\dim(S) = 2$ and $S$ is compact;
\item $S$ is a compact ball quotient.
\end{enumerate}

Moreover, any special subvariety generically contained in $\mathcal{M}_g$ satisfies $\dim(S) \leq 2g - 1$ if $g$ is even, and $\dim(S) \leq 2g$ if $g$ is odd.
\end{thm}

Ball quotients, which are locally Hermitian spaces of the form $\Gamma \backslash \mathbb{B}^n$ (where $\Gamma$ is neat), are particularly notable examples of Shimura varieties. Picard modular varieties, discussed in more details in \Cref{examples}, belong to this class. As far as we know, not much is known about which types of ball quotients can arise in $\mathcal{M}_g$. 

A recent theorem of Yeung, combined with the above works, elucidates the study Shimura varieties on the complement of the hyperelliptic locus $\mathcal{H}_g$ in the open Torelli locus (he actually describes also the case $g=2,3,4$). See also \cite{zbMATH07629821} for another approach. 

\begin{thm}[Yeung {\cite[Thm. 1 and Thm. 2]{zbMATH07574854}}]\label{yeungthm}
Let $g\geq 5$, the set $j(\Mg - \mathcal{H}_g)$ contains no sub-Shimura varieties of $\Ag$.

\end{thm}
The above theorem is really about sub-Shimura varieties contained in $j(\Mg - \mathcal{H}_g)$, rather than the weaker notion of \emph{generically contained}. We will comment more on this in \Cref{teichshi}.

To take care of the Shimura curves, the author proves that they are indeed Shimura-\teic curves and uses a result of M\"{o}ller \cite{zbMATH05919111} (in genus $>5$) and Aulicino and Norton in genus 5 \cite{zbMATH07347328}. We will get back to this point after we discuss more about \emph{\teic curves}.

\subsection{Work of Toledo}\label{sectworktol}

Toledo \cite{zbMATH04057672} proved that most of compact totally geodesic curves in the Siegel moduli space $\mathcal{A}_g$ cannot lie in the image of the Torelli morphism. The meaning of "most" is in terms of the holomorphic sectional curvature of Siegel space:

\begin{thm}\label{toledothm}
Let $g\geq 3$, and $C \subset \Mg$ a smooth compact curve of genus $h$ such that $j(C)\subset \Ag$ is totally geodesic of curvature $-1/l$. Then $l\leq (g-1)/3$.
\end{thm}
\begin{proof}[Sketch of the proof]
By restricting the universal family of curves $\mathcal{C}_g\to \Mg$ to $C$ we obtain a surface $S$ and a fibration $\pi : S \to C$ in smooth genus $g$ curves. We compute the Chern numbers of $S$:
\begin{displaymath}
c_2(S)=c_1(C)c_1(\text{fibre})=4(g(C)-1)(g-1)
\end{displaymath}
for $c_1(S)^2$ we use Grothendieck-Riemann-Roch. So we just need to compute $ch^1(R^1\pi_* \mathcal{O}_S)$. To see that it is equal to $-l(g(C)-1)$ we use the  Gauss-Bonnet formula and the fact that $j(C)$ is totally geodesic of curvature $-1/l$.

Since $S$ is necessarily minimal and of general type, the Miyaoka-Yau inequality gives
\begin{displaymath}
c_1^2= 4(g(C)-1)(2g-2+3l) \leq 3c_2^2=12(g(C)-1)(g-1),
\end{displaymath}
which implies the desired bound.
 \end{proof}

\begin{rmk}\label{rmktoledo}
It seems that one can improve the above bound to $l< (g-1)/3$, since the equality in MY implies that $S$ is a ball quotient (as one obtains from Yau's proof of the Calabi's conjecture). But a result of Liu \cite{zbMATH00980882} shows that a compact complex-hyperbolic surface $S$ does not admit nonsingular holomorphic fibrations over complex curves. Unfortunately it seems that the result is not widely accepted by the community. However, another proof appears also in \cite{zbMATH05816763}.
\end{rmk}

\begin{cor}
For $g\leq 4$, $\Mg$ contains no compact totally geodesic curves. 
\end{cor}
\begin{proof}
The result is interesting only for $g=3,4$, in this range $\Mg$ contains indeed many compact curves. In \cite[Sec. 2]{zbMATH04057672}, Toledo rules out the case $l=1$, so this follows from the above theorem.
\end{proof}

\subsection{Related works}
We collect here some related work
\begin{itemize}
\item Another  fruitful method for excluding Shimura curves is the use of the so called \emph{Arakelov inequalities}. We refer to \cite{zbMATH05635169} for details on this powerful technique. 
\item See also \cite{2024arXiv240306217L, zbMATH05702783} for further investigations.
\item In \cite{2024arXiv240106577D}, de Gaay Fortman and Schreieder show, among other things, that for a very general principally polarized abelian variety of dimension at least four, no power is isogenous to a product of Jacobians of curves. 
\item In \cite{zbMATH06108421}, Andreatta discusses an analogue of the Coleman-Oort conjecture for the locus of degenerate irreducible curves.
\end{itemize}

\newpage

\part{Teichm\"{u}ller viewpoint}\label{part2}

\section{Dynamical and metric perspective}
We now introduce the second main source of special subvarieties of $\Mg$ (and related moduli spaces like $\mathcal{M}_{g,n}$ and $\Og$). For an introduction of this circle of ideas, we refer to \cite{zbMATH07666127}.

\subsection{Metric geometry} This is a large world with contributions of many people. We simply refer to the recent survey \cite{2024arXiv240500761C}.

Let $X$ be a smooth quasi-projective variety. The \emph{Kobayashi pseudometric} $d_K=d_{K,X}$ is the largest pseudometric on a (smooth quasi-projective) variety $X$ such that
\begin{displaymath}
d_K(f(x),f(y))\leq \rho (x,y),
\end{displaymath}
for all holomorphic maps $f: \Delta \to X$ (where $\rho (x,y)$ is the Poincaré metric)--cf. Kobayashi's book \cite{zbMATH01161515}.

\begin{defi}
A complex manifold $Y$ is said to be \textbf{Kobayashi hyperbolic} if the Kobayashi pseudo-metric is in fact a metric: $d_{K,X}(p,q)>0, \forall p\neq q$.
\end{defi}
\begin{defi}\label{kobgeod}
A subvariety $Y \subset X$ is called \textbf{K-totally geodesic} if $d_{K,Y}$ is equal to the restriction to $Y$ of $d_{K,X}$.
\end{defi}
 The definition makes sense also if $d_K$ is not a metric, and for example, for abelian varieties (where $d_K=0$) the K-totally geodesic subvarieties are translate of sub-abelian varieties \cite{zbMATH07056564}.

We can represent $\Mg$ (or better its complex points) via the uniformization $\Mod(g) \backslash \mathcal{T}_g$. The \teic space $ \mathcal{T}_g$ is inhomogeneous, for example $\Aut(\mathcal{T}_g)=\Mod(g)$ is discrete. Nevertheless there are a lots of isometric holomorphic map  $F:\mathbb{H} \to  \mathcal{T}_g$. We call such curves \emph{complex \teic geodesics} ($\C$T-geodesic). Indeed every two points can be connected by a $\C$T-geodesic. The projection of the image of $F$ to $\Mg$ is almost always dense. If it lies in an algebraic curve $C\subset \Mg$, we call $C$ a \emph{\teic curve}. We will now explain that they are rare objects and what is known about their existence.
\begin{rmk}\label{rmkroy}
Royden, in the 70s, proved that the \teic metric and the Kobayashi metric are the same \cite{zbMATH03352015, zbMATH03472497}. The result that the holomorphic automorphism group of the \teic space is exactly the mapping class group is a consequence of this fact. Furthermore $\Mg$ does not admit any nontrivial automorphisms or correspondences (as long as $g\geq 3$), \cite[Thm. 6.1]{mochizuki}.
\end{rmk}From now on, we see moduli space $\Mg$ both as metric space and an algebraic variety. The metric comes from the \teic distance (of \kob) between $[X],[Y]\in \Mg$.

It is proved in \cite[Thm. 1.2]{zbMATH06260641} that, unless $g=1$, the Torelli map $j: \Mg \to \Ag$ is not an isometry for the Kobayashi metric.
\subsection{Flat uniformization and $\operatorname{GL}_2(\R)^+$-action}
\begin{thm}
Let $g\geq 1$. Every $[X]\in \Mg$ can be presented can be presented as the quotient of a polygon  $\mathbb{P}\subset \R^2$ by translations.
\end{thm}
\begin{rmk}
For example, in genus 1, one can write $X=\C/L$ for some lattice $L$. Alternatively, if we choose a parallelogram $\mathbb{P}\subset\C$ that is a fundamental domain for the action of $L$, we construct $X$ by gluing together opposite sides of $\mathbb{P}$.
\end{rmk}
The problem of the above theorem is that there are many different ways to present $X$ in this way. The better version is the following:
\begin{thm}
Every $(X,\omega)\in \Og$ can be presented as $(X,\omega)= (\mathbb{P},dz)/ \sim$.
\end{thm}
(The induced flat metric on $X$ has, in general, isolated singularities of negative curvature corresponding to vertices of $\mathbb{P}$.)

As alluded to in the introduction, on the Hodge bundle (without the zero section) $\Og$ there exists a natural action of $\operatorname{GL}_2(\R)^+$. We explain it in the flat picture. Write $\R^2=\C$ with $z=x+iy$. Given an element $g\in \operatorname{GL}_2(\R)^+$ we consider
\begin{displaymath}
g \cdot (X,\omega)=(X_g,\omega_g)=(g(\mathbb{P}),dz)/ \sim
\end{displaymath}
where $g$ naturally acts on $\mathbb{P}\subset \R^2$.

In fact, the action is most interesting when restricted to $\operatorname{SL}_2(\R)$, and we now explain it more geometrically. By considering the $K\cdot A \cdot K$ decomposition, it is enough to describe the action of  $a_t=  \begin{bmatrix}
    e^t & 0  \\
    0 & e^{-t}
  \end{bmatrix}$ and of the rotation by $\theta$, $k_\theta \in \operatorname{SO}(2)$:
  \begin{itemize}
  \item The action of $a_t $ is by stretch and shrink, as the following picture shows:
  \begin{center}
 \includegraphics[width=0.5\textwidth]{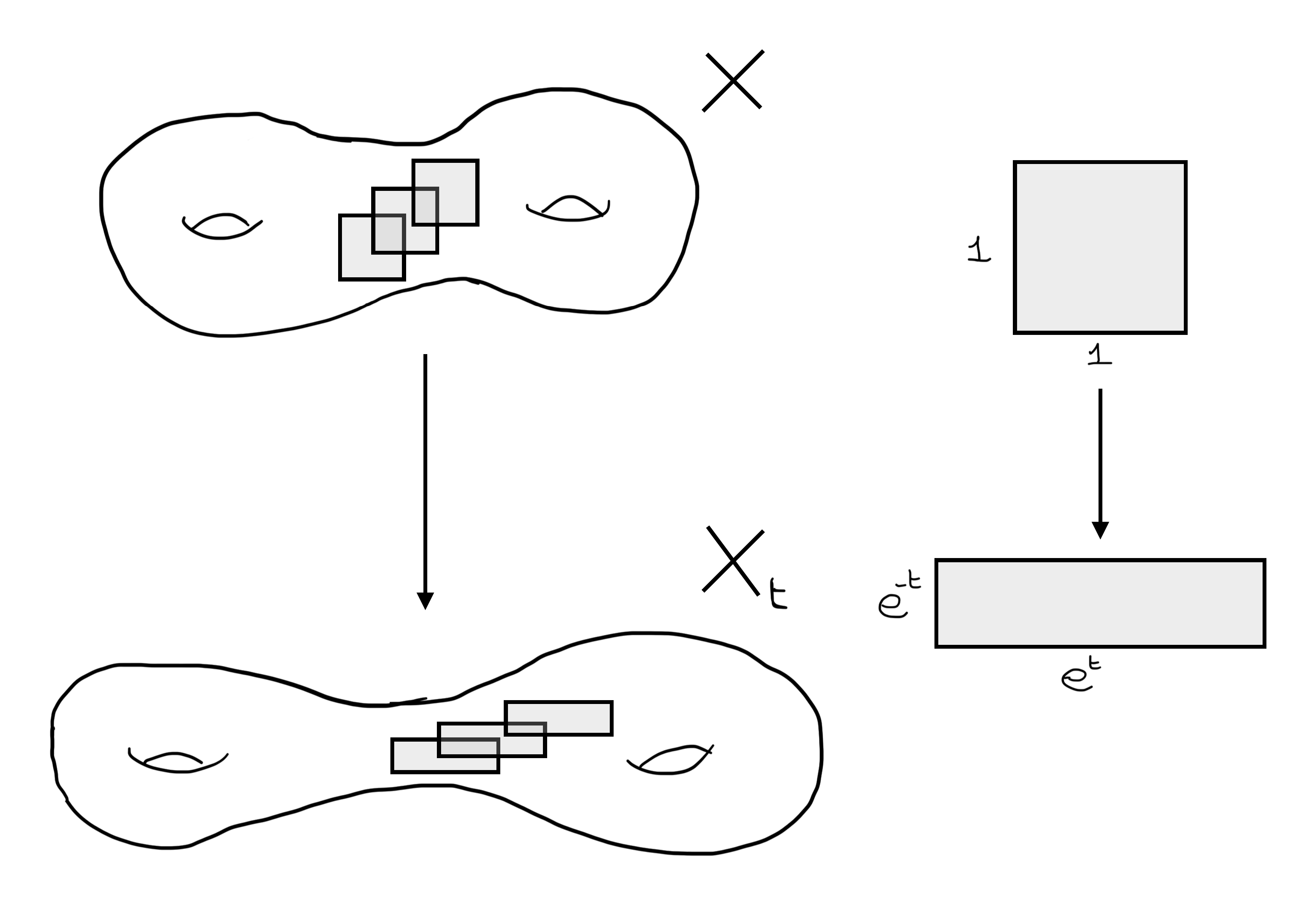}
\end{center}
\item The action of $k_\theta$ is given by
\begin{displaymath}
k_\theta \cdot (X,\omega)=(X, e^{i \theta } \omega)
\end{displaymath}
(the two $X$s appearing above have the same complex structure).
  \end{itemize}

\subsection{\teic curves}\label{teiccurves}
Consider the digram:
\begin{center}

\begin{tikzpicture}

  \matrix[matrix of math nodes, row sep=2em, column sep=2em, text height=1.5ex, text depth=0.25ex] (m) {
    \Gl_2(\R)^+ & \Gl_2(\R)^+ \cdot (X,\omega)\subset \Og \\
   \Gl_2(\R)^+/SO(2)\cong \mathbb{H}  & \mathcal{M}_g \\
  };

  \path[->]
    (m-1-1) edge (m-1-2) 
    (m-1-1) edge (m-2-1) 
    (m-1-2) edge (m-2-2) 
 (m-2-1) edge node[above] {$f$} (m-2-2); 
\end{tikzpicture}
\end{center}
Denote by $\operatorname{SL}(X,\omega)\subset \operatorname{SL}_2(\R)$ the stabilizer of $(X,\omega)$, we observe that $f$ factors through $V:=\operatorname{SL}(X,\omega)\backslash \mathbb{H} $.
\begin{question}\label{questionteich}
When is $V$ algebraic? More generally, when does $f$ factorise through a strict algebraic subvariety?
\end{question}
We refer to the image of $f$ as a \emph{complex geodesic}. We remark that $\operatorname{SL}(X,\omega)$ is always a discrete subgroup of $ \operatorname{SL}_2(\R)$, which is in fact often trivial. For this theory, the $V$s that are not Zariski dense certainly play a \emph{special} role.

 \begin{rmk}
  It is a deep theorem of Eskin, Mirzakhani, and Mohammadi, that all $\Gl_2(\R)^+$-orbit closures
are invariant varieties (see \Cref{properties}), but one almost never needs to appeal to this theorem to show a given variety is invariant.
\end{rmk}
\subsubsection{Square-tiled surfaces} Consider a translation surface $(X,\omega)$ (cf. the next section for a detailed discussion) such
that its period point has rational coordinates in the relative cohomology. We can assume $\omega$ has entries in $\Z[i/N]$. Fixing a reference point $0\in X$, we have a covering map
\begin{displaymath}
X \to \C/ \Z + i \Z , p \mapsto  N \cdot \int_0^p \omega
\end{displaymath}
which has degree $N [\omega] \cap [\overline{\omega}]$.

In this case, the orbit   $\Gl_2(\R)^+ \cdot (X, \omega)$ is a \teic  curve, with field of affine definition $\Q$. Indeed, the stabilizer of $(X, \omega)$ is a finite index subgroup of $\operatorname{SL}_2(\Z)$.

\subsubsection{The regular n-gon (after Veech)} We present some of the results of Veech \cite{zbMATH04107186}.

\begin{thm}[Veech]
For $(X,\omega)= (\mathbb{P}_n,dz)/\sim$ (the regular $n$th-gon), $\operatorname{SL}(X,\omega)$ is always a lattice. Furthermore
\begin{displaymath}
\operatorname{SL}(X,\omega)=
\begin{cases}
\Delta(2,n,\infty) & \text{if } n=2g+1 \text{ is odd}, \\
\Delta(n/2,\infty,\infty) & \text{if } n =4g, 4g+2\text{ is even}.
\end{cases}
\end{displaymath}

\end{thm}
In particular, the above theorem produces three examples of \teic curves in $\mathcal{M}_2$ (i.e. for $n=5,8,10$). See also \cite[Sec. 3]{zbMATH07666127} for more about billiards.

\begin{rmk}
C. McMullen kindly pointed out to us that the examples $C_{t,p}$ from \Cref{ttv} are \teic curves associated to billiards in a regular $p$-gon. Indeed any Riemann surface parametrized by such \teic curve has a factor with RM by the trace field $K$ of the Veech group (in this case the triangle group $\Delta(2,p,\infty)$), so $K =\Q(\zeta_p +\zeta_p^{-1})$. Since $p=2g+1 $, we have $g = [K:\Q]$ so in fact the whole $\Jac(X)$ has RM.
\end{rmk}

\subsubsection{Triangle group series of Bouw-M\"{o}ller \cite{zbMATH05779356}.}
There is a two-parameter family of \teic curves with Veech groups $\Gamma_{m,n}$ commensurable to triangle reflection groups $\Delta (m,n,\infty)$. See also \cite{zbMATH06180410}.

See also \cite[Sec. 7]{zbMATH01963976} for a related discussion regarding the case $m=2$ and $3$, as well as explicit algebraic equations.

\subsection{Classification in genus 2, after McMullen}\label{genus2mc}
Since there is already a beautiful account in \cite{zbMATH07666127, zbMATH01963976}, we simply describe some of the results of McMullen that classify \teic curves in genus two.
\begin{thm}[McMullen]
There are infinitely many \teic curves in $\mathcal{M}_2$
\end{thm}
In fact, in $\Omega \mathcal{M}_2$ the invariant subvarieties are:
\begin{itemize}
\item $\Omega \mathcal{M}_2(2)$ (which is Zariski closed);
\item  $\Omega \mathcal{M}_2(1,1)$ (open dense);
\item $\Omega E_D$, the locus of $(X,\omega)$ such that $ \Jac(X)$ has real multiplication by $\mathcal{O}_D=\Z[x]/x^2+bx+c $ with $D:=b^2-4c \geq 4$ and congruent to 0 or 1 mod 4, and $\omega$ is an eigenform for the real multiplication (i.e. $T^*\omega = \lambda \omega$);
\item $\Omega W_D = \Omega \mathcal{M}_2(2)\cap \Omega E_D$ which is invariant since it is the intersection of two invariants. 
\end{itemize}
Finally $\Omega W_D $ projects to $\mathcal{M}_2$, with image $W_D$, a finite union of \teic curves. Determining the number of connected components is hard, in that directions McMullen proves:

\begin{thm}[McMullen]
$W_D$ is irreducible when $D\neq d^2$, except it has two components when $D>9$ and congruent to $0$ or $1$ mod $8$.
\end{thm}

The conditions about RM appearing above are the first hint of the link between the fist two parts of text.
\subsubsection{Explicit example in genus 2, after Kumar and Mukamel}
The authors give the first explicit algebraic models of \teic curves of positive genus. This is  based on the study of certain Hilbert modular forms and the use of Ahlfors’s variational formula to identify eigenforms for real multiplication. We record here just one example of their results:

\begin{thm}[{\cite[Thm. 4.1]{zbMATH07004492}}]
The Jacobian of the curve
\begin{displaymath}
y^2=x^5-2x^4-12x^3-8x^2+52x+24
\end{displaymath}
admits real multiplication by $\Z[\sqrt{3}] $ with eigenforms $dx/y$ (which has a double zero) and $x \cdot dx/y$.
\end{thm}

\subsection{Shimura-\teic curves and applications}\label{teichshi}
We continue with the link from \Cref{coconj}. Looking for special behaviours of families of curves in $\Mg$ one is naturally brought to the following. A curve $C\subset \Mg$ is \emph{Shimura-\teic} (ST-curves) if it has both properties, i.e. it is a \teic curvea and, when considered in $\Ag$, Shimura curves. 

We start by recalling the work of M\"{o}ller \cite{zbMATH05919111} plus Aulicino and Norton for $g=5$ \cite{zbMATH07347328}:
\begin{thm}\label{reicshcurves}
For $g = 2$ and $g\geq 5$ there are no ST-curves.
In both $\mathcal{M}_3$ and in $\mathcal{M}_4$ there is only one ST-curve:
\begin{displaymath}
\{y^4=x(x-1)(x-t)  \} \subset \mathcal{M}_4,
\end{displaymath} 
\begin{displaymath}
\{y^6 =x(x-1)(x-t)\} \subset \mathcal{M}_4.
\end{displaymath}

\end{thm}
The following natural analogue seems to be open, cf. \cite[Prob. 6 (page 789)]{zbMATH06272629}. We call an irreducible (Zariski closed) subvariety $W\subset \mathcal{M}_g$ \emph{generically Shimura} if the Zariski closure of $j(W)\subset \mathcal{A}_g$ is a Shimura subvariety (or an Hecke translate of such), where $j$ is the Torelli map.
\begin{question}
Are there \teic curves $V\subset \mathcal{M}_g $ that are generically Shimura?
\end{question}
Clearly \Cref{cogeom} implies that the answer should be negative, for $g>>0$. Nevertheless, the above question could be a good starting point in favour of the Coleman--Oort conjecture.

Now we give a sketch of the proof of \Cref{yeungthm}. A key input is the following:

\begin{thm}[Antonakoudis, {\cite[Thm. 1.1]{zbMATH06748171}}]
There is no holomorphic embedding of complex unit ball of dimension $>1$ into $\mathcal{T}_g$ which is isometric with respect to the Kobayashi metrics.
\end{thm}

\begin{proof}[Sketch of the proof of \Cref{yeungthm} (after Yeung)] If we consider the Torelli map as a map between stacks, it is known to be an immersion apart from the hyperelliptic locus (where it is 2 to 1).
\begin{enumerate}
\item Any symmetric variety $S$ in $j(\Mg - \mathcal{H}_g)\subset \Ag$ has to be of real rank 1 as a locally symmetric space (Hain \cite{MR1722540}, cf. also \Cref{thmmoonenetal}). Therefore $S$ has to be a complex ball quotient.
\item Yeung observes that a ball quotient in $j(\Mg - \mathcal{H}_g)\subset \Ag$ which is a sub-Shimura variety of $\Ag$ is necessarily totally geodesic in $\Mg$ for the Kobayashi metrics, see \cite[Lem. 2]{zbMATH07574854}
\item Antonakoudis result recalled above implies that the complex dimension of $S$ is 1 and hence $M$ is a hyperbolic Riemann surface. 
\item Such a Riemann surface $S$ has to be a Shimura-\teic in the above terminology, since the Kobayashi metric, which is well-known to be the same as the \teic metric on \teic spaces (cf. \Cref{rmkroy}), is the same as the natural hyperbolic metric on $S$ as it is a totally geodesic curve in $\Ag$
\end{enumerate}
Now we conclude thanks to \Cref{reicshcurves}. 

\end{proof}
(M\"{o}ller already notices that his results, and the work of Viehweg-Zuo, implies that there is no Shimura curve in $\Ag$ for $g\geq 6$ that lies entirely in $\Mg$).

\begin{rmk}
Mok \cite[Thm. 5.1]{zbMATH07629821} proved that in a Hermitian symmetric space of noncompact type, a totally geodesic complex submanifold is automatically given by holomorphic isometric embedding with respect to the Kobayashi metric.
\end{rmk}

\section{Translation surfaces and their stratified moduli space}

A \emph{translation surface} is a pair $(X,\omega)$, where $X$ is a curve and $\omega$ is a non-zero global section of the cotangent bundle of $X$. If $g \geq 1$ is the genus of $X$, then $\omega$ has $2g-2$ zeros, counted with multiplicities. The Hodge bundle $\Omega \mathcal{M}_g \to \mathcal{M}_g$ is naturally equipped with a stratification by smooth (though not necessarily irreducible, cf. \Cref{irrcomp}) algebraic subvarieties $\Omega \mathcal{M}_g(\kappa)$, indexed by the multiplicities of zeros of $\omega$. 

A holomorphic 1-form $\omega$ on $X$ provides a collection of ``charts'' on $X$ mapping to $\mathbb{C}$, where the transition maps are translations. These charts ramify at finitely many points corresponding to the zeros of $\omega$ and are locally described by $z \mapsto \int_{z_0}^z \omega$. These charts induce the so-called \emph{period coordinates} on $\Omega \mathcal{M}_g(\kappa)$. As a result, there is an action of $\operatorname{GL}_2(\mathbb{R})^+$ on $\Omega \mathcal{M}_g(\kappa)$, locally represented as a diagonal action on a product of copies of $\mathbb{C} \cong \mathbb{R}^2$ (that was also described in the flat picture in the previous section). While this action seems outside the scope of algebraic geometry, it turns out to be intimately related to Hodge theory. As explained before, cf. \Cref{questionteich}, of particular interest in the proposed investigation are the $ \operatorname{GL}_2(\R)^+$-invariant varieties\footnote{It turns out that this class coincides with the orbit closures $\mathcal{N} := \overline{\operatorname{GL}_2(\mathbb{R})^+ \cdot (X,\omega)} \subset \Omega \mathcal{M}_g(\kappa)$, where $\overline{(\cdot)}$ denotes the topological closure, also known as \emph{affine invariant submanifolds} thanks to a deep theorem of \cite{zbMATH06487150}.}.

 These submanifolds are, in fact, algebraic subvarieties of $\Omega \mathcal{M}_g(\kappa)$ and project to interesting subvarieties of $\mathcal{M}_g$. The fundamental invariants associated with $\mathcal{N}$ are three natural numbers: $r$, $d$, and $t$, called the cylinder rank, degree, and torsion corank of $\mathcal{N}$, respectively, as we will see in more details in the next section.

A remarkable constraint on the distribution of orbit closures was discovered by Eskin, Filip, and Wright (for a beautiful overview and further references we simply refer to Filip's notes \cite{filipnotes}):

\begin{thm}[{\cite[Thm. 1.5]{zbMATH06890813}}]\label{thmEFWoriginal}
In each (irreducible component of each) stratum $\Omega \mathcal{M}_g(\kappa)$, all but finitely many orbit closures have rank 1 and degree at most 2. 
\end{thm}

We will come back on this in \Cref{diagramatyp}, and give a sketch of its proof following \cite{2024arXiv240616628B}.

\subsection{Properties of invariant subvarieties}\label{properties}

Given $(X,Z)$ a pair of a Riemann surface and a finite set of points $Z=\{z_1, \dots, z_n\} \subset X$ (corresponding to the zeros of $\omega$), we have a short exact sequence:

\begin{equation}\label{eqrelcoho}
0 \to \tilde{H}^0(Z; \Z) \to H^1(X,Z; \Z)\to H^1(X;\Z)\to 0.
\end{equation}
The middle term is \emph{relative cohomology} and the the last term the \emph{absolute cohomology}. Note that the form $\omega$ naturally induces a class inside both $H^1(X,Z(\omega); \Z)$ and $H^1(X;\Z)$ by integrating over homological cycles and using homology-cohomology duality.

Fix a base point $(X_0, \omega_0 )\in \Omega \mathcal{M}_g (\kappa)$, and look at the developing map/period coordinates, where $\widetilde{-}$ denotes the universal covering
\begin{equation}\label{def:dev}
\operatorname{Dev}: \widetilde{\Omega \mathcal{M}_g (\kappa)}\to H^1(X_0,Z(\omega_0); \C).
\end{equation}
Concretely, the map sends $(X',\omega')$ to $\omega' \in H^1(X',Z(\omega'); \C)$ and then applies the flat transport to obtain an element of $H^1(X_0,Z(\omega_0); \C)$.
By a theorem of Veech, locally, $\operatorname{Dev}$ is biholomorphic.

We have a version of \eqref{eqrelcoho} for flat bundles above $\Omega \mathcal{M}_g (\kappa)$ (see also \cite[3.1.13]{filipnotes}):
\begin{equation}
\label{bundextseq}
0 \to W_0 \to H^1_{rel}\to H^1 \to 0.
\end{equation}
The notation is justified by the fact that $W_0$ denotes the weight-zero local subsystem for the mixed VHS structure on $H^{1}_{rel}$.

We recall a special case of a theorem of Eskin, Mirzakhani, and Mohammadi \cite[Thm. 2.1]{zbMATH06487150} see also \cite[Thm. 1]{filipnotes}:
\begin{thm}[Linearity/topological rigidity]\label{linearity}
Each orbit closure $\mathcal{N} \subset \Omega \mathcal{M}_g(\kappa)$ is locally in period coordinates a \emph{linear manifold}, i.e. any sufficiently small open $U\subset \mathcal{N}$ is such that $\operatorname{Dev}(U)$ is an open set inside a linear subspace.
\end{thm}

\begin{rmk}
The very first investigations of the notion of \emph{linear manifold} that comprises \teic curves, Hilbert modular surfaces and the ball quotients of Deligne and Mostow (that we will review in \Cref{digballquotients}) can be found in the article of M\"{o}ller  \cite{zbMATH05323603}.
\end{rmk}

Let $\mathcal{N} \subset \Omega \mathcal{M}_g(\kappa)$ be an orbit closure, and denote by $T \mathcal{N}$ its tangent bundle. The linearity of orbit closures, as recalled in \Cref{linearity}, allows one to realize $T \mathcal{N}$ as a local subsystem of $H^1_{rel}$. As a consequence we have another short exact sequence of real bundles above $\mathcal{N}$:
\begin{equation}
0 \to W_0 (T \mathcal{N})\to T \mathcal{N}\to H^1 (T \mathcal{N})\to 0 .
\end{equation}
To explain the notation: $H^1 (T \mathcal{N})$ is just the image of $T \mathcal{N} \subset H^1_{rel}$ in $H^1$  (the absolute cohomology), and $W_0(T \mathcal{N}) = T \mathcal{N} \cap W_0$.

\begin{thm}[Isolation property, {\cite[Thm. 2.3]{zbMATH06487150}, see also \cite[Thm. 4.1.8.]{filipnotes}}]\label{isolation}
For every sequence of orbit closures $(\mathcal{N}_i)$ there are $\operatorname{SL}_2(\mathbb{R})$-invariant measures $\mu_i$, and after passing to a subsequence $(\mathcal{N}_i,\mu_i)$ there is another linear immersed submanifold $\mathcal{M}$, with finite $\operatorname{SL}_2(\mathbb{R})$-invariant measure $\mu$ such that $\mathcal{N}_i\subset \mathcal{M}$ and $\mu_i \to \mu$.
\end{thm}

We can now recall the following theorem of Filip \cite{zbMATH06575346} (which builds on the earlier work \cite{zbMATH06641855}, as well as earlier work of Wright \cite{zbMATH06323273} and M\"{o}ller \cite{zbMATH05015436}), see also \cite[Thm. 2 and Thm. 4.4.2]{filipnotes}. We denote by $K=K_{\mathcal{N}}\subset \R$ the smallest field over which $\mathcal{N}$ is $K$-linear in the sense of \Cref{linearity}. We have:

\begin{thm}[S. Filip]\label{mainthmalgfilip}
Let $\mathcal{N}\subset \Omega \mathcal{M}_g(\kappa)$ be an orbit closure. There exist, up to isogeny, a factor $\mathcal{F}\subset \mathcal{J}$ of the relative Jacobian over $\mathcal{N}$, and a subgroup $\mathcal{S}$ of the free abelian group on the zeros of $1$-forms, such that:
\begin{enumerate}
\item $\mathcal{F}$ admits real multiplication by $K$ (i.e. $K=\operatorname{End}^0(\mathcal{F}_{\mathcal{N}})$), which is, in particular, a totally real number field;
\item The Abel-Jacobi map, possibly twisted by real multiplication, $AJ: \mathcal{S}\to \mathcal{F}$ assumes torsion values;
\item For each $(X,\omega)\in \mathcal{N}$, $\omega$ is an eigenform for the real multiplication on the fibre of $\mathcal{F}$ above $X$ (see also \cite[\S 4.2.17]{filipnotes}).
\end{enumerate}
Furthermore, these conditions, together with a dimension bound, characterize $\mathcal{N}$.
\end{thm}
(See \cite[\S 4.5]{filipnotes} for more details on the \emph{twisted torsion} condition.)

\begin{cor}\label{thmstrorbit}
Each orbit closure is algebraic, that is a Zariski closed subvariety $\Omega \mathcal{M}_g$. In fact, orbit closures are defined over $\overline{\Q}$, and their Galois conjugates are again orbit closures. 
\end{cor}
From \Cref{mainthmalgfilip}, one may introduce \emph{real multiplication, torsion and eigenform} invariants (see \Cref{diagramatyp}, for a geometric description). In particular, we write
\begin{itemize}
\item  $r=r _\mathcal{N}$, the \emph{(cylinder) rank} $\frac{1}{2}\dim H^1(T\mathcal{N})$ (which is in fact an integer, since $T\mathcal{N}$ is known to be symplectic thanks to \cite{zbMATH06801922});
\item $d=d_\mathcal{N}$, to simply denote the degree of the number field defining the linear equations of $\mathcal{N}$ (that is $[K:\Q]$);
\item And finally $t=t_\mathcal{N}$ the \emph{torsion corank} is $\dim W_0(T \mathcal{N})$.
\end{itemize}
With this notation, the dimension bound appearing at the end of \Cref{mainthmalgfilip} is just
\begin{displaymath}
\dim \mathcal{N}=2r_{\mathcal{N}}+t_{\mathcal{N}}.
\end{displaymath}

In fact the above theorem is close to the philosophy explained in \Cref{sec2}. Namely of finding subvarieties of $\mathcal{M}_g$ rich in Hodge symmetries.

\subsubsection{Further remarks on linearity}
\begin{itemize}
\item There are examples of linear subvarieties
of $\Og (\kappa)$ which are not invariant, namely the Linear Hurwitz spaces. Let indeed $E$ be an elliptic curve and $H_{g,d}(E)$ be the coarse Hurwitz moduli space parametrizing $[X]\in \Mg$ together with a degree $d$ branched cover $X\to E$. By pullback along the branched cover we obtain a $1$-form on each $X$, and therefore a map (in fact an immersion) $H_{g,d}(E)\to \Og$. The (connected components) of the intersection $H_{g,d}(E)\cap  \Og(\kappa)$, are linear subvarieties and are not invariant unless $E$ is defied over $\overline{\Q}$, see  \cite[Sec. 1.2.2]{2022arXiv220206031K}. In fact for $g>2$ these are the only algebraic leaves of the isoperiodic foliation \cite[Thm. 1.5]{zbMATH07662889}.

\item  We clarify the relation between the two notions here (linear and $\operatorname{SL}_2(\R)$-invariant), cf. \cite[Thm. 5.1]{zbMATH06731862}. First of all any closed set $M \subset \Og$ locally defined by real linear equation is $\Sl_2 (\R)$-invariant. Furthermore, let $N \subset \Omega \mathcal{M}_g(\kappa)$ be an algebraic variety whose irreducible components have dimension at least $d \geq 1$. Suppose that for every $(X,\omega) \in N$ there is a $d$-dimensional subspace $L = L_{(X,\omega)}$, defined over a real number field, such that $[\omega] \in L \subset H^1(X, Z(\omega); \mathbb{C})$. Then $N$ is locally defined by real linear equations in period coordinates, and $\dim(M) = d$. In particular, as remarked above, it is $\Sl_2(\mathbb{R})$ invariant.

\end{itemize}
For a discussion about further special subvarieties known as the Eierlegende Wollmilchsau and Ornithorynque, we refer to \cite[Sec. 7]{zbMATH06436256}
 
\begin{rmk} There are higher dimensional examples of invariant subvarieties recently constructed by Eskin, McMullen, Mukamel, and Wright \cite{zbMATH06731862, zbMATH07268737}. We refer also Goujard's survey for the Bourbaki seminar \cite{Goujard2021} for an overview of such beautiful results.
\end{rmk}

\subsection{Totally geodesic subvarieties}\label{kdefsection}

\begin{thm}[{\cite[Thm. 1.2]{zbMATH07225031}}]\label{wrightthm}
There are only finitely many (maximal) totally geodesic subvarieties (in the sense of \Cref{kobgeod}) of $\mathcal{M}_{g,n}$ of dimension greater than $1$.
\end{thm}
We recall how the above theorem is proven, and the relationship between totally geodesic subvarieties and orbit closures, simply following the end of \cite[\S 2]{zbMATH07225031}. By Theorem 1.3 in \emph{op. cit.}, if $N$ is a totally geodesic subvariety of dimension at least 2, then $\Omega N$ has rank at least 2. Here we denote by $\Omega N$ the locus of square roots of quadratic differentials in the largest dimensional stratum of the cotangent bundle to $N$.
Since $\Omega N$ determines $N$, and there are a finite list of strata that may contain $\Omega N$ for $N$ a totally geodesic subvariety of $\mathcal{M}_{g,n}$, the result in turn follows from \Cref{thmEFWoriginal}.

Another interesting result is the following, which shows the difference between the symmetric and the \kob metrics on a Hilbert modular variety $\mathcal{H}_K$.
\begin{thm}[McMullen {\cite[Thm. 1.9]{mcmullen2023triangle}}]\label{mcmullenexa}
There exists a compact geodesic curve $f:V \to \mathcal{H}_K$ with $\dim \mathcal{H}_K =6$ such that $f(V)$ is not contained in any proper Shimura subvariety $S\subset \mathcal{H}_K$
\end{thm}

\subsection{Irreducible components}\label{irrcomp}

The problem of classifying irreducible components is also hard, as we saw in \Cref{genus2mc}. For completeness, we recall the main result of Kontsevich and Zorich \cite[Thms. 1,2 ]{zbMATH02001031}: If $g\geq 4$
\begin{itemize}
\item $\Og (2g-2)$ and $\Omega \mathcal{M}_{2k+1}(2k,2k)$ each have three connected components: the hyperelliptic one, and two more (distinguished by even/odd spin structures).
\item $\Omega \mathcal{M}_{2k}(2k-1,2k-1)$ has two connected components, a hyperelliptic and a non-hyperelliptic one.
\item When each entry of $\kappa$ is divisible by $2$ hen there are two connected components (distinguished by spin structures).
\item All other strata are connected.
\end{itemize}  
 In the remaining (low) genera, we have:
\begin{itemize}
\item $g=3$: $\Og (2,2)$ and $\Og (4)$ each have two components, the hyperelliptic and the odd spin one. The remaining strata are connected.
\item $g=2$: there are two connected strata $\Og (1,1)$ and $\Og (2)$.
\end{itemize}

\subsection{The work of Apisa and Wright}
Mirzakhani conjectured that the only invariant subvarieties of
$\Og$ of rank 2 or more are those coming from strata
of 1-forms or quadratic differentials. Some rank 2 counterexamples to the conjecture were constructed, but it follows from \Cref{thmEFWoriginal} that there can be only finitely many
counterexamples to the conjecture in any fixed stratum. Even better:

\begin{thm}[{\cite[Thm. 1.1]{zbMATH07734889}}]
Let $\mathcal{M}\subset \Og$ be an invariant subvariety with $\operatorname{rank}(\mathcal{M}) \geq \frac{g}{2}+1$. Then $\mathcal{M}$ is either a connected component of a stratum, or the locus of all holonomy double covers of surfaces in a stratum of quadratic differentials.
\end{thm}

\newpage
\part{Bi-algebraic viewpoint and Zilber-Pink conjecture}\label{part3}

\section{Zilber-Pink Conjecture and Some Predictions}

The André-Oort conjecture for $\mathcal{A}_g$ we discussed in \Cref{sectionao} is a special case of a more general, widely open conjecture about \emph{atypical intersections}. The conjecture has various levels of generality, ranging from semi-abelian varieties to the case of variations of mixed Hodge structures. Here, we focus on the case of the moduli space of principally polarized abelian varieties and its subvarieties.

\subsection{The Conjecture}\label{zpconj}
In 1999, Bombieri, Masser, and Zannier \cite{zbMATH01402564, zbMATH06022187} began studying other types of atypical intersections—e.g., curves against algebraic subgroups of multiplicative groups. Independently, Zilber \cite{zbMATH01818752} explored similar questions in the setting of exponential sum equations and the Schanuel conjecture. Pink \cite{zbMATH05013101}, motivated by unifying the Mordell-Lang and André-Oort conjectures, proposed his own version of the conjecture. A further important step came with the advent of the Pila-Zannier strategy for Manin-Mumford (see the recent book \cite{zbMATH07516542}).

Let $\VV$ be a family of principally polarized abelian varieties on an irreducible smooth quasi-projective variety $S$. We identify $\VV$ with the associated polarized variation of Hodge structures (VHS) of weight 1 obtained by interpolating the various $H^1(-,\Z)$.

At this point, it is useful to formally introduce the \textbf{Hodge locus} of $(S,\VV)$:
\begin{equation}
\HL (S,\VV^\otimes) := \{s \in S(\C) : \MT(\VV_s) \subsetneq \MT(\VV)\},
\end{equation}
where $\MT(\VV)$ denotes the Mumford-Tate group at a very general point of $S$, fixed and denoted by $0$. Some of the terminology from Hodge theory was already introduced in \Cref{NLsection} and \Cref{sectionao}; we refer to the appendix for a quick but self-contained introduction to Shimura varieties and their sub-Shimura varieties.

Informally, the Hodge locus parametrizes the closed points of $S$ corresponding to Hodge structures with additional \emph{Hodge symmetries}. It is easy to observe that such a Hodge locus is a countable union of subvarieties of $S$, known as \textbf{special subvarieties} for\footnote{The same results for arbitrary pure VHS is a deep result of Cattani, Deligne, and Kaplan \cite{CDK}. The difficulty lies in the fact that the period domain associated to $\VV$ is often non-algebraic. For a quick overview of the ZP philosophy in this this more general setting, as well as a discussion on functional transcendence, we refer also to \cite{2025arXiv250203071B}.} $(S, \VV)$. The datum $\VV$ is encoded in the associated period map 
\begin{equation}
\psi: S \to \mathcal{A}_g, \ \ \ s \mapsto [\VV_s].
\end{equation}
Here $\mathcal{A}_g$ is considered as the period domain of Hodge structures of weight $1$,
which is naturally a homogeneous space under $G$, the real points of $\MT(\VV)$. In particular, $\HL(S,\VV^\otimes)$ can be interpreted as the preimage of sub-Shimura varieties of $\mathcal{A}_g$, i.e., subvarieties that are maximal among those with a given generic Mumford-Tate group (cf. \Cref{defspec}). In short, special subvarieties $Y$ are components of $\psi^{-1}(\Gamma_Y \backslash D_Y)$, and, in this writing, we always assume that $\Gamma_Y \backslash D_Y$ is the smallest sub-Shimura variety containing $\psi(Y)$. (If $S$ is fixed, it is more natural to factorize $\psi$ through $\Gamma_S \backslash D_S\subset \Ag$.)

The link between Part I and Part II of the text appears in the following open-ended question, which is part of a more general philosophy on the study of families of motives:
\begin{question}
Can the Hodge locus of $\mathcal{M}_g$ (with respect to the universal family of curves) be explicitly described?
\end{question}
Some special examples were already discussed in \Cref{secrm} and \Cref{seccm}. See also \cite[Sec. 5]{arXiv:1703.10377}, for a related discussion.

To quantitatively describe the behaviour of the Hodge locus one is naturally lead to the following definition:
\begin{defi}
A special subvariety $Y = \psi^{-1}(\Gamma_Y \backslash D_Y)^0$ is \textbf{atypical} if 
\[
\codim_{\psi (S)} (\psi (Y)) < \codim_{D_S} (D_Y),
\]
and \textbf{typical} otherwise.
\end{defi}
The atypicality condition can be rephrased as an excess of dimension of the intersection.
\begin{conj}[Zilber–Pink Conjecture]\label{zpstatement}
There are only finitely many maximal atypical subvarieties for $(S,\VV)$. In particular, the atypical Hodge locus is not Zariski-dense in $S$.
\end{conj}
The formulation we gave above might not look the easiest to digest. However, it is the one that easily generalizes to the context of more general VHS--discussed in \Cref{sectionvhs}.

\begin{rmk}
We show how the above implies the André–Oort conjecture (AO), which asserts that CM points accumulate only on sub-Shimura varieties. It is part of the general theory that a\footnote{By the existence of one CM point, it is easy to deduce their density (even for the analytic topology) by applying Hecke translations.} CM point exists on any Shimura variety (this fact is sometimes attributed to D. Mumford). Let $S$ be a subvariety of $\mathcal{A}_g$. CM points of $S$ belong to $\HL(S, \VV^\otimes)$, as long as $\dim \psi(S) > 0$. They are atypical whenever $S$ is not a sub-Shimura variety. Therefore, \Cref{zpstatement} posits that the only possibility for the Zariski closure of a set of CM points is to accumulate on sub-Shimura varieties.
\end{rmk}

If $S = \mathcal{A}_g$, the Zilber–Pink conjecture trivially holds true, since there are no atypical special subvarieties (as alluded to in the above remark). In particular, it says nothing on $\Mg$, for $g=1,2,3$. For certain larger $g$s, we saw in \Cref{sec2} various examples of atypical special subvarieties-- determining whether or not they are maximal is often not that easy. See for example \Cref{ttv}, \Cref{dejongnoot}, \Cref{paul}, and \eqref{fermat}.

It can happen, for \emph{combinatorial reasons}, that every special component is atypical. If this is the case, the Zilber-Pink conjecture appears even more striking, predicting that the Hodge locus consists of finitely many maximal components. We record here one case (another, in a different direction, can be found in \cite[Thm. 1.5, Cor. 1.6]{2021arXiv210708838B}):

\begin{conj}\label{conjcurve}
Let $g \geq 3$, and let $C$ be a Hodge generic\footnote{I.e., it is not contained in any smaller Shimura subvariety (or equivalently, its algebraic monodromy is as large as possible, namely $\operatorname{Sp}_{2g}$). In the previous notation, this is to say that $\Gamma_S \backslash D_S=\Ag$.} curve in $\Ag$. For all but finitely many $x \in C(\C)$, $\MT(x) = \operatorname{GSp}_{2g}$; i.e., the collection of intersections between $C$ and the sub-Shimura varieties of $\Ag$ is a finite set.
\end{conj}

\begin{proof}[Proof of \Cref{conjcurve} assuming ZP]
Let $x \in C(\C)$ be a point with a Mumford-Tate group smaller than expected. This can either correspond to a typical or an atypical intersection. The former occurs, by definition (and the fact that $C$ is Hodge generic), only when 
\begin{displaymath}
\{x\} \subset C \cap \Gamma_x \backslash D_x
\end{displaymath}
for some sub-Shimura variety $\Gamma_Y \backslash D_Y \subset \Ag$ of codimension 1 (since $C$ is one-dimensional). The condition $g \geq 3$ guarantees that $\Ag$ contains no codimension-one sub-Shimura variety, and therefore the claim follows from \Cref{zpstatement}.
\end{proof}

One can actually describe \emph{quantitatively} the Hodge locus, not only its atypical part. (In fact, one can prove that \Cref{zpstatement} implies \Cref{conj-typical}, cf. \cite{2021arXiv210708838B, 2023arXiv230316179K}.)

\begin{conj}[Density of the Typical Hodge Locus] \label{conj-typical}
The following are equivalent: 
\begin{enumerate}
\item $\HL(S, \VV^\otimes)_\typ$ is non-empty; 
\item $\HL(S, \VV^\otimes)_\typ$ is dense in $S(\C)$; 
\item There exists a Hermitian symmetric subspace $D' \subset D_S$ such that $\dim S - \codim_{D_S} D' \geq 0$.
\end{enumerate}
\end{conj}

\begin{rmk}\label{rmktypical}
Interestingly, the first works towards \Cref{conj-typical} were inspired by the study of Jacobians. See the works of Colombo-Pirola \cite{colombopriola} and Izadi \cite{izadi}, who were motivated by the following question: Given an integer $k$ between $1$ and $g-1$, is the set of genus $g$ curves whose Jacobian contains a $k$-dimensional abelian subvariety dense in $\Mg$? Later, Chai \cite{chai} provided a version of this result for arbitrary Shimura varieties. The situation is now understood in the broader context of variational Hodge theory, cf. \cite{2021arXiv210708838B} (for the pure case) and \cite{2024arXiv240616628B} (for the mixed case). We will expand on this in \Cref{tyin}. Finally, we remark here that in the 80s McMullen noticed that the set of Riemann surfaces admitting a map to an elliptic curve
are dense in $\Mg$; see also \cite{zbMATH04157417} for related techniques. (This fact also follows from the density of square tiled surfaces we reviewed earlier. in \Cref{part2})
\end{rmk}

\subsubsection{First Applications}
A $g$-dimensional abelian variety $A$ is said to be a \emph{Weyl CM} abelian variety if $\End^0(A)$ is a field of degree $2g$ over $\Q$ whose Galois closure has degree $2g \cdot g!$ over $\Q$. It can be shown that, in a suitable sense, most CM abelian varieties are of this type.
\begin{thm}[Chai-Oort,Tsimerman {\cite{zbMATH06074022}}]
For every $g > 3$, the number of Weyl CM points in the Torelli locus is finite.
\end{thm}
\begin{proof}[Sketch of the Proof]
Suppose there are infinitely many genus $g$ Weyl CM Jacobians. The Zariski closure of this set contains an irreducible positive-dimensional subvariety. By AO, this is a special subvariety. The Weyl condition implies that this is a Hilbert modular variety (of dimension $g$), but we have seen before (in \Cref{thmmoonenetal}) that this can occur only for $g \leq 3$, cf. \cite{zbMATH05263283, zbMATH06296347}.
\end{proof}

\subsection{Predictions of the Zilber-Pink conjecture}\label{predsection}

Van der Geer and Oort \cite[Sec. 5]{zbMATH01445240} actually expect excess intersection of the Torelli locus and certain special subvarieties of $\Ag$: ``\emph{\dots one expects excess intersection of the Torelli locus and the loci corresponding to abelian varieties with very large endomorphism rings; that is, one expects that they intersect much more than their dimensions suggest.''}. 

In \Cref{secrm}, we mentioned various results supporting this belief. However these examples where usually related to the field $\Q(\zeta_p +\zeta_p^{-1})$. What would actually happen for a \emph{generic} totally real number field?

Let $K$ be a totally real number field of degree $g$. Consider the set
\begin{displaymath}
E_{g,K}:=\{x\in \mathcal{M}_g : \End(Jx)^0\cong K\}.
\end{displaymath}

\begin{conj}\label{conjendo}
Fix $g\geq 7$, each $E_{g,K}$ is contained in a finite union of Hilbert modular varieties. Moreover $E_K$ is empty for all but finitely many $K$.  
\end{conj}

\begin{conj}\label{conjendosmall}Let $g \leq 6$. For each $K$, $E_{g,K}$ is dense in $\mathcal{M}_g(\C)$.\end{conj}

For $g=1,2,3$ the above conjecture trivially holds true.

\begin{rmk}
For a related discussion regarding the case of quaternionic multiplication (once more an example of large endomorphism ring), we refer to \cite[Rmk. 7.3.4]{arXiv:2303.00804}.
\end{rmk}

\begin{thm}
Zilber-Pink (i.e. \Cref{zpstatement}) implies \Cref{conjendo} and \Cref{conjendosmall}.
\end{thm}
As alluded to above, the set $E_{g,K}$ is related to the geometric problem of understanding the intersections between $\Mg$ and Hecke translates of Hilbert modular varieties (cf. \Cref{examples}). It is indeed the pull-back to $\Mg$ of the set
\begin{displaymath}
\{ a \in \Ag:  \End(a)^0=K\}
\end{displaymath}
which is contained in a infinite union of Hilbert modular varieties. (This set is Zariski dense in $\mathcal{A}_g$.) Roughly, there is indeed a $\mathcal{H}_K:=\operatorname{SL}_2(\mathcal{O}_K)\backslash \mathbb{H}^g$ totally geodesic embedded in $\Ag$, corresponding to the abelian varieties with $\End$ (not $\End^0$) containing $\mathcal{O}_K$, and then one applies Hecke correspondences to cover all possible orders in $K$.

In particular we are trying to intersect $\Mg$, a $3g-3$ dimensional variety, with a $g$ dimensional one (the Hilbert modular varieties associated to $K$), inside the ambient variety $\Ag$ of dimension $g(g+1)/2$. As usual, we write $j: \Mg\to \Ag$.
\begin{proof}[ZP implies \Cref{conjendo} (large $g$)]
If $E_{g,K}$ is empty, there is nothing to prove. The first claim is that each $E_{g,K}$ is not Zariski dense in $\Mg$, i.e. finitely many Hecke translates suffices (rather than countably many). Consider $j^{-1}(\mathcal{H}_K)$ and its Hecke translates. Since $g\geq 7$, each component of such preimage is an atypical special subvariety. Moreover, to be in $E_{g,K}$, we are looking at points that lie in $j^{-1}(\mathcal{H}_K)$ but not on any smaller sub-Shimura variety. In particular they are automatically maximal, i.e. the only sub-Shimura variety of $\Ag$ containing them is one of the $\mathcal{H}_K$. \Cref{zpstatement} implies directly the claim.

The latter claim is that $E_K$ is empty for all but finitely many (isomorphism class of degree $g$ totally real) fields $K$. This follows once more from \Cref{zpstatement}, maximality, and the fact that $\mathcal{H}_K$ truly varies with $K$: If $E_{K_i}\neq \emptyset$ for a sequence $(K_i)_{i\in \N}$, then ZP implies that there exists finitely a sub-Shimura variety containing the $\bigcup E_{K_i}$.
\end{proof}
\begin{proof}[ZP implies \Cref{conjendosmall} (small $g$)]
For $g\leq 6$, we are looking at typical intersections, as $3g-3 + g \geq g(g+1)/2$. The claim is now just a reformulation of \Cref{conj-typical}, which is known to follow from ZP. (More details on a similar argument will be provided during the proof of \Cref{thmmumf}).
\end{proof}

\begin{prop}
ZP holds true for the one-dimensional family of curves described in \Cref{ttv}.
\end{prop}
\begin{proof}
This follows from André-Oort since here if the Mumford-Tate group drops at a point $p$, then $p$ has to be a CM point.
\end{proof}

\subsection{A question of Serre and Gross}\label{questionserregross}
In \cite{zbMATH03271259} Mumford shows the existence of
principally polarized abelian varieties $X$
of dimension $4$ having trivial endomorphism ring, that are not Hodge
generic in $\mathcal{A}_4$: they have an exceptional Hodge class in
$H^4(A^2, \Z)$. Serre asks if one can describe ``as
explicitly as possible'' such abelian varieties \emph{of Mumford's
  type}, cf. also
\cite[Proble 1]{zbMATH01587304}.

\begin{thm}[{\cite[Thm. 3.17]{2021arXiv210708838B}}]\label{thmmumf}
There are infinitely many smooth projective curves $X/\overline{\Q}$ of genus 4 whose Jacobians have a Mumford-Tate group isogenous to a $\Q$-form of the complex group $\mathbb{G}_m \times \Sl_2^3$. 
\end{thm} 
(For a discussion on a related question of Oort see \cite[Rmk 11.6]{2021arXiv210708838B} and \cite{2024arXiv240306217L}).

\begin{proof}[Sketch of the proof]
Let $\mathcal{M}_4$ be the moduli space of curves of genus $4$, and $j : \mathcal{M}_4 \hookrightarrow \mathcal{A}_4$
be the Torelli morphism. Denote the image of $j$ by $
\mathcal{T}_4^0=j (\mathcal{M}_4)\subset \mathcal{A}_4$, and by
$\mathcal{T}_4$ its Zariski closure. It is well
known that $\mathcal{T}_4$ is Hodge generic in $\mathcal{A}_4$ and that $
1=10-9= \codim_{\mathcal{A}_4} (\mathcal{T}_4).$

Recall that $\mathcal{A}_4$ contains a special curve $Y$ whose generic
Mumford-Tate group $\mathbf{H}$ is isogenous to a $\Q$-form of $\Gm \times
(\Sl_2)^3/\C$, as proven by Mumford \cite{zbMATH03271259}. It's not hard to see that 
 $\mathcal{T}_4$ cuts many Hecke
translates $Y_n$ of $Y$, that is: the union, varying $n$, of $\mathcal{T}_4 \cap Y_n$ is dense in $\mathcal{T}_4 $. In particular, upon extracting a sub-sequence of $Y_n$, we have
\begin{equation}
Y_n \cap \mathcal{T}_4^0 \neq \emptyset,
\end{equation}
since it is not possible that all intersections happen on $\mathcal{T}_4 - \mathcal{T}_4^0 $.

Since the $Y_n$ are one dimensional, for each $P \in Y_n $, there are only two possibilities:
\begin{itemize}
\item $P$ is a special point, i.e. $\MT(P)$ is a torus (so $P$ corresponds to a 4-fold with CM); 
\item $\MT(P)= \MT(Y_n)$, which is isogenous to some $\Q$-form of $\Gm \times (\Sl_2)^3$.
\end{itemize}
Therefore, to conclude, we have to find a $n$ and a non-special $P
\in Y_n \cap \mathcal{T}_4^0$. Heading for a contradiction, suppose
that, for all $n$, all points of $Y_n \cap \mathcal{T}_4^0$ are
special. By density of of the intersections, 
this means that $\mathcal{T}_4$ contains a dense
set of special points. Andr\'{e}--Oort now implies that
$\mathcal{T}_4$ is special, which is the contradiction we were looking 
for. 

\end{proof}

\subsection{Further results and remarks}
\subsubsection{G-function method}
There are a recent series of works, initiated by ideas of C.Daw and M.Orr, on height bounds on atypcal points, following a strategy first envisioned by Y. André, using what is now referred to as the \emph{G-functions method}. See for example \cite{2023arXiv230613463D, 2022arXiv220111240P, 2023arXiv230101857U} and the notes of André \cite{andre:hal-04799970}.

We just mention one result that fits well in the narrative:

\begin{thm}[Papas {\cite[Thm. 1.1]{2022arXiv220111240P}}]
Let $C$ be a smooth irreducible curve in $\Mg$ defined over $\Qbar$ whose image in $\Ag$ is Hodge generic. Assume that $g$ is odd with $g\geq 3$ and that the curve intersects the 0-dimensional stratum of the boundary of the Baily-Borel compactification of $\Ag$.

Then the set of $t \in C(\C)$ whose corresponding Jacobian, $\Jac t$,  is non-simple is finite.
\end{thm}
Perhaps these kind of results can eventually be applied to \teic curves (even though they are not Hodge generic, as we explained early on), to obtain interesting arithmetic consequences.

\subsubsection{Other interesting subvarieties}

On a different direction, a feature of the Zilber-Pink philosophy is that it applies to any subvariety of $\Ag$ and, often, the results described above use only the fact that $\Mg$ is Hodge generic and has dimension $3g-3$. This might seem a deficit of the theory, but at the same time one can investigate interesting subvarieties of $\Mg$ with the same techniques. To elucidate this point, we present one example that doesn't seem to be already in the literature.

Let $d\geq 3$, and $\mathcal{U}_d$ be the moduli space of plane curves of degree $d$. By this we mean the quotient of the open locus of smooth curves
in the Hilbert scheme of plane curves of degree $d$ by the action of $\operatorname{Aut}(\mathbb{P}^2)=\operatorname{PGL}_3(\C)$. It has dimension $\frac{(d+2)(d+1)}{2}-9$ and comes with an immersive period map into $\Mg$, for $g=\frac{(d-1)(d-2)}{2}$. Since the algebraic monodromy of the VHS is the $\Q$-algebraic group $\operatorname{Sp}_{2g}$, then $\mathcal{U}_d$ is Hodge generic in $\Ag$.

\begin{prop}
The set of plane curves of degree $d $ whose Jacobian contains an elliptic curve is dense in $\mathcal{U}_d(\C)$ (with respect to the euclidean topology).
\end{prop}

\begin{proof}
It is enough to check that $\dim \mathcal{U}_d > \codim_{\Ag} (\mathcal{A}_{g-1}\times \mathcal{A}_1)=g-1$., explicitly:
\begin{displaymath}
\frac{(d+2)(d+1)}{2}-9 \geq \frac{(d-1)(d-2)}{2} -1.
\end{displaymath}
The result then follows from \Cref{dense}, which is a weaker but unconditional version of \Cref{conj-typical}.
\end{proof}

\subsection{Digression on degree $d$ surfaces}

A simple, yet rich, scenario for applications of the Zilber-Pink philosophy (once extended to arbitrary VHS) is the study of the so called Noether-Lefschetz locus, which we briefly mentioned for $\Mg$ in \Cref{NLsection}. Let $U_{2,d}\subset \mathbb{P}(H^0(\mathcal{O}(d)))$ be the parameter space of smooth degree $d$ surfaces in $\mathbb{P}^3_{\C}$. Define
\begin{displaymath}
\NL_d:=\{[X]\in U_{2,d} : \operatorname{Pic}(\mathbb{P}^3_{\C})\to
\operatorname{Pic}(X) \text{  is not an isomorphism} \}. 
\end{displaymath}
If $d\geq 4$, it is a countable union of (strict) subvarieties of $U_{2,d}$.
For a surface $X$ outside $ \NL_d$, every curve on $X$ has the pleasant and useful property that it is the complete intersection of $X$
with another surface in $\mathbb{P}^{3}$. This object has been the subject of many beautiful studies by Griffiths, Green, Voisin, Ciliberto, Harris, and Miranda, and many others. We refer to the introduction of \cite{2023arXiv231211246B} for the link between ZP and the distribution of the components of $\NL_d$, and to Theorem 20 in \emph{op. cit.} for a discussion of a more general setting.

In this setting, it is worth stating the analogue of \Cref{conjcurve} which again follows from the most general ZP, cf. \cite[Conj. 3]{2023arXiv231211246B}:
\begin{conj}
Suppose that $\mathcal{Y}\to \mathbb{P}^1$ is a Lefschetz pencil of
degree $d$ surfaces in $\mathbb{P}^3$. If $d\geq 5$ the Picard number
of $\mathcal{Y}_s$ is greater or equal to $2$ for at most a finite
number of values of $s\in \mathbb{P}^1 (\C)$.  
\end{conj}

\section{Functional transcendence}
In this section we review the class of weakly special subvarieties and explain an important tool that will be used in the later sections, namely the Ax-Schanuel theorem(s). This is the key result of functional transcendence and it reduces the ZP conjecture to a Galois problem, as demonstrated in \cite{MR3867286}. 
\subsection{Weakly special subvarieties of Shimura varieties}\label{sec71}
Sub-Shimura varieties of $\Ag$ (and their Hecke translates) are part of a more general class of subvarieties called \emph{weakly special}. We now review some of their properties, cf. \cite{zbMATH01353476, MR2825237}, or \cite{2024arXiv240616628B} for the most general setting. A pair $(S,\V)$ (with a smooth base point $0\in S$) as in \Cref{zpconj} defines a local system, which corresponds to a representation of $\pi_1(S,0)$ into $\Gl_n (\Q)$. We denote by $\mathbf{H}_S$ its \emph{algebraic monodromy}, that is the neutral component of the Zariski closure of the image of  $\pi_1(S,0)$ in $\Gl_n /\Q$. It is a $\Q$-algebraic group. (If $S$ is not smooth one wither computes the monodromy of its smooth locus, or of its normalization.)

\begin{defi}\label{weakspe}
The weakly special subvarieties of $(S,\VV)$ are the closed irreducible algebraic subvarieties $Y\subset S$ maximal among the closed irreducible algebraic subvarieties $Z$ of $S$ whose algebraic monodromy group $\mathbf{H}_Z$ (with respect to $\VV_{| Z}$) equals, possibly up to conjugation due to the change of base point, $\mathbf{H}_Y$. 
\end{defi}

We write $\pi: \mathbb{H}_g\to \Ag (\C)= \operatorname{Sp}_{2g}(\Z) \backslash  \mathbb{H}_g$ for the uniformization map. Let $\mathcal{L}=\mathcal{L}(V_\Q , \varphi)$ denote the symplectic Grassmanian of Lagrangian, i.e. the scheme parametrizing maximal isotropic subspaces of $V_\Q$ with respect to the form $\varphi$. The map 
\begin{equation}\label{eqbor}
B: \mathbb{H}_g\to \mathcal{L}(\C)
\end{equation}
that sends a Hodge structure $y\in \mathbb{H}_g$ to the corresponding Hodge filtration $\operatorname{Fil}^0 \subset V_\C$ is the open immersion (the Borel embedding) of the $g$-dimensional Siegel space into its compact dual (after denoted by ${\mathbb{H}_g}^{\vee}$).

\begin{defi}
An irreducible subvariety $Y \subset \Ag$ is \emph{bi-algebraic} if some (equivalently any) analytic component of $\pi^{-1}(Y)$ is algebraic in the following sense: there exists an algebraic subvariety $Y^{\vee} \subset \mathcal{L}$ such that $Y=B^{-1}(Y^{\vee})$.
\end{defi}

\begin{thm}[Moonen, Ullmo-Yafaev] \label{thmweaksp} Let $Y \subset \Ag$ be an irreducible subvariety. The following are equivalent:
\begin{itemize}
\item $Y$ is weakly special;
\item $Y$ is bi-algebraic;
\item the smooth locus of $Y$ is totally geodesic with respect to the Bergamm metric (equivalently $Y$ is the image of a totally geodesic Hermitian subdomain $D_Y \subset \mathbb{H}_g$);
\item There is a sub-Shimura datum $ (H,D_H)$ of the Shimura datum $(\operatorname{GSp}_{2g},\mathbb{H}_g)$ such that the adjoint Shimura datum splits as a product: $(H^{\operatorname{ad}},D_H^{\operatorname{ad}})=(H_1,D_1)\times (H_2,D_2$), and $Y$ is the image in $S$ of $D_1\times {\{x_2\}}$, for some $x_2 \in D_2$.
\end{itemize}
Furthermore, a weakly special subvariety is special (in the sense of \Cref{defspec}) if and only if it contains a CM point.
\end{thm}
The original sources \cite[Sec. 4]{zbMATH01353476} and \cite[Thm. 1.2]{MR2825237} contain the most details on the subject; see also \cite[Sec. 4]{MR3821177} for a nice overview. We just sketch one highlight: A key input in the theory is the following, which links the monodromy group to the Mumford-Tate:

\begin{thm}[Deligne-Andr\'{e} monodromy theorem]\label{fixedpart}
Let $s\in S$. We have:
\begin{itemize}
\item[] \emph{Normality}. The monodromy group at $s$, $\mathbf{H}_S$ is a normal subgroup of the derived subgroup of the generic Mumford-Tate $G^{\text{der}}$ of $\VV$;
\item[] \emph{Maximality}. Suppose $S$ contains a special point. Then the monodromy equals $G^{\text{der}}$.
\end{itemize}
\end{thm}


\begin{rmk}[Linearity]\label{rmklink}
In analogy with what we described in \Cref{linearity} (the linearity properties of orbit closures), we add some words on the Harish-Chandra and the Borel \eqref{eqbor} embeddings (following  \cite[Sec. 3.1]{MR2825237}). Fix a point $0\in \mathbb{H}_g$, and let $\mathfrak{p}^+\subset \Lie(\operatorname{Sp}_{2g})$
be the holomorphic tangent bundle at $0$ in $\mathbb{H}_g$. Thanks to the work of Harish-Chandra, $\mathbb{H}_g$ can be canonically realised as a bounded symmetric domain in  $\mathfrak{p}^+ \cong \mathbb{C}^{g(g+1)/2}$. This gives an equivalent notion of algebraic subvariety of $\mathbb{H}_g$ (independent of the base point). We will give a more detailed description of the Harish-Chandra realization in \Cref{secfingeod}.

In this description, the bi-algebraic subvarieties are in fact \emph{linear}: they are of the form $V\cap \mathbb{H}_g$, for $V\subset  \mathbb{C}^{g(g+1)/2}$ a complex vector subspace corresponding to the tangent space of a totally geodesic Hermitian subdomain. This is a well know fact of the theory, see for example \cite[Prop. 2.4]{zbMATH07629821}.
\end{rmk}

\subsection{Ax-Schanuel after Bl\'{a}zquez-Sanz, Casale, Freitag, and Nagloo}\label{sectionAS}
Thanks to \Cref{fixedpart} (or Deligne's semisimplicity theorem), we observe that the monodromy groups appearing here are always semisimple (the monodromy of a mixed VHS is a semidirect product of a unipotent group and a semisimple one). In particular, for this section will simply focus on such algebraic groups. 

Let $S$ be an (irreducible, complex) smooth variety, and $H$ a complex linear algebraic group which is semisimple. Let $\pi : P \to S$ be an $H$-\emph{principal} bundle, that is $H$ acts on $P$ (on the right, with the action denoted by $\cdot$) and such that the action induces an isomorphism
\begin{displaymath}
P \times H \simeq P \times_S P, \  (p,g )\mapsto (p,p\cdot g).
\end{displaymath}
In particular the fibers of $\pi$ are principal homogeneous spaces, and the choice of a point $p \in P_s:=\pi^{-1}(s)$ induces an isomorphism $H \cong P_s $, given by $g \mapsto p \cdot g$. We have the following exact sequence of $S$-bundles
\begin{equation}\label{firstdefver}
0 \to \textrm{ker} \, d \pi \to TP \to TS \times_{S} P \to 0.
\end{equation}
A \emph{connection} $\nabla$ is a section of \eqref{firstdefver}. We say that $\nabla$ is $H$-\emph{principal} if it is $H$-equivariant. We additionally say that $\nabla$ is \emph{flat} if it lifts to a morphism of Lie algebras: the equation $\nabla_{[v,w]} = [\nabla_{v}, \nabla_{w}]$ holds on the level of maps of $\mathbb{C}$-schemes. From now on all connections are $H$-principal and flat.

\begin{defi}\label{mininvdef}
An \emph{integral submanifold} of $P$ is a locally closed complex submanifold whose tangent subbundle lies entirely in the image of $\nabla$. A \emph{leaf} is a maximal integral submanifold. An irreducible subvariety $V \subset P$ is a \emph{minimal $\nabla$-invariant variety} if it is the Zariski closure of a $\nabla$-horizontal leaf. 
\end{defi}

Let $V$ be a minimal $\nabla$-invariant variety. Then the \emph{Galois group of $S$}
\begin{displaymath}
H_V:=\Gal(\nabla)=\Gal (S)=\{g\in H : V \cdot g =V\},
\end{displaymath}
is an algebraic subgroup of $H$ and $\pi_{|V}$ is an $H_V$-principal bundle (cf. \cite[Lem. 2.3]{2021arXiv210203384B}). We can assume, from now on, that $H=\Gal(\nabla)=\Gal(S)$. 

\begin{defi}\label{def:nablasp}
A Zariski closed irreducible subvariety $Y \subset S$ is \emph{$\nabla$-special} if its Galois group $\Gal(Y) = \Gal (\nabla_{| Y})$ is strictly contained in $\Gal(S)$, and it is maximal for this property.
\end{defi}

Let $\pi : P \to S$ with flat $H$-equivariant connection $\nabla$. The following is the main theorem of Bl\'{a}zquez-Sanz, Casale, Freitag, and Nagloo (see also \cite{2022arXiv220805182B} and references therein).

\begin{thm}[{\cite[Thm. 3.6]{2021arXiv210203384B}}]\label{thm:newAS}
Let $V$ be a subvariety of $P$, $x\in V$, and let $\mathcal{L} \subset P$ be a horizontal leaf through $x$. Let $U$ be an analytic irreducible component of $V \cap \mathcal{L}$.

 If $\dim V < \dim U + \dim H$,
then the projection of $U$ in $S$ is contained in a $\nabla$-special subvariety.
\end{thm}

\begin{rmk}\label{standardbundle}
In the theory of Shimura variety, there is an automorphic bundle, usually called \emph{the standard bundle}. See e.g.  \cite[Sec. 3]{zbMATH04120309}. Over a Shimura variety associated to $(G,X)$, it will give indeed an algebraic $G$-torsor naturally equipped with the desired data. We will denote it by $\pi: P \to \mathcal{A}_g$. Here $\nabla$-special subvarieties are exactly the weakly special ones, described by \Cref{thmweaksp}.
\end{rmk}

The above discussion implies the following, originally due to Mok, Pila and Tsimerman
\cite[Thm. 1.1]{as} in the case of Shimura varieties (Bakker and Tsimerman in general):

\begin{thm}[Ax-Schanuel for VHS]\label{astheorem}  
Let $W \subset S \times \mathbb{H}_g$ be an algebraic subvariety. Let $U$ be an
irreducible complex analytic component of $W \cap S \times_{\Ag} \mathbb{H}_g$
such that 
\begin{displaymath}
\codim_{S \times \mathbb{H}_g} U< \codim_{S \times \mathbb{H}_g} W+ \codim_{S
  \times \mathbb{H}_g} (S \times_{\Ag}
\mathbb{H}_g)\;\;.  
\end{displaymath}
Then the projection of $U$ to $S$ is contained in a strict weakly
special subvariety of $S$.
\end{thm}

\subsubsection{Schanuel's conjecture}

It is worth pausing for a second on why the above theorems are called Ax-Schanuel and how they link to transcendence theory.

\begin{conj}[Schanuel's conjecture] Given $z_1, \dots ,z_n\in \C$ that are linearly independent over $\Q$, the field extension
\begin{displaymath}
\Q (z_1,\dots , z_n, \exp{(z_1)},\dots, \exp{(z_n)})
\end{displaymath}
 has transcendence degree at least $n$ over $\Q$.  
\end{conj}
By replacing $\Q\subset \C$ by $\C \subset \C[[t_1,\dots, t_m]]$, one obtains a \emph{functional} version of the above, which is in fact a theorem of Ax:
\begin{thm}[Ax \cite{zbMATH03366984}]
Let $x_1,\dots, x_n \in \C[[t_1,\dots, t_m]]$ have no constant term and be linearly independent over $\Q$. Then $\operatorname{tr.deg.}_\C \C(x_1,\dots, x_n, e^{x_1},\dots , e^{x_n})$ is at least $ n +\text{rank} \left(\frac{\partial x_i}{\partial t_j}\right)$.  
\end{thm}
There is a \emph{Geometric} equivalent formulation, which we recall now:
\begin{cor}
 Let $W\subset \C^n \times (\C^*)^n$ be an irreducible algebraic subvariety. Let $U$ be an irreducible analytic component of $W \cap \Pi$, where $\Pi$ is the graph of the exponentiation map. Assume that the projection of $U$ to $(\C^*)^n$ is not contained in a translate of any proper algebraic subgroup. Then $\dim W = \dim U +n$.
\end{cor}

To see the equivalence between the two formulation, notice that:
\begin{itemize}
\item $\text{tr.deg.}_\C \C(x_1,\dots, x_n, e^{x_1},\dots , e^{x_n})$ is the dimension of the Zariski closure of $U$ in $\C^n \times (\C^*)^n$; 
\item $\text{rank} \left(\frac{\partial x_i}{\partial t_j}\right)$ is the (analytic) dimension of $U$; 
\item The $x_i$ have no constant terms and are linearly independent over $\Q$, then the projection of $U$ to $\C^n$ contains the origin and is not contained in a linear subspace defined over $\Q$. 
\end{itemize}

\subsection{Typical Intersections in Hodge theory}\label{tyin}
This is a brief continuation of the discussions appearing around \Cref{conj-typical} and \Cref{rmktypical}. We would like to explain how Hodge theory actually gives a simple combinatorial criterion to decide whether $\HL(S,\VV^\otimes)_{\typ}$ is empty or not. Denote by $ \Psi$ the period map associated to $(S,\VV)$. We follow the recent works of Eterovic-Scanlon, Khelifa-Urbanik \cite{2022arXiv221110592E, 2023arXiv230316179K}. They proved:

\begin{thm}[Eterovic-Scanlon, Khelifa-Urbanik]\label{dense}
If $(\mathbf{H}, D_H) \subset (\mathbf{G}, D_G)$ is a strict sub-Hodge datum such that $\dim \Psi (S^{\operatorname{an}})  + \dim D_M \geq \dim D$, then 
\begin{displaymath}
\HL(S, \V^{\otimes}, \mathbf{H}):=\{s\in S : \exists g \in \mathbf{G}(\Q)^+, \MT(\V_s)\subset g \mathbf{H} g^{-1}\}
\end{displaymath}
 is dense in $S^{\operatorname{an}}$.
\end{thm}
The full ZP conjecture would predict that the above $\subset $ can be replaced by an equality. This is the meaning of the density of typical Hodge locus and it is still open. 

\begin{proof}[Sketch of the proof]
 Denote by $\tilde{S}$ the universal cover of $S$, and $\tilde{s}\in \tilde{S}$ a Hodge generic point. Fix $(\mathbf{H}, D_H) \subset (\mathbf{G}, D_G)$ a Hodge sub-datum satisfying the dimension inequality and $g\in \mathbf{G}(\R)$ such that $\tilde{s}\in g \cdot D_H$. Consider $\mathcal{ U}=\tilde{S}\cap g \cdot D_H \subset D_G$, which contains $\tilde{s}$. Since $\tilde{s}$ is Hodge generic, Ax-Schanuel (\Cref{thm:newAS}) implies that $\mathcal{ U}$ has an analytic irreducible component of the expected dimension (at $\tilde{s}$). The same holds true for any $g '$ sufficiently close to $g$, and we conclude by density of the rational points of $\mathbf{G}$ in the real ones.
\end{proof}

\subsection{Ax-Schanuel in Families}
We now explain one of the links between Ax-Schanuel and the Zilber-Pink conjecture, more precisely with its geometric part.

 In Hodge theory, the special intersections in $S$ are particular cases of \emph{weakly special subvarieties}, i.e. the maximal subvarieties of $S$ with a given monodromy group. The plan is roughly as follows:
\begin{enumerate}
\item[I.] Relate the atypicality inequality to an analogous one appearing in an opportune period torsor $P \to S$. 
\item[II.] Over-parametrize by an algebraic parameter space $\mathcal{Y}$ all subvarieties of $P$ that can give rise to atypical intersections with flat leaves in $P$ above $S$;
\item[III.] Use an \emph{Ax-Schanuel Theorem in families} to show that, all atypical intersections with leaves parameterized by $\mathcal{Y}$ project into the fibres of finitely many families of weakly special subvarieties of $S$ (not necessarily atypical themselves); and
\item[IV.] Conclude by either using some rigidity or maximality property.
\end{enumerate}
To implement the above strategy, Hodge theory (essentially through the fixed part theorem, cf. \Cref{fixedpart}) is needed at two steps. It is used to guarantee, a priori, that the atypical intersections we are interested in can be over-parametrized by algebraic families (of subvarieties of $P$), as well as to say that the weakly special subvarieties one obtains after applying Ax-Schanuel fibre-by-fibre can be assembled in countably many families (of subvarieties of $S$).

Here we highlight the use of the Ax-Schanuel in $P$ rather than the Hodge-theoretic one associated to some period domain $D$. At first sight, this might look like a small and technical difference but, in reality, it makes the scope of the Zilber-Pink theory much broader. The reason is that the Hodge-theoretic Ax-Schanuel considers only atypical intersections inside $S \times D$, which usually has much smaller dimension than $P$. This means that there can be more atypical intersections in $P$, since there is more freedom to impose algebraic conditions on leaves $\mathcal{L} \subset P$ then on the graph of some period map. This difference is crucial for treating the case of orbit closures: the space $\Omega \mathcal{M}_g(\kappa)$ supports a $\Z$VMHS, but the atypicality of the orbit closure $\mathcal{N}$ can be detected only in an opportune automorphic bundle above the mixed Shimura variety associated to the $\Z$VMHS. Said in more plain terms: the conditions defining orbit closures include algebraic conditions on periods of the one-form $\omega$ which do not come from conditions on Hodge structures, but can be imposed inside $P$.

In practice, we need the following two results, that can be found in \cite[Sec. 3.3]{2024arXiv240616628B}, and follow from \Cref{thm:newAS}:
\label{axschanfamsec}
\begin{prop}
\label{nondenseA}
Let $(P, \nabla)$ be an algebraic $H$-principal bundle with connection on an algebraic variety $S$, and let $f : \mathcal{Z} \to \mathcal{Y}$ be a family of subvarieties of $P$. Then the locus
\[ \mathcal{Y}(f,e) := \{ (x, y)\in P \times \mathcal{Y} : \dim_{x} (\mathcal{L}_{x} \cap \mathcal{Z}_{y}) \geq e \} \]
is algebraically constructible.
\end{prop}

For a family of $ h : C \to B$ of subvarieties of $S$, we define:
\begin{align*}
\mathcal{Y}(f,e,h) := \left\{ (x, y) \in P \times \mathcal{Y} : \begin{array}{c} \textrm{ exists an embedded analytic germ } \\ (x, \mathcal{D}) \subset (x, \mathcal{Z}_{y} \cap \mathcal{L}_{x}) \textrm{ of dimension } \\ \textrm{ at least } e \textrm{ mapping into a fibre of }h\end{array} \right\}
\end{align*}

The following says that if one applies the Ax-Schanuel theorem to a family of subvarieties then the resulting intersections belong to a family of weakly special subvarieties. 

\begin{thm}
\label{protogeoZP}
Fix some family of subvarieties of $P$, $f : \mathcal{Z} \to \mathcal{Y}$, and an integer $e \geq 1$, such that each fibre of $f$ has dimension $< \dim H + e$.

 Suppose that there exists a countable collection of families $\{ h_{i} : C_{i} \to B_{i} \}_{i=1}^{\infty}$ of subvarieties of $S$, all of whose fibres are weakly special, and that all weakly special subvarieties of $S$ are among the fibres of these families. Then the set $\mathcal{Y}(f,e)$ introduced in \Cref{nondenseA} is contained in a finite union of $ \mathcal{Y}(f,e,h_{i})$.
\end{thm}

\subsection{The geometric part}

From \Cref{protogeoZP} and some various facts from Hodge theory, one can deduce of the so called \emph{geometric part of ZP} in its most general setting: admissible graded-polarizable integral VMHS. A similar argument will be discussed in more detail in \Cref{proofEFW}.

\begin{thm}[Geometric mixed ZP, {\cite[Thm. 7.1]{2024arXiv240616628B}}]\label{thm:mixedgeomzp}
There is a finite set $\Sigma= \Sigma_{(S,\V)}$ of triples $(\mathbf{H},D_H, \mathbf{N})$, where $(\mathbf{H},D_H)$ is some sub-Hodge datum of the generic Hodge datum $(\mathbf{G}_{S},D_{S})$, $\mathbf{N}$ is a normal subgroup of $\mathbf{H}$ whose reductive part is semisimple, and such that the following property holds.

For each monodromically atypical maximal (among all monodromically atypical subvarieties) $Y \subset S$ there is some $(\mathbf{H},D_H, \mathbf{N})\in \Sigma$ such that, up to the action of $\Gamma$, $D^{0}_{Y}$ is the image of $\mathbf{N}(\R)^+ \mathbf{N}(\C)^{u} \cdot y$, for some $y \in D_H$.
\end{thm}

The above theorem intermediately implies \Cref{geomao}



\section{Strata of abelian differentials}

We are now ready to go back to the setting of \teic geometry, as in \Cref{properties}, and try to use what we learned on functional transcendence in the process. A first comparison with the case of Shimura varieties was discussed in \Cref{rmklink}.
\subsection{The work of Klingler and Lerer}

The initial observation for the \textbf{bi-algebraic} perspective for $\Og (\kappa)$ and the developing map is that orbit closures are \textbf{linear} in period coordinates. One is therefore naturally lead to consider linear subvarieties and, more generally, the bi-algebraic ones, as first promoted in \cite{2022arXiv220206031K}.
\begin{defi}\label{bialgdef}
A subvariety $W\subset \Omega \mathcal{M}_g (\kappa)$ is \textbf{bi-algebraic} if it is algebraic in the period coordinates. It is $\Qbar$-bi-algebraic if it is moreover defined over $\Qbar$, for both structures.
\end{defi}

Klingler and Lerer proposed a general framework and made the following conjecture, inspired by the André--Oort conjecture (\Cref{sectionao}):

\begin{conj}[{\cite[Conj. 2.13]{2022arXiv220206031K}}]
Let $S \subset \Omega \mathcal{M}_g(\kappa)$ be an irreducible algebraic subvariety containing a Zariski-dense set of arithmetic points. Then $S$ is $\overline{\mathbb{Q}}$-bi-algebraic.
\end{conj}

They also prove:
\begin{thm}[{\cite[Thm. 2.6]{2022arXiv220206031K}}]
 Any invariant linear subvariety of $\Og (\kappa)$ of rank $1$ and degree at least $2$ does not contain a Zariski-dense set of arithmetic points. In particular any \teic curve of degree at least $2$ contains only finitely many arithmetic points.
\end{thm}
On the geometric properties of bi-algebraic subvarieties, they show:
\begin{thm}[{\cite[Thms. 2.8, 2.9]{2022arXiv220206031K}}]\label{thm0000}
The bi-algebraic curves in $\Og (\kappa)$ satisfying condition $(\star)$ (cf. Def. 6.5 in \emph{op. cit.}) are linear, as well as the ones contained in an isoperiodic leaf.

\end{thm}


We review some results related to the geometry of bi-algebraic subvarieties (the arithmetic properties are harder to investigate).

\subsection{Non-linear bi-algebraic subvarieties, after Deroin and Matheus}
We briefly discuss \cite[Thm. 1.1, 1.3]{zbMATH07811831}, as well as some of the computations from \cite[Thm. 8.3]{zbMATH06149478}. The condition $(\star)$ from \Cref{thm0000} indicates that the families of curves whose Jacobians possess non-trivial fixed parts are potential sources of non-linear bi-algebraic subvarieties, and in fact:

\begin{thm}[{\cite[Thm. 1.1, 1.3]{zbMATH07811831}}]\label{nonline}
 There is a non-linear bi-algebraic curve, resp. surface, of abelian differentials of genus 7, resp. 10. They are given by 
 
\begin{equation}
C_t =\{ y^6=x(x-1)(x+1)(x-t)\},  \ \ \  \omega_t=\frac{x^2dx}{y^5};
\end{equation}
\begin{equation}
 C_{a,b,c}=\{y^6=x(x-1)(x-a)(x-b)(x-c)\}, \ \ \  \omega_{a,b,c}=\frac{dx}{y^5} 
\end{equation}
 for $(a,b,c)$ in a generic algebraic surface $S\subset \C^3$, ($a,b,c$ distinct and $\notin \{0,1\}$).
\end{thm}

\begin{question}
What is the monodromy/generic Mumford-Tate group of the above families (e.g. what is the weakly special closure of $ C_{a,b,c}$)? Can the set of CM points of those families be described?
\end{question}

\subsection{Orbit closures as (a)typical intersections}\label{diagramatyp}

In this section we give a geometric reinterpretation of the conditions appearing in \Cref{mainthmalgfilip}, following \cite[Sec. 6]{2024arXiv240616628B} (which is inspired by the work of Filip). Unsurprisingly, we study orbit closures via various period maps. For every $n$, let $\mathcal{A}_{g,n} \to \mathcal{A}_g$ denote the fibration whose fibre  over a point $[A]\in \mathcal{A}_g$ is $\operatorname{Sym}^{[n]}A$ (the unordered $n$-tuples of points of $A$). The varieties $\mathcal{A}_{g}$ and $\mathcal{A}_{g,n}$ can be interpreted as mixed Shimura varieties, cf. \Cref{mixedshim}. For $\mathcal{A}_{g}$ this is the standard Siegel space description. In general one has a natural bijection 
\begin{displaymath}
\mathcal{A}_{g,n} \simeq \{ (H, E) : H \in \Sp_{2g}(\mathbb{Z}) \backslash \mathbb{H}_{g}, \hspace{0.5em} E \in \textrm{Ext}^{1}_{\textrm{MHS}}(H, \mathbb{Z}^{n})) \} ,
\end{displaymath}
where we view $\Sp_{2g}(\mathbb{Z}) \backslash \mathbb{H}_{g} \simeq \mathcal{A}_{g}$ as the moduli space for principally polarized pure weight one $\mathbb{Z}$-Hodge structures, and we view $\mathbb{Z}^{n}$ as a pure weight zero Hodge structure; see \cite[4.3.14]{filipnotes}. One obtains an algebraic mixed period map $\varphi : \Omega \mathcal{M}_{g}(\kappa) \to \mathcal{A}_{g,n}$ which sends $(X, \omega)$ to the isomorphism class of the extension. The data of $\varphi$ is equivalent to the data of the mixed $\mathbb{Z}$VHS $H^{1}_{\textrm{rel}}$. Note that $\varphi$ is quasi-finite, cf. the discussion at the end of \cite[\S 3.4]{2022arXiv220206031K}. We will typically work with the factorization
\begin{displaymath}
\Omega \mathcal{M}_g(\kappa) \xrightarrow{\Omega \varphi} \Omega \mathcal{A}_{g, n} \to  \mathcal{A}_{g, n},
\end{displaymath}
of $\varphi$, where the first arrow is obtained by seeing the differential $\omega$ on the Jacobian of $X$.

We think of an orbit closure $\mathcal{N}$ (with associated invariants $r,d,t$) as an irreducible component of $(\Omega \varphi)^{-1}(E[\mathcal{N}])$, where $E[\mathcal{N}] \subset \Omega \mathcal{A}_{g,n}$ is an algebraic subvariety constructed using the conditions appearing in \Cref{mainthmalgfilip}. In accordance with the structure of \Cref{mainthmalgfilip}, the variety $E[\mathcal{N}]$ is constructed in three stages, depicted in the following diagram:

\begin{equation}
\label{Filipdiagram}
\begin{tikzcd}
  \mathcal{N} \arrow[r] \arrow[d, hookrightarrow]
  &   E[\mathcal{N}] \arrow[r,  two heads] \arrow[d, hookrightarrow]
   &    M[\mathcal{N}]  \arrow[r,  two heads] \arrow[d, hookrightarrow]
    & S[\mathcal{N}] \arrow[d, hookrightarrow] \\    
  \Omega \mathcal{M}_g(\kappa ) \arrow[r]
&\Omega \mathcal{A}_{g,n} \arrow[r,  two heads]
& \mathcal{A}_{g,n} \arrow[r,  two heads]
&\mathcal{A}_{g}
 \end{tikzcd}:
 \end{equation}
 
\begin{enumerate}
\item Let $S[\mathcal{N} ] \subset \mathcal{A}_g$ be the smallest weakly special subvariety of $\mathcal{A}_g$ containing the image of $\mathcal{N}$. It lies in the locus $R[\mathcal{N}]$ of $g$-dimensional principally polarized abelian varieties that have real multiplication \emph{of the same type of $\mathcal{N}$} (in the sense of \Cref{mainthmalgfilip}). The latter is a special subvariety of $\mathcal{A}_{g}$ whose generic Mumford-Tate group is isomorphic to $\left(\operatorname{Res}_{K/\Q} {\operatorname{GSp}_{2r,K}} \right)\times \operatorname{GSp}_{2 (g-dr)}$.
\item Let $M [\mathcal{N}] \subset \mathcal{A}_{g,n}$ be the smallest weakly special subvariety of $\mathcal{A}_{g,n}$ containing $\mathcal{N}$. It lies in $T[\mathcal{N}]$, the bundle over $R [\mathcal{N}]$ consisting of $n$-tuples of points satisfying the same torsion conditions as $\mathcal{N}$.
\item Let $E[\mathcal{N}]$ be the bundle of $1$-forms over $M[\mathcal{N}]$ inside in the $K$-eigenspace containing the restriction of the section $\omega$ to $\mathcal{N}$. 
\end{enumerate}
The notation $S[\mathcal{N}]$ is chosen as we think of the variety $S[\mathcal{N}]$ as giving the \emph{Shimura condition}. Similarly, $M[\mathcal{N}]$ gives the \emph{mixed-Shimura condition} (see also \Cref{mixedshim}), and $E[\mathcal{N}]$ includes also the \emph{eigenform} condition in addition to the previous two. 

The conditions $R[\mathcal{N}]$ and $T[\mathcal{N}]$ are simply geometric reinterpretations of the first two items in \Cref{mainthmalgfilip}, whereas the conditions $S[\mathcal{N}]$ and $M[\mathcal{N}]$ may in principle be stronger. The condition $E[\mathcal{N}]$ corresponds to the third condition in \Cref{mainthmalgfilip}.

\begin{defi}[Relative (a)typicality --- Intersection theoretic version]\label{defatyfili}
Let $\mathcal{M}$ be a orbit closure in some stratum $\Omega \mathcal{M}_g (\kappa)$ (possibly equal to component of the latter). Then an orbit closure $\mathcal{N} \subset \mathcal{M}$ is said \emph{(intersection theoretically) atypical} (relative to $\mathcal{M}$) if
\begin{displaymath}
\codim_{E[\mathcal{M}]} \mathcal{N} < \codim_{E[\mathcal{M}]} (E[\mathcal{N}]) +  \codim_{E[\mathcal{M}]} (\mathcal{M}),
\end{displaymath}
and \emph{(intersection theoretically) typical} (relative to $\mathcal{M}$) otherwise.
\end{defi}

We can now discuss a proof of \Cref{thmEFWoriginal} inspired by the philosophy of atypical intersections. More precisely, with the vocabulary of \Cref{defatyfili}:

\begin{thm}[{\cite[Thm. 6.5]{2024arXiv240616628B}}]
\label{orbitclosurefinthm}
Let $\mathcal{M} \subset \Omega \mathcal{M}_g$ be an orbit closure (possibly equal to an irreducible component of some stratum $\Omega \mathcal{M}_g (\kappa)$). Then $\mathcal{M}$ contains at most finitely many (strict) orbit closures $\mathcal{N}$ that are:
\begin{enumerate}
\item Maximal, i.e. the only suborbit closures of $\mathcal{M}$ containing $\mathcal{N}$ are $\mathcal{N}$ and  $\mathcal{M}$;
\item Atypical (relative to $\mathcal{M}$) in the sense of \Cref{defatyfili}.
\end{enumerate}
In particular in each (connected component of each) stratum $\Omega \mathcal{M}_g (\kappa)$, all but finitely many orbit closures have rank 1 and degree at most 2.
\end{thm}

For simplicity we treat only the case $\mathcal{M} = \Omega \mathcal{M}_g(\kappa)$ some fixed orbit closure. 

\subsubsection{Sketch of the proof of \Cref{orbitclosurefinthm}}\label{proofEFW}

We now let $(\mathcal{H}, \nabla)$ be the algebraic vector bundle with regular singular connection associated to $\mathbb{V} = H^{1}_{\textrm{rel}}$, and consider the bundle $P$ and map $\nu$ constructed above. The graded quotient of $\mathcal{H}$ which corresponds to $H^{1}_{\textrm{abs}}$ we denote $\mathcal{H}_{\textrm{abs}}$. Then $\omega$ is an algebraic section of $\mathcal{H}_{\textrm{abs}}$, and we then obtain any algebraic map $r : P \to \mathcal{H}_{abs, s_{0}}$ given by 
\[ [\eta \in \Hom(\mathcal{H}_{s}, \mathcal{H}_{s_{0}})] \mapsto \eta(\omega_{s}) . \]

\begin{lemma}
\label{mapfromPsurj}
The map $r$ lands inside a unique eigenspace $E_{\mathcal{M}} = E_{\mathcal{M},s_{0}} \subset \mathcal{H}_{\textrm{abs},s_{0}}$ for the field of real multiplication $K_{\mathcal{M}}$ associated to $\mathcal{M}$. The complex algebraic variety $(E_{\mathcal{M}} \setminus \{ 0 \}) \times \ch{D}^{0}$ consists of a single $\mathbf{H}(\mathbb{C})$-orbit, and the map $r \times \nu : P \to (E_{\mathcal{M}} \setminus \{ 0 \}) \times \ch{D}^{0}$ is $\mathbf{H}(\mathbb{C})$-equivariant and surjective.
\end{lemma}

\begin{proof}
Let $\mathcal{L} \subset P$ be a leaf coming from the rational structure of the underlying local system $\mathbb{V}$. Then from the third part of \Cref{mainthmalgfilip}, the image $r(\mathcal{L})$ lies in a $K_{\mathcal{M}}$-eigenspace $E_{\mathcal{M}} = E_{\mathcal{M},s_{0}}$. Because $\mathcal{L}$ is Zariski dense in $P$ as a consequence of the previous discussion on leaves, we have $r^{-1}(\overline{r(\mathcal{L})}^{\textrm{Zar}}) = P$ and hence $r(P) \subset E_{\mathcal{M}}$. 

It is clear by construction that the map is $\mathbf{H}(\mathbb{C})$-invariant, so surjectivity will follow if we can show that in fact $(E_{\mathcal{M}} \setminus \{ 0 \}) \times \ch{D}^{0}$ is an orbit of $\mathbf{H}(\mathbb{C})$. This is easy in the case $\mathcal{M}$ equals a stratum and more delicate in general.
\end{proof}

Given a suborbit closure $\mathcal{N} \subset \mathcal{M}$, we will write $Z_{\mathcal{N}} := E_{\mathcal{N}} \times \ch{D}^{0}_{\mathcal{N}}$ for the analogous subvariety of $E_{\mathcal{M}} \times \ch{D}^{0}$ associated to $\mathcal{N}$. It is well-defined up to the action of the monodromy group $\Gamma_{\mathcal{M}}$.

\begin{prop}
\label{ineqinsidebundle}
We have $\dim (E_{\mathcal{N}} \times \ch{D}^{0}_{\mathcal{N}}) = \dim E[\mathcal{N}]$. If $\mathcal{N}$ is an intersection-theoretically atypical suborbit closure of $\mathcal{M}$, then we have the following inequality of formal codimensions:
\[ \codim_{P} \mathcal{N} < \codim_{P} (r \times \nu)^{-1}(E_{\mathcal{N}} \times \ch{D}^{0}_{\mathcal{N}}) + \codim_{P} \mathcal{M} . \]
\end{prop}

\begin{proof}
The first equality of dimensions is formal and can be seen by translating the abelian variety data into its Hodge-theoretic incarnation. For the second we use the fact that $P \to ( E_{\mathcal{M}} \setminus \{ 0 \}) \times \ch{D}^{0}$ has constant fibre dimension to rewrite the inequality in \Cref{defatyfili}. The key fact is that
\begin{align*}
\codim_{P} (r \times \nu)^{-1}(E_{\mathcal{N}} \times \ch{D}^{0}_{\mathcal{N}}) &= \dim [E_{\mathcal{M}} \times \ch{D}^{0}_{\mathcal{M}}] -  \dim [E_{\mathcal{N}} \times \ch{D}^{0}_{\mathcal{N}}] \\
&= \dim E[\mathcal{M}] - \dim E[\mathcal{N}] .
\end{align*}
\end{proof}

\subsubsection{Over-parametrization}\label{sectionoverpar}

We describe a family of subvarieties of $P$ that (over) parametrizes all the data that can give rise to suborbit closures of $\mathcal{M}$, regardless whether such an orbit closure is typical or atypical. 

\begin{lemma}
\label{orbitoverparamlem}
There exists an algebraic family $f : \mathcal{Z} \to \mathcal{Y}$ of subvarieties of $Z_{\mathcal{M}} = E_{\mathcal{M}} \times \ch{D}^{0}$ such that all subvarieties $Z_{\mathcal{N}} \subset Z_{\mathcal{M}}$ associated to suborbit closures $\mathcal{N} \subset \mathcal{M}$ arise as a fibre $f^{-1}(y)$ for some $y \in \mathcal{Y}$.
\end{lemma}

\begin{proof}
It suffices to handle the two ``factors'' separately. In particular, since $E_{\mathcal{N}} \subset E_{\mathcal{M}}$ is an inclusion of linear subspaces, it is clear all possible choices for such a subspace are parameterized by a union of Grassmannian varieties. It thus suffices to show that all weakly special subdomains of $\ch{D}^{0}$ belong to a common algebraic family, which is a consequence of \Cref{gaolemma} below.
\end{proof}

Let $(\mathbf{H}_{\mathcal{M}},\ch{D}^0_{\mathcal{M}})$ be the monodromy datum associated to $(\mathcal{M},\V_{| \mathcal{M}})$ (where $\V $ is the standard $\Z$VMHS on some stratum containing $\mathcal{M}$). 

\begin{lemma}\label{gaolemma}
There is an algebraic family (over a disconnected base) $f: \mathcal{Z}\to \mathcal{Y}$ of subvarieties of $\ch{D}^0_{\mathcal{M}}$ such that, for every weakly special sub-datum $(\mathbf{H}_i,\ch{D}_i)$, $\ch{D}_i$ appears as a fibre of $f$.
\end{lemma}

\begin{proof}
This is a consequence of \cite[Lem. 12.3]{zbMATH06801925} and \cite[Sec. 8.2]{zbMATH07305885}. In particular, it is enough to take as base $\mathcal{Y}$ the disjoint union of finitely many copies of $\mathcal{G} \times \ch{D}^0_{\mathcal{M}}$, where the group $\mathcal{G}$ is defined right above \cite[Lem. 12.3]{zbMATH06801925}. The algebraic family obtained parametrizes orbits at some point $x\in \ch{D}^0_{\mathcal{M}}$ of some $\mathcal{G}(\C)$-translate of a finite set of representatives of weakly special subdomains of $\ch{D}^0_{\mathcal{M}}$.  (In the pure case a stronger result can be found for example in \cite[\S 4.2]{zbMATH06492665}.)
\end{proof} 

In what follows we abusively write $f : \mathcal{Y} \to \mathcal{Z}$ for the pullback of the family in \Cref{orbitoverparamlem} to $P$ along the map $r \times \nu$. We write $f^{(j)} : \mathcal{Y}^{(j)} \to \mathcal{Z}^{(j)}$ for the subfamily where the fibres have dimension $j$.

\begin{proof}[Sketch of the proof of \Cref{orbitclosurefinthm}]

 Let $(\mathcal{N}_{i})_{i\in \N}$ be an infinite sequence of intersection theoretically atypical (relatively to $\mathcal{M}$) suborbit closures that don't lie in any other suborbit closure. The first input is as the beginning of \cite[Proof of Thm. 1.5]{zbMATH06890813}: the Zariski closure of their union is a finite union of orbit closures, so it is enough to show that they are not Zariski dense in $\mathcal{M}$. 
We may fix the dimension $e$ of the suborbit closures $\mathcal{N}_{i}$ we consider; it suffices to prove the theorem for each $e$ separately. Similarly, we may fix the dimension $j$ of the inverse images $(r \times \nu)^{-1}(E_{\mathcal{N_{i}}} \times \ch{D}^{0}_{\mathcal{N}_{i}})$. 

Now let us apply \Cref{protogeoZP} to our situation using the family $f^{(j)}$. The hypotheses are satisfied by \Cref{ineqinsidebundle}, which gives the atypicality inequality with $e = j - \dim \mathbf{H}_{\mathcal{M}}$, and some standard Hodge theory which produces the desired weakly special families. We therefore obtain finitely many families $\{ h_{i} : C_{i} \to B_{i} \}_{i=1}^{m}$ of strict weakly special subvarieties of $S$ such that each $\mathcal{N}_{i}$ maps into a fibre of one of the $h_{i}$, and each map $C_{i} \to \mathcal{M}$ is quasi-finite. The proof is then completed by the following: 

\begin{lemma}
\label{orbitclnofamlem}
Let $h : C \to B$ be an algebraic family of subvarieties of $\mathcal{M}$ such that $\pi : C \to \mathcal{M}$ is quasi-finite. Then only finitely many fibres of $h$ contain a maximal suborbit closure of $\mathcal{M}$.
\end{lemma}

\end{proof}

\section{Functional transcendence for strata of meromorphic differentials}\label{ftmerdiff}

In this section we discuss the Ax-Lindemann and Ax-Schanuel for \textbf{meromorphic} differentials and some applications. 

Our starting point is the Ax-Lindemann for abelian differentials, conjectured by Klingler and Lerer \cite[Conj. 2.5]{2022arXiv220206031K}. Shortly after it was proved by Bakker-Tsimerman \cite[Thm. 5.4]{2022arXiv220805182B} as a consequence of new Ax-Schanuel theorems in the period torsor, as we recalled above as \Cref{thm:newAS}. We will find their theorem at the end of the section, see \Cref{albt}, but we take this occasion to generalize the picture, utilizing essentially the same circle of ideas. In a related direction, Pila and Tsimerman studied an Ax-Schanuel theorem for complex curves and differentials \cite[Thm. 3.2]{2022arXiv220204023P}.
\begin{rmk}
At the moment there aren't similar functional theoretic statements whose conclusion is a linear subvariety/an orbit closure. A first attempt can be found in \cite{2023arXiv231007523B}. 
\end{rmk}

Let's start by recalling the more general meromorphic setting. Let $\mu = (\mu_1,\dots, \mu_n)$ be an \textbf{integer} partition of $2g-2 $, and denote by $\Og(\mu)$ the stratum of (meromorphic) differentials of type $\mu$. I.e. the parameter space of pairs $(X,\omega)$ for $[X]\in \Mg$ and $\omega$  a meromorphic differential with zeros of order $\mu_i>0$, poles of order $\mu_i>0$, and marked points corresponding to $\mu_i=0$. Each stratum inherits an algebraic structure as a stratification of the Hodge bundle. We denote by $Z(\omega)$ and $P(\omega)$ the zeros and poles of $\omega$ respectively. Locally at a point $(X_0, \omega_0)$, we pick a choice of basis for the relative homology $H_1(X_0\setminus P(\omega_0),Z(\omega_0);\Z)$ (whose dual naturally comes equipped with a mixed Hodge structure). This basis extends locally via the flat connection and local orbifold coordinates known as \emph{period coordinates} are obtained by integrating the differential by this basis. Let $S$ be a connected component of $\Og(\mu)$ (cf. the work of Boissy and \cite[3.3.2]{filipnotes}). As in the usual case of abelian differentials, we have a holomorphic developing map
\begin{equation}\label{eqdev}
\operatorname{Dev}: \widetilde{S}\to H^1(X_0\setminus P(\omega_0),Z(\omega_0);\C),
\end{equation}
and set
\begin{equation}
\Pi := \{(v, \pi \operatorname{Dev}^{-1}(v) ), \forall v\in  H^1(X_0\setminus P(\omega_0),Z(\omega_0);\C)  \}\subset H^1(X_0\setminus P(\omega_0),Z(\omega_0);\C)\times S
\end{equation}
where $\pi : \widetilde{S}\to S$ is the uniformization map.
 See for example \cite[Sec. 3.3 and 3.4]{filipnotes},  \cite{2023arXiv230503309M}, and we refer also to \cite{zbMATH07141512} (and related works of the authors) for more general settings. What we need here is simply a variety/orbifold whose tangent space underlines a $\Z$VMHS.

\begin{rmk}
Bakker and Mullane \cite[App.]{2023arXiv230503309M} noticed that strata of meromorphic differentials can contain $\R$-linear, but non-algebraic manifolds, in contrast to the setting of holomorphic differentials (and Filip's theorem).
\end{rmk}

\subsection{Some results and applications}
We start with a simple observation regarding the monodromy of bi-algebraic subvarieties (the definition is the same as the one appearing in \Cref{bialgdef}).

\begin{prop}
Let $S\subset \Og (\mu)$ be a bi-algebraic subvariety. If $S $ is strictly contained in $ \Og (\mu)$ then the monodromy group of $\V_{| S} $ is strictly contained in the monodromy of $\V_{| \Og (\mu)}$. Furthermore bi-algebraic curves are maximal with this property, i.e. they are weakly special for the pair $(\Og (\mu), \V)$, in the sense of \Cref{thmweaksp}.
\end{prop}
\begin{proof}
Let $s=(X_0,\omega_0)\in S(\C)$ be a smooth point, and $\tilde{s}\in H^1:= H^1(X_0\setminus P(\omega_0),Z(\omega_0);\C)$ be a point above $s$. Since $S$ is bi-algebraic, there is an algebraic variety $Y_S\subset H^1$ above $S$, passing through $\tilde{s}$. We can consider the orbit $O_s:=\pi_1(S(\C),s) \cdot \tilde{s} \subset H^1$. It lies in $Y_S$. Since $Y$ is algebraic, the Zariski closure of $O_s$ is contained in $Y_S$. It is easy to observe that $\overline{O}_s^{Zar}$ is equal to $\mathbf{H}_S(\C)\cdot \tilde{s}$. If $S$ is strictly contained in the stratum, then $\mathbf{H}_S(\C)\cdot \tilde{s}$ has to be strictly contained in $H^1$, which therefore implies that $\mathbf{H}_S(\C)$ is not the full group $\mathbf{H}_{\Og (\mu)}$.

In general, the above argument shows that 
\begin{displaymath}
Y_S= \bigcup_{p \in Y_S}\overline{O_p}^{Zar}.
\end{displaymath}
but it is not clear whether one orbit suffices. If $S$ has dimension one, this is clearly the case, for dimension reasons. Therefore curves are weakly special, i.e. every bigger subvariety has bigger monodromy.
\end{proof}

We take this occasion to offer a statement and a proof of a stronger result, namely an Ax-Schanuel for meromorphic differential (along the lines of thought of \cite{2024arXiv240616628B}). Set $H^1:= H^1(X_0\setminus P(\omega_0),Z(\omega_0);\C)$.
\begin{thm}\label{asforabeliandiff}
Let $W\subset  H^1\times \Og (\mu)$ be an algebraic subvariety. Let $U$ be an analytic component of $W \cap \Pi$. If $U$ is atypical, i.e. 
\begin{displaymath}
\dim W - \dim U < \dim H^1,
\end{displaymath}
then the projection of $W$ to $\Og (\mu)$ lies in a strict bi-algebraic subvariety.
\end{thm}
Let $ \Og (\mu)$ be a stratum as above, and $\V$ its natural VHS, with target in a mixed Shimura variety $\mathcal{S}$, generalizing the setting described in \Cref{diagramatyp}. As observed in \Cref{standardbundle}, we have the so called \emph{standard automorphic} bundle $P\to \mathcal{S}$, which is a $G$-torsor for $G$ the monodromy of $\V$. Given $S\subset \Og (\mu)$ an irreducible subvariety, we have (cf. \cite{2022arXiv220805182B}) a factorisation of \eqref{eqdev}:
\begin{displaymath}
\widetilde{S}\xrightarrow{\sigma} P=P(\V_{| S}) \xrightarrow{r} H^1.
\end{displaymath}
The latter arrow $r$ is algebraic and $\mathbf{H}(\C)$-equivariant, where $\mathbf{H}_S$ it the monodromy group of $\V_{| S}$ (or the smooth locus of $S$).

\begin{proof}
By pulling back to $P$, we can apply \Cref{thm:newAS}, cf. proof of \cite[Thm. 5.4]{2022arXiv220805182B}.
\end{proof}

\begin{cor}[Bakker-Tsimerman's Ax-Lindemann for abelian differentials, {\cite[Thm. 5.4]{2022arXiv220805182B}}]\label{albt}
For any algebraic subvariety $Y$ of $ H^1(X_0,Z(\omega_0);\C)$, the Zariski closure of $\pi(Y_0)$ is bi-algebraic for any component $Y_0$ of $\operatorname{dev}^{-1}(Y)$.
\end{cor}
\begin{proof}
We simply explain how it formally follows from \Cref{asforabeliandiff}: just set $W=Y\times \pi (Y_0)^{\operatorname{Zar}}$.
\end{proof}

\section{Related Results and Conjectures}
\subsection{Sketch of the proof of André-Oort} \label{sectaoproof}

We thought it is worth sketching, in a very crude form, the proof of AO since it is a result that was used many times in the text and whose mere existence highly influenced the angle of these notes. A further reference for an introduction to the bi-algebraic point of view is \cite{MR3821177}. The proof consists of three main steps that form the so called \emph{Pila-Zannier strategy} (originated in \cite{zbMATH05292756}):
\begin{itemize}
\item definability in some o-minimal structure of the restriction of the uniformization map to a semi-algebraic fundamental set $\mathcal{F}$ for the action of $\Gamma$ on $X$
\item functional transcendence (in the easier case called \emph{Ax-Lindemann})
\item Pila-Wilkie’s counting theorem and a good lower bound for the size of Galois orbits of special points (which we detail below).
\end{itemize}

Thanks to the averaged Colmez conjecture, Tsimerman proved \cite{MR3744855} (see also \cite[Prob. 14]{zbMATH01587304}):
\begin{thm}[Large Galois Orbits]
Let $g \geq 1$. There exists $\delta_g >0$ such that if $E$ is a CM field of degree $2g$, $\Psi$ a primitive CM type for $E$, $A$ an abelian variety with CM by $(\Oo_E, \Psi)$, then the field of moduli of $A$ satisfies
\begin{displaymath}
[\Q (A):\Q] \gg_g |\operatorname{disc} (E)|^{\delta_g}.
\end{displaymath}
\end{thm}

As an application of such result, Orr and Skorobogatov proved:

\begin{thm}[{\cite[Thm. A]{zbMATH06944055}}]
There are only finitely many $\C$-isomorphism classes of abelian varieties of CM type of given dimension which can be defined over number fields of given degree.
\end{thm}

\subsection{Isogenies over the algebraic numbers}
 Another notable example is the work of Tsimerman \cite{zbMATH06074023}, who proved the existence of a $\overline{\Q}$ abelian variety isogenous to no Jacobian (using also the main result of \cite{zbMATH07794951} that we discussed before).

We recall the following folklore question.
\begin{question}
Let $g\geq 4$, and let $ A$ be a generic abelian variety over an algebraically closed field $k$. What is the lowest integer $g'$ such that there exists a Jacobian $J$ of dimension $g'$ and a surjection $J\to A$?
\end{question}

In characteristic zero, the first results were obtained by Chai-Oort \cite{zbMATH06074022}, Tsimerman \cite{zbMATH06074023}, and after in a more general setting by Masser-Zannier \cite{zbMATH07168647}. For simplicity we just cite the most recent and general one \cite{2023arXiv230205860T}:
\begin{thm}[Tsimerman]
For any two integers $g\geq 4$ and $g' \leq 2g -1$, there exist $g$-dimensional abelian varieties over $\overline{\Q}$ which are not quotients of a Jacobian of dimension $g'$.
\end{thm}
Regarding the existence of abelian varieties not isogenous to Jacobians over characteristics $p$ function fields see \cite{2021arXiv210502998S}.

\subsection{Ceresa cycles}
Another exciting result related to the Zilber–Pink conjecture is the study of the torsion locus of the Ceresa normal function by Gao and Zhang \cite{2024arXiv240701304G}, as well as related work by Kerr-Tayou \cite{2024arXiv240619366K} and Hain \cite{2024arXiv240807809H}. (Hain's approach is not related to ZP, but argues by induction on the genus-- the base case $g=3$ uses a result of Collino and Pirola).

Let $g\geq 2$, and $[X]\in \Mg$. For each $x\in X(\C)$, the Abel--Jacobi mapping $\mu_x : X \to \Jac (X)$ sends $p$ to $[p]-[x]$ and embeds $X$ in its Jacobian. Denote by $X_x$ the 1-cycle in $\Jac X$ obtain by image of $X$ along $\mu_x$, and by $X_x^-$ the image of $X_x$ under the multiplication by $-1$. The Ceresa cycle
associated to $(X, x) $ is the algebraic 1-cycle
\begin{displaymath}
\operatorname{Ce} (X,x):=X_x-X_x^-.
\end{displaymath}
When $X$ is a general curve of genus $\geq 3$, the Ceresa cycle is not algebraically equivalent to 0, as proven by Ceresa \cite{zbMATH03855282} in the 80s.

\begin{thm}
For all $g\geq 3$, the normal function $\nu $ of the Ceresa cycle has the maximum possible rank, namely $3g-3$.
\end{thm}

Gao and Zhang have proved a version of the above theorem and used it to prove that, for all $g\geq 3$, there is a Zariski open subset of $\Mg /\Q$ on which the Bloch-Beilinson height of the Ceresa and Gross-Schoen cycles have the Northcott property, see indeed \cite[Thm. 1.1 and Cor. 1.2]{2024arXiv240701304G} .

\subsection{Compact families of curves}
We describe some results on the maximal compact subvarieties of Siegel modular varieties \cite{2024arXiv240406009G}. We study the maximal dimension of a compact subvariety of $V$, which we denote $\operatorname{mdim}_c(V)$. We define also the maximal dimension of a compact subvariety of $V$ passing through a very general point of $V$: $\operatorname{mdim}_{c,gen}(V)$. The next theorem also uses \Cref{dense} which is once more related to typical intersections, as described in \Cref{tyin}. Keel and Sadun \cite{zbMATH01963977} proved that $\operatorname{mdim}_c(\Ag) \leq  g(g-1)/2 -1$ for any $g\geq 3$, see \cite[Thm. A]{2024arXiv240406009G} for stronger bounds.

Denote by $\Mg^{ct}$ the moduli space of curves of compact type. In \cite[Cor. C]{2024arXiv240406009G}, the authors prove:
\begin{thm}[Grushevsky, Mondello, Salvati Manni, Tsimerman]
For $p \in  \Ag$ a very general point, the maximal dimension of a compact subvariety of $\Ag $ containing p is equal to $g -1$.

For $\Mg^{ct}$ the following equality holds:
\[
\mathrm{mdim}_{\C}(\mathcal{M}_g^{\mathrm{ct}}) = \left\lfloor \frac{3g}{2} \right\rfloor - 2 \quad \text{for any } 2 \leq g \leq 23,
\]
and for the image of $\mathcal{M}_g^{\mathrm{ct}}$ along the (proper) extended Torelli map $J: \mathcal{M}_g^{\mathrm{ct}} \to \Ag$ we have the following upper bound:
\[
\mathrm{mdim}_{\C}(J(\mathcal{M}_g^{\mathrm{ct}})) \leq g - 1 \quad \text{for any } 2 \leq g \leq 15.
\]
\end{thm}

In contrast (there's also a more recent simple algebraic proof by D. Chen), the situation on strata of holomorphic differentials presents a difference:
\begin{thm}[Gendron \cite{zbMATH07226776}]
The stratum $\Og(\kappa)$ of holomorphic differentials of type $\kappa$ does not contain complete curves.
\end{thm}

\subsection{A digression on (non-arithmetic) ball quotients}\label{digballquotients}
We briefly discussed ball quotients in \Cref{knowncas} and we wish now to further elaborate on this topic. We will actually consider quotients by arbitrary lattices $\Gamma \subset \operatorname{PU}(1,n)=\operatorname{Aut}(\mathbb{B}^n_\C)$. The study of such lattices has indeed fascinating connections with the material exposed so far.

\begin{rmk}
It's interesting to observe that weight one $\C$-Hodge structures of signature $(1,n) $ are parameterized by the complex ball $\mathbb{B}^n_{\C}$, and weight one $\R$-Hodge structures (on a vector space of dimension $2g$) are parameterized by the Siegel space $\mathbb{H}_g$. 
\end{rmk}

The Margulis super-rigidity theorem \cite{margulisbook} implies the
arithmeticity of irreducible lattices in all noncompact semisimple Lie groups except $\operatorname{SO}(1,n)$
and $\PU(1, n)$, i.e., for all but real and complex hyperbolic lattices. Specifically, Margulis proved
superrigidity of irreducible lattices in semisimple Lie groups of real rank at least two. Using
the theory of harmonic maps, this was later extended to the rank one groups $\operatorname{Sp}(1, n)$ and $F^{(-20)}_ 4$
by Corlette \cite{MR1147961} and Gromov–Schoen \cite{MR1215595}. Real hyperbolic lattices are known to be softer and more flexible than their complex counterpart and non-arithmetic lattices in $\operatorname{SO}(1,n)$ can be constructed for every $n$. Non-arithmetic complex hyperbolic lattices have been found in $\PU(1,n)$, only for $n=1,2,3$. 

We briefly survey some of the results on complex hyperbolic lattices that are closest to the narrative of our text.
\begin{itemize}
\item Deligne-Mostow \cite{zbMATH03996010}. Every configuration of distinct points $\underline{t}=(t_1,\dots , t_n)\in \C^n$ determines a curve $X_{d,\underline{t}}$ by setting $y^d=\prod (x-t_i)$. The braid group $B_n$ acts on the cohomology of $X_{d,\underline{t}}$. Restricting attention to a single eigenspace, one obtains an irreducible
unitary representation 
\begin{displaymath}
\rho_q: B_n \to U(H^1(X_{d,\underline{t}})_q)\cong U(r,s)
\end{displaymath}
They identified numerical criteria to guarantee when the image is a discrete subgroup of $U(r,s)$, when it is arithmetic or non-arithmetic (in $U(1,n)$). 
\item Thurston \cite{zbMATH01260617} concerning moduli spaces of flat Euclidean structures with conical singularities on the two dimensional sphere. This gives a geometric interpretation of the Deligne-Mostow complex hyperbolic structures. 
\item In \cite{zbMATH00404254}, Veech extends to compact (oriented) surfaces of arbitrary genus several basic results of Thurston’s approach. 
\item Ghazouani and Pirio  \cite{zbMATH07362604, zbMATH06824843} studied moduli spaces of flat tori with cone singularities and prescribed holonomy by means of geometrical methods relying on surgeries on flat surfaces.
\item Another link between the \teic viewpoint on $\Og$ and ball quotients appear in the work of Kappes-M\"{o}ller \cite{zbMATH06543257}.
\item In \cite{arXiv:2303.17929}, the authors give formulas for the Chern classes of linear submanifolds of the moduli spaces of Abelian differentials and hence for their Euler characteristic. As an application, they obtain an algebraic proof of the theorems of Deligne-Mostow and Thurston that suitable compactifications of moduli spaces of $k$-differentials on the 5-punctured projective line are ball quotients.
\end{itemize}

\begin{rmk}
In \cite{zbMATH06149478}, McMullen offers an account of the unitary representations of the braid group that arise via the Hodge theory of cyclic branched coverings of the projective line.
\end{rmk}

The following was proved independently by Bader, Fisher, Miller and Stover and Baldi--Ullmo by completely different methods (dynamics and Hodge theory, respectively) and resembles \Cref{wrightthm} (see also earlier work \cite{BFMS}). 

\begin{thm}[{\cite[Thm.\ 1.1]{BFMS2}}, {\cite[Thm.\ 1.2.1]{BU}}]\label{thm:Finite}
Suppose that $M$ is a non-arithmetic finite-volume complex hyperbolic $n$-manifold, $n \ge 2$. Then $M$ contains only finitely many irreducible totally geodesic divisors.
\end{thm}

An irreducible totally geodesic divisor is simply an immersed complex codimension one totally geodesic complex submanifold. The approach of the latter team \cite{BU} is actually to build a $\Z$VHS on the ball quotient $M$ realizing the totally geodesic subvarieties as atypical intersections and then apply a result of ZP type. The VHS generalizes the modular embeddings of triangle groups that we have seen, behind the scenes, in \Cref{dejongnoot} and \Cref{mcmullenexa}.
\begin{rmk}
This circle of ideas linked to ZP actually leads to interesting results also in the arithmetic case. This new line of investigation was considered in \cite{2024arXiv240203601B}.
\end{rmk}

Interestingly both \Cref{thm:Finite} and  \Cref{wrightthm} can now be proven via atypical intersection techniques. It is natural to wonder whether there's a deeper link between the two world. For example:
\begin{question}
How do Kobayashi totally geodesic subvarieties (in the sense of \Cref{kobgeod}) distribute in general? 
\end{question}
For example, on a general smooth hypersurface $X$ in $\mathbb{P}^n$ of degree big enough (in fact $>O(n^{2n+6})$), the Kobayashi psuedo-metric is a metric, as proven by Brotbek \cite{zbMATH06827883}. Does $X$ contain only finitely many maximal totally geodesic subvarieties?

\subsubsection{Some remarks about Shimura varieties}\label{secfingeod}

We recall some facts about the Kobayashi (pseudo)-metric, along with a brief overview of McMullen's work \cite{zbMATH06260641}, and \cite{zbMATH06748171}. The first fact is that that we have already encountered is that the Kobayashi metric on Teichmüller spaces agrees with the \textbf{Teichmüller metric}.

The Kobayashi metric of a bounded symmetric domain $X$ does not coincide with its Hermitian symmetric metric, unless $X$ has rank one, i.e., $X \cong \mathbb{B}^n$. Let $X$ be a bounded \textbf{symmetric} domain and $x \in X$. As explained in \Cref{rmklink}, there is an essentially unique holomorphic embedding $i: X \to \C^N$ whose image is a strictly convex circular domain centered at $0=i(x)$ — the so-called Harish-Chandra realization of $X$ (centered at $x$). A useful description is as follows: There exists a finite-dimensional linear subspace $V_X \subset M_{n,m}(\C)$, the space of complex $n \times m$ matrices, such that 
\begin{displaymath}
X=\{v \in V_X : ||v||<1\}
\end{displaymath}
is the unit ball for the operator norm on $V_X$, where $||v||= \sup_{||\xi||_2=1} ||v(\xi)||_2$. There is a natural identification $T_x(X) \cong V_X \cong \C^N$. 
The Kobayashi norm, as discussed in \cite[(p. 462)]{zbMATH06748171}, on $V_X$ coincides with the operator norm, and the Kobayashi distance from the origin is given by $2d(0,v) = \log\left(\frac{1+||v||}{1-||v||}\right)$. 
\begin{question}
How do the \kob geodesic subvarieties of a Shimura variety behaved? Do they share similar properties of their smaller class of ``standard geodesics'' (described in \Cref{thmweaksp})?
\end{question}
\subsection{Other special families of curves}
Weakly special subvarieties were described in terms of their algebraic monodromy. In this short section we try to look for finer monodromic properties.
\begin{question}
Is there a curve $C\subset \mathcal{A}_g$ (projective or not) which is Hodge generic and $\pi_1$-injective? Can we find $C$ in $\mathcal{M}_g$? 
\end{question}
In certain Hilbert modular varieties, a positive answer follows from \Cref{mcmullenexa}.

We now briefly recall an example of \emph{thin} monodromy after Nori \cite{MR832040}, i.e. monodromy that have infinite index in the $\Z$-points of the Zariski closure of their monodromy. (In contrast, for example, to $\Mg$ whose $\pi_1$ dominates $\operatorname{Sp}_{2g}(\Z)$.)

\begin{thm}[Nori]
The family over $S_c:=\C-\{0,1,c^{-1}\}$ given by $A_\lambda= E_\lambda \times E_{c \lambda}$. The image of the associated monodromy representation is Zariski dense in $\Sl_2\times \Sl_2$ but it has infinite index in $\Sl_2(\Z)\times \Sl_2(\Z)$.
\end{thm}
Nori also shows that such image is not finitely presented, as a consequence of a group theoretic result of Bieri.

\begin{proof}[Sketch of the proof]
Consider the family of elliptic curves (for $\lambda\in \C -\{ 0,1\}$)
\begin{displaymath}
E_\lambda =\{y^2=x(x-1)(x-\lambda)\}.
\end{displaymath}
Let $c\in\C -\{ 0,1\} $ and $S_c:=\mathbb{P}^1_{\C}- \{0,1,c^{-1}, \infty\}$. There is a family of abelian surfaces $\mathcal{A}\to S_c$  whose fiber at $\lambda$ is 
\begin{displaymath}
\mathcal{A}_\lambda = E_\lambda \times E_{\lambda \cdot c}
\end{displaymath}
with associated monodromy representation
\begin{displaymath}
\rho : \pi_1(S_c)\to \Sl_2(\Z)\times \Sl_2(\Z).
\end{displaymath}
Nori shows (1) that the image of $\rho$ is Zariski dense in $\Sl_2\times \Sl_2$ (so the family is Hodge generic in $\mathcal{A}_1\times \mathcal{A}_1\subset \mathcal{A}_2$) and (2) that $\operatorname{Im}(\sigma)$ has infinite index in $\Sl_2(\Z)\times \Sl_2(\Z)$.

Both claims follow from the fact that inclusion of $S_c$ into $(\C -\{ 0,1\})^2$ given by $\lambda \mapsto (\lambda, \lambda \cdot c)$ induces an exact sequence of groups
\begin{equation}
\pi_1(S_c)\to \pi_1((\C -\{ 0,1\})^2)\to \Z \to 1.
\end{equation}
Define $f: (\C -\{ 0,1\})^2\to \C^*$, $(x,y)\mapsto x^{-1}y$. Since $S_c=f^{-1}(c)$, we find a complex of groups
\begin{displaymath}
\pi_1(S_c)\to \pi_1((\C -\{ 0,1\})^2)\to \pi_1(\C^*) \to 1
\end{displaymath}
such that all fibres of $f$ are smooth and non-empty, and the restriction of the complement of $f^{-1}(1)$ is a fibre bundle over $\C-\{0,1\}$ with connected fibres. This is enough to show exactness.
\end{proof}
\subsection{Further questions on the Hodge locus}

Two questions for which we expect a positive answer and that are perhaps within reach:
\begin{question}
Does every Shimura variety contain a $\overline{\Z}$-Hodge generic point? Does $\Mg$ contain a $\overline{\Z}$-Hodge generic point?
\end{question}

\begin{question}
For each $g\geq 3$, is there a curve $C\subset \Ag$, defined over $\overline{\Q}$ (resp. $\overline{\Z}$) with empty Hodge locus (i.e. $\MT(c)= \operatorname{GSp}_{2g}$, for all $c\in C$)? Can $C$ be found in $\Mg$?
\end{question}
In the first part of the question $C$ is implicitly required to be irreducible and Zariski closed, otherwise one can simply remove a finite number of points from $C$.

\newpage

\appendix

\section{Quick introduction to Shimura varieties and Hodge theory}\label{notationshimura}

We collect here some standard notation we often use.
\begin{itemize}
\item Given an algebraic group $G/\Q$, $G^{\text{der}} \subset G$ denotes its derived subgroup, $Z(G)$ its center, $G\to G^{\text{ab}}:=G/G^{\text{der}}$ its abelianization and $G\to G^{\text{ad}}:=G/Z(G)$ its adjoint quotient. Moreover by reductive group we always mean \emph{connected reductive};

\item We denote by $\A_f$ the (topological) group of finite $\Q$-adeles, i.e. $\A_f = \widehat{\Z}\otimes \Q$, endowed with the adelic topology. Given a subgroup $K \subset G(\A_f)\subset \prod_\ell G(\Q_\ell)$, we write $K_\ell \subset G(\Q_\ell)$ for the projection of $K$ along $G(\A_f)\to G(\Q_\ell)$. If a compact open subgroup of $G(\A_f)$ may be confused with a number field $K$, we denote the former by $\widetilde{K}$;
\item As in \cite[0.2]{deligneshimura} we write $(-)^0$ to denote an algebraic connected component and $(-)^+$ for a topological connected component, e.g. $G(\R)^+$ is the topological connected component of the identity of the group of the real points of $G$. We write $G(\R)_+$ for the subgroup of $G(\R)$ of elements that are mapped into the connected component $\Gad(\R)^+\subset \Gad(\R)$. Finally we set $G(\Q)^+:=G(\Q)\cap G(\R)^+$ and $G(\Q)_+ := G(\Q) \cap G(\R)_+$. 
\end{itemize}
Given a complex algebraic variety $S$, we denote by $S^{\an}$ the complex points $S(\C)$ with its natural structure of a complex analytic variety. We denote by $\pi_1(S)$ the topological fundamental group of $S^{\an}$ and by $\piet{(S)}$ the étale one. Unless it is necessary, we omit the base point in the notation. Given an algebraic variety $S$ defined over a field $K$, $S/K$ from now on, we write $S_{\overline{K}}$ for the base change of $S$ to $\operatorname{Spec} (\overline{K})$.
\subsection{Shimura varieties} 

Special instances of Shimura varieties were originally introduced by Shimura in the 60's. Deligne outlined the theory of Shimura varieties in \cite{deligneshimura, delignetravaux}. For introductory notes on this vast subject we  mention \cite{milnenotes}. In this short
appendix, we try to summarise some of the aspects the theory of Shimura varieties that are relevant to the main text, roughly following \cite[Sec. 2.2]{Baldi2020}.

Let $\DT$ denote the real torus $\Res_{\C / \R} (\Gm)$, which is usually called the \emph{Deligne torus}. 
\begin{defi}\label{shimuradatum}
A \emph{Shimura datum} is pair $(G,X)$ where $G$ is a reductive $\Q$-algebraic group and $X$ a $G(\R)$-orbit in the set of morphisms of $\R$-algebraic groups $\Hom(\DT, G_\R)$, such that for some (equivalently all) $h \in X$ the following axioms are satisfied:
\begin{itemize}
\item[SD1.] $\text{Lie}(G)_\R$ is of type $\{(-1,1), (0,0), (1,-1)\}$;
\item[SD2.] The action of the inner automorphism associated to $h(i)$ is a Cartan involution of $G^{\text{ad}}_\R$. That is, the set
\begin{displaymath}
\{g \in G^{\text{ad}}(\C) \ : \ h(i)g h(i)^{-1}=g\}
\end{displaymath}is compact;
\item[SD3.] For every simple $\Q$-factor $H$ of $G^{\text{ad}}$, the composition of $h: \DT \to G_\R$ with $G_\R \to H_\R$ is non trivial.
\end{itemize}
\end{defi}

Let $(G,X)$ be a Shimura datum and $\widetilde{K}$ a compact open subgroup of $G(\A_f)$. Notice that $G(\Q)$ acts on the left on $X \times G(\A_f)$ by left multiplication on both factors and $\widetilde{K}$ acts on the right just by right multiplication on the second factor. We set
\begin{displaymath}
\Sh_{\widetilde{K}}(G,X) := G(\Q)  \backslash \left ( X \times G(\A_f) / \widetilde{K} \right).
\end{displaymath}
Let $X^+$ be a connected component of $X$ and $G(\Q)^+$ be the stabiliser of $X^+$ in $G(\Q)$. The above double coset set is a disjoint union of quotients of $X^+$ by the arithmetic groups $\Gamma_g:=G(\Q)^+ \cap g\widetilde{K}g^{-1}$ where $g$ runs through a set of representatives for the finite double coset set $G(\Q)^+ \backslash G(\A_f )/\widetilde{K}$. Baily and Borel \cite{MR0216035} proved that $\Sh_{\widetilde{K}}(G,X)$ has a unique structure of a quasi-projective complex algebraic variety. Moreover if $\widetilde{K}$ is neat, then $\Sh_{\widetilde{K}}(G,X)$ is smooth.

\begin{rmk}
Arithmetic subgroups of $G$ are in particular lattices in $G(\R)$. That is discrete subgroups of finite covolume. 
\end{rmk}
For every inclusion $K_1 \subset K_2$ we have a map $\Sh_{K_1}(G,X)\to \Sh_{K_2}(G,X)$, which is an algebraic map again by a result of Borel. If $K_1$ is normal in $K_2$ then it is the quotient for the action by the finite group $K_2/K_1$ and therefore these morphisms are finite (as morphisms of schemes). That means that we can take the limit (projective limit) of the system of these, in the category of schemes, which we denote by $\Sh(G,X)$. We have
\begin{displaymath}
\Sh(G,X)(\C)= G(\Q)  \backslash X \times G(\A_f).
\end{displaymath}
Denote by $\pi$ the projection  
\begin{displaymath}
\pi: \Sh(G,X)\to \Sh_{\widetilde{K}}(G,X)= G(\Q ) \backslash (X \times G(\A_f)/\widetilde{K}).
\end{displaymath}

Given a Shimura datum $(G,X)$ there exists an \emph{adjoint Shimura datum} $(\Gad,\Xad)$ where $\Xad$ is the $\Gad(\R)$-conjugacy class of the morphism $h^{\text{ad}}$, defined as the composition of any $h:\DT \to G$ and $G\to \Gad$. The construction gives a natural morphism of Shimura data $(G,X)\to (\Gad, \Xad)$ and, choosing a compact open $\widetilde{K}^{\text{ad}}$ in $\Gad (\A_f)$ containing the image of $\widetilde{K}$, we obtain also a finite morphism of Shimura varieties
\begin{displaymath}
\Sh_{\widetilde{K}}(G,X)\to \Sh_{\widetilde{K}^{\text{ad}}}(\Gad, \Xad).
\end{displaymath}

\subsection{Shimura varieties as moduli spaces of Hodge structures}\label{sectionhodgeshimu}
 We discuss how conditions SD1-3 of Definition \ref{shimuradatum} imply that connected components of $X$ are Hermitian symmetric domains and faithful representations of $G$ induce variations of polarisable $\Q$-Hodge structures. For an introduction to Hodge theory we refer to the book \cite{MR2918237}.

\subsubsection{Hodge structures}\label{hodgesection}
Let $V_\R$ be a finite dimensional $\R$-vector space and $V_\C= \C \otimes_\R V_\R$ it complexification. The complex conjugation acts on $V_\C$ by $\lambda \otimes v \mapsto \overline{\lambda} \otimes v$. A \emph{Hodge decomposition} of $V_\R$ is a direct sum decomposition of $V_\C$ into $\C$-subspaces $V ^{p,q}$ indexed by $\Z^2$:
\begin{displaymath}
V_\C = \oplus_{(p,q)\in \Z\times \Z} V^{p,q},
\end{displaymath}
such that $\overline{ V^{p,q}}=V^{q,p}$. An $\R$-\emph{Hodge structure} is a finite dimensional $\R$-vector space $V_\R$ with a Hodge decomposition. The type of this Hodge structure is the set of $(p,q)$ for which $V^{p,q}\neq 0$. The datum of an $\R$ Hodge structure is equivalent to the one of a real representation of the Deligne torus
\begin{displaymath}
h: \DT \to \Gl(V_\R).
\end{displaymath}
With this interpretation we have natural notions of dual, homomorphism, tensor product, direct sum and irreducibility of Hodge structures. Fixing an $n\in \Z$, the subspace
\begin{displaymath}
V_n :=\oplus_{p+q=n} V^{p,q}
\end{displaymath}
is stable under complex conjugation and it is referred to as a \emph{Hodge structure of weight n}.

A $\Q$-\emph{Hodge structure} (HS from now on) is a finite dimensional $\Q$-vector space $V$ together with an $\R$-HS on $V\otimes \R$ and a $\Z$-HS is a free $\Z$-module of finite rank $V_\Z$ together with a $\Q$-HS on $V_\Z \otimes \Q$.

\begin{defi}
Let $(V,h)$ be a $\Q$-HS of weight $n$. A polarisation of $(V,h)$ is a bilinear map of $\Q$-Hodge structures $q: V \otimes V \to \Q(-n)$, that is $q_\R(h(z)v,h(z)w)= (z\overline{z})^nq_\R(v,w)$, such that
\begin{displaymath}
q(v,h(i)w)
\end{displaymath}
defines a positive definite symmetric form on $V_\R$. Finally $(V,h)$ is called \emph{polarisable} if there exists a polarisation $q$ on it.
\end{defi}
The category of $\Q$-HS is abelian and the category of polarisable $\Q$-HS is semisimple.
\begin{rmk}
Since $h(i)^2=(-1)^n$, $q$ is symmetric if the weight is even, alternating if it is odd.
\end{rmk}

All the HSs we will consider will be polarisable (and often pure). We therefore simply say \emph{Hodge structure} to mean polarisable Hodge structure. If it is clear from the context, by HS we could also mean $\Q$-HS. Moreover a weight and a type will in general be fixed. We conclude this short section with an important definition.
\begin{defi}\label{mumofrdtategrouo}
Let $(V_\Q,h)$ be a $\Q$-HS. The \emph{Mumford--Tate group} of $(V,h)$, simply denoted by $\MT(h)$, is the $\Q$-Zariski closure of the image of
\begin{displaymath}
h: \DT \to \Gl(V_\R)
\end{displaymath}
in $\Gl(V_\Q)$. That is the smallest $\Q$-subgroup $H$ of $\Gl(V_\Q)$ such that $H_\R$ contains $h(\DT)$.
\end{defi}
\subsubsection{Variations of Hodge structures}\label{sectionvhs}
Let $S$ be a smooth quasi-projective variety and let $\mathcal{V}=(\mathcal{V}_\Z,\F,\mathcal{Q}_\Z)$ a \emph{polarized variation of} $\Z$-\emph{Hodge structure} on $S$ (VHS from now on). That is the data of
\begin{itemize}
\item A local system $\mathcal{V}_\Z$ with a flat quadratic form $\mathcal{Q}_\Z$;\
\item A holomorphic locally split filtration $\F$ of $\mathcal{V}_\Z \otimes_\Z \Oo_{S^{\an}} $ such that the flat connection $\triangledown$ satisfies Griffiths transversality:
\begin{displaymath}
\triangledown(\F^p)\subset \F^{p-1} \text{  for all  }p;
\end{displaymath}
\item $(\mathcal{V}_\Z,\F,\mathcal{Q}_\Z)$ is fiberwise a $\Z$-HS.
\end{itemize}
In the same way we have definitions of $\Z, \Q$, $K$ and $\R$-VHS, where $K$ is a sub-field of $\R$.

Let $\lambda: \widetilde{S}\to S$ be the universal cover of $S$ and fix a trivialisation $\lambda^* \mathcal{V}\cong \widetilde{S}\times V$. 

It is well know that there exists a countable union $\Sigma \subsetneq S$ of proper analytic subspaces of $S$ such that:
\begin{itemize}
\item For $s\in S - \Sigma $, $\MT_s \subset \Gl(V)$ does not depend on $s$, nor on the choice of $\tilde{s}$. We call this group \emph{the generic Mumford--Tate group of} $\mathcal{V}$ and we simply write it as $G$;
\item For all $s$ and $\tilde{s}$ as above, with $s\in \Sigma$, $\MT_s$ is a proper subgroup of the generic Mumford--Tate group of $\mathcal{V}$.
\end{itemize}
To be more precise $\Sigma$ is also known to be a countable union of algebraic subvarieties of $S$. This follows indeed from the work of Cattani, Deligne and Kaplan \cite{CDK}.

\subsubsection{Period domains}\label{recapperioddomains}
Let $(V_\Z,q_\Z)$ be a polarized $\Z$-Hodge structure. Let $G$ be the $\Q$-algebraic group $\Aut(V_\Q,q_\Q)$. Consider the space $D$ of $q_\Z$-polarized Hodge structures on $V_\Z$ with specified Hodge numbers (it is homogeneous for $G$). Fixing a reference Hodge structure, we write $D=G(\R)/M$ where $M$ is a subgroup of the compact unitary subgroup $G(\R)\cap U(h)$ with respect to the Hodge form $h$ of the reference Hodge structure.

Let $S$ be a smooth quasi-projective complex variety. By \emph{period map}
\begin{displaymath}
S^{\an}\to \Gamma \backslash D
\end{displaymath} 
we mean a holomorphic locally liftable Griffiths transverse map, where $\Gamma$ is a finite index subgroup of $G(\Z)=\Aut (V_\Z,q_\Z)$. A period map $S^{\an}\to G(\Z) \backslash D$ is equivalent to the datum of a $\Z$-VHS on $S$ with generic Mumford--Tate group $G/\Q$. The period map lifts to $\Gamma \backslash D$ if $\Gamma$ contains the image of the monodromy representation of the corresponding $\Z$-VHS. 

\begin{rmk}
Shimura varieties are particular cases of period domains. See indeed \cite[Prop. 1.1.14 and Cor. 1.1.17]{delignetravaux}. Period domains, in the generality defined above, do not have a natural algebraic structure \cite{MR3234111}. Shimura varieties could indeed be thought as the \emph{most special} example of period domains. 
\end{rmk}

\subsection{Canonical models and reflex field}\label{canonicalmodel}
To be of arithmetic interest, Shimura varieties must admit models over a number field. Thanks to the work of Borovoi, Deligne, Milne and Milne-Shih, among others, the $\C$-scheme
\begin{displaymath}
\Sh (G,X)= G(\Q)  \backslash \left ( X \times G(\A_f) \right),
\end{displaymath}
together with its $G(\A_f)$-action, can be canonically defined over a number field $E:= E(G,X) \subset \C$ called the \emph{reflex field} of $(G,X)$. That is there exists an $E$-scheme $\Sh (G,X)_E$ with an action of $G(\A_f)$ whose base change to $\C$ gives $\Sh (G,X)$ with its $G(\A_f)$-action. For the precise definition of \emph{canonical model} we refer to \cite[section 2.2]{deligneshimura}. For the easier fact that Shimura varieties can be defined over $\Qbar$, we refer for example to \cite{MR791585}.

It follows that for every compact open subgroup $\widetilde{K}$ of $G(\A_f)$, the variety $\Sh_{\widetilde{K}}(G,X)$ admits a canonical model over $E$ in such a way that Hecke correspondences commute with the Galois action. Moreover the map $\pi: \Sh (G,X)\to\Sh_{\widetilde{K}}(G,X)$ is defined over $E$. In general we write $K$ for a finite extension of $E$ such that $\Sh_{\widetilde{K}}(G,X)(K)$ is not empty.

\subsection{Some examples of Shimura varieties}\label{examples}
We briefly discuss some of the main motivating examples of Shimura varieties. For brevity we actually describe here five examples of \emph{connected} Shimura varieties. 
\begin{exs}[Modular curves]\label{modularcurvesex}
Denote by $ \mathbb{H}$ the complex upper half plane. It becomes a Hermitian symmetric space when endowed with the metric $y^{-2}dxdy$. The $\Sl_2$-action
\begin{displaymath}
\left( \begin{matrix} 
a & b \\
c & d 
\end{matrix}\right) \cdot z =\frac{az+b}{cz+d}, \text{ where  } \left( \begin{matrix} 
a & b \\
c & d 
\end{matrix}\right) \in \Sl_2(\R) \text{  and  } z \in \mathbb{H},
\end{displaymath}
identifies $\Sl_2(\R)/\{\pm I\}$ with the group of holomorphic automorphisms of $\Hh$, where $I\in \Sl_2(\R)$ denotes the identity. For any $x+iy \in \Hh$
\begin{displaymath}
x+iy= \left( \begin{matrix} 
\sqrt{y} & x/\sqrt{y} \\
0 & 1/\sqrt{y} 
\end{matrix}\right) \cdot i,
\end{displaymath}
and so $\Hh$ is a homogeneous space. Let $\Gamma$ be a congruence subgroup of $\Sl_2(\Z)$, that is a subgroup containing 
\begin{displaymath}
\Gamma_0(N):=\{M \in \Sl_2(\Z) : M \equiv I \operatorname{mod} N \},
\end{displaymath}
for some $N$. The curve $\Gamma \backslash \Hh$ is a connected modular curve and can be realized as a moduli variety for elliptic curves with some level structure.
\end{exs}

Let $L$ be a totally real number field.  What if we want to take the quotient of $\Hh$ by (subgroups of) $\Sl_2(\Oo_L)$? The problem is that $\Sl_2(\Oo_L)$ is in general not discrete in $\Sl_2(\R)$. This can be circumnavigated by a construction called \emph{Weil restriction}. 
\begin{exs}[Hilbert modular varieties]\label{hilbertex}
Let $n_L$ be the degree of $L$ over $\Q$. The subgroup $\Sl_2(\Oo_L)$ can naturally be identified as $G(\Z)$, where $G$ is the Weil restriction from $L$ to $\Q$ of $\Sl_2/L$. Indeed, using the $n_L$-real embeddings of $L$, we see that $\Sl_2(\Oo_L)$ acts on the product of $n_L$ copies of $\Hh$. Hilbert modular varieties are obtained as quotients $\Gamma \backslash \Hh^{n_L}$, where $\Gamma$ is a congruence subgroup of $G(\Z)$ and naturally parametrise principally polarized $n_L$-dimensional abelian varieties with $\Oo_L$-multiplication (with some level structure).
\end{exs}

\begin{exs}[Siegel modular varieties]\label{agex}
Let $g$ be an integer $\geq 1$. Consider the Siegel upper half space
\begin{displaymath}
\Hh_g:= \{M \in M_g(\C): M \text{  is symmetric and  } \operatorname{Im}(M) \text{  is positive definite } \}.
\end{displaymath}
By Riemann every matrix in $\Hh_g$ is the period matrix of some principally polarized abelian variety, unique up to isomorphism of polarized abelian varieties. The Siegel modular variety, also denoted as $\mathcal{A}_g$, is the quotient $\operatorname{Sp}_{2g}(\Z) \backslash \Hh_g$. As in the case of modular curves one can see that there is a bijection between the complex points of the Siegel modular variety and isomorphism classes of principally polarized $g$-dimensional abelian varieties. 
\end{exs}

Finally one of the oldest and simplest examples of Shimura varieties. They were first studied by Picard in 1881.
\begin{exs}[Picard modular surfaces]\label{picardmodular}
Let $E$ be an imaginary quadratic extension of $\Q$ and $V$ be a $3$-dimensional $E$-vector space (which we consider as a vector space over $\Q$). Fix an integral structure on $V$, given by an $\Oo_E$-lattice $L \subset V$ and let
\begin{displaymath}
J : V \times V\to E
\end{displaymath}
be a non-degenerate Hermitian form on $V$, satisfying $J(\alpha u ,\beta v)=\overline{\alpha}\beta \overline{J(u,v)}$ and which is $\Oo_E$-valued on $L$ and has signature $(1,2)$ over $V\otimes \R$. Let $G':= \operatorname{SU}(J,V)/\Q$ be the special unitary group of $J$, viewed as a semisimple algebraic group over $\Q$. The group $G'(\R)/K'$, for any compact maximal $K' \subset G'(\R)$ can be identified with the complex two-dimensional ball. The Picard modular group of $E$ is
\begin{displaymath}
G'(\Z):= \{\gamma \in G'(\Q):  \gamma L =L\},
\end{displaymath}
and given a finite index subgroup $\Gamma$ of $G'(\Z)$ we obtain a \emph{Picard modular surface} $\Gamma \backslash X$. Picard modular surfaces parametrise $3$-dimensional homogeneously polarized abelian varieties with $\Oo_E$-multiplication, signature $(1,2)$ and some level structure. 
\end{exs}

Part of the beauty of the theory of Shimura varieties is that many problems stated in the moduli language translate in the ``$(G, X)$-language''. Such translation often allows us to use in a more transparent way powerful group-theoretical tools and, in many cases, it is the first step towards understanding more general period domains. 

A remark regarding the difference between Shimura varieties and connected Shimura varieties is in order.
\begin{rmk}
Already in the example of modular curves we described the connected case when we decided to consider the upper half plane $\Hh$, rather than
\begin{displaymath}
\Hh^{\pm}:=\C - \R,
\end{displaymath} 
with its natural $\Gl_2$-action. The main advantage of Shimura varieties, over their connected counterparts, is that their canonical models do not depend on a realisation of $G$ as the derived group of a reductive group nor on the level structure. For example the curve $\Gamma_0(N)\backslash \Hh$ admits a (geometrically connected) model over $\Q(\mu_N)$, rather than over $\Q$.
\end{rmk}

\subsection{Sub-Shimura varieties of $\Ag$, some examples}
Special curves in $\mathcal{A}_2$ are of three types, as we now recall. Curves parametrising abelian surfaces
\begin{enumerate}
\item with quaternionic multiplication;
\item isogenous to the square of an elliptic;
\item isogenous to the product of two elliptic
curves, at least one of which has complex multiplication.
\end{enumerate}

More generally, Moonen \cite{2024arXiv240520673M} recently gave a  classification of all 1-dimensional Shimura subvarieties of $\Ag$,  in terms of Shimura data.

\subsection{Mixed Shimura varieties}\label{mixedshim}
Since, in \Cref{diagramatyp}, we discussed mixed Shimura varieties (namely $\mathcal{A}_{g,n}$ and its special subvarieties), we end by recalling their definition and some properties.

\begin{defi}
 A \emph{(connected) mixed Shimura datum} is a pair $(G,D^+)$ where
\begin{itemize}
\item $G$ is a connected linear algebraic group defined over $\Q$, with unipotent radical $W$, and an algebraic subgroup $U \subset W$ which is normal in $G$;
\item $D^+\subset \Hom (\DT_{\C}, G_\C)$ is a connected component of an orbit under the subgroup $G(\R)\cdot U(\C)\subset G(\C)$;
\end{itemize}
satisfying axioms (i)-(vi) in \cite[Def. 2.1]{zbMATH05013101}. A \emph{(connected) mixed Shimura variety} associated to $(g,D^+)$ is a complex manifold of the form $\Lambda \backslash D^+$ where $\Lambda$ is a congruence subgroup of $G(\Q)_+$ acting freely on $D^+$.
\end{defi}

A mixed Shimura datum allows to take into consideration groups of the form $\operatorname{GSp}_{2g} \ltimes \Ga^{2g}$. For suitable congruence subgroups the associated connected mixed Shimura variety is the universal family of abelian varieties over the moduli space of principally polarised abelian varieties (with some $n$-level structure). The point is that every (principally polarised) abelian variety can be realised as a fibre of such a family. This is indeed what we used in \Cref{diagramatyp} to take care of the torsion conditions appearing in the Hodge theoretic description of orbit closures.

As for the pure case, there are notions of special and weakly special subvarieties of mixed Shimura varieties (see \cite[Sec. 4]{zbMATH05013101}). As the reader may expect every irreducible component of the intersection of special subvarieties (resp. weakly special) is again special (resp. weakly special), and a weakly special subvariety containing a special point is itself special. For example special points in the universal family of abelian varieties correspond to torsion points in the fibers $A_s$ over all special points $s\in \Ag$.

\newpage

\bibliographystyle{abbrv}
\bibliography{biblio.bib}

\Addresses

\end{document}